\documentclass[aos,preprint]{imsart}

\arxiv{arXiv:1711.08328}

\startlocaldefs
\usepackage{xcolor}

\usepackage{amsthm,amsmath,amssymb,amsbsy,bbm,mathrsfs,supertabular,eurosym,graphicx,enumitem}

\usepackage[authoryear]{natbib}
\newtheorem{thm}{Theorem}
\newtheorem{lem}{Lemma}
\newtheorem{prop}{Proposition}
\newtheorem{df}{Definition}
\newtheorem{cor}{Corollary}
\newtheorem{ass}{Assumption}

\theoremstyle{BBstyle4}

\newcommand{\pa}[1]{\left({#1}\right)}
\newcommand{\norm}[1]{\left\|{#1}\right\|}
\newcommand{\cro}[1]{\left[{#1}\right]}
\newcommand{\ab}[1]{\left|{#1}\right|}
\newcommand{\ac}[1]{\left\{{#1}\right\}}

\newcommand{\pen}{\mathop{\rm pen}\nolimits}

\newcommand{\dfleche}[1]{\,\displaystyle{\mathop{\longrightarrow}_{#1}}\,}
\newcommand{\CP}[1]{\stackrel{\mathrm{P}}{\dfleche{#1}}}

\newcommand{\CV}[1]{\dfleche{#1}}

\newcommand{\E}{{\mathbb{E}}}
 
\renewcommand{\L}{{\mathbb{L}}}
\newcommand{\N}{{\mathbb{N}}}
\renewcommand{\P}{{\mathbb{P}}}
\newcommand{\Q}{{\mathbb{Q}}} 
\newcommand{\R}{{\mathbb{R}}}

\newcommand{\sA}{{\mathscr{A}}}
\newcommand{\sB}{{\mathscr{B}}}
\newcommand{\sC}{{\mathscr{C}}}

\newcommand{\sF}{{\mathscr{F}}}
\newcommand{\sG}{{\mathscr{G}}} 
\newcommand{\sH}{{\mathscr{H}}}

\newcommand{\sJ}{{\mathscr{J}}}

\newcommand{\sL}{{\mathscr{L}}}

\newcommand{\sP}{{\mathscr{P}}}
 
\newcommand{\sR}{{\mathscr{R}}}
\newcommand{\sS}{{\mathscr{S}}}

\newcommand{\sW}{{\mathscr{W}}}
\newcommand{\sX}{{\mathscr{X}}}

\DeclareMathAlphabet{\mathscrbf}{OMS}{mdugm}{b}{n}

\newcommand{\sbS}{{\mathscrbf{S}}}

\newcommand{\cA}{{\mathcal{A}}}
\newcommand{\cB}{{\mathcal{B}}}

\newcommand{\cF}{{\mathcal{F}}}

\newcommand{\cK}{{\mathcal{K}}}
 
\newcommand{\cM}{{\mathcal{M}}}
\newcommand{\cN}{{\mathcal{N}}}

\newcommand{\cP}{{\mathcal{P}}}

\newcommand{\cS}{{\mathcal{S}}} 

\newcommand{\cU}{{\mathcal{U}}}

\newcommand{\gh}{{\mathbf{h}}}

\newcommand{\gs}{{\mathbf{s}}} 
\newcommand{\gt}{{\mathbf{t}}}
\newcommand{\gu}{{\mathbf{u}}}
\newcommand{\gv}{{\mathbf{v}}}
\newcommand{\gw}{{\mathbf{w}}}
\newcommand{\gx}{{\mathbf{x}}}

\newcommand{\gB}{{\mathbf{B}}}

\newcommand{\gI}{{\mathbf{I}}}

\newcommand{\gL}{{\mathbf{L}}}

\newcommand{\gP}{{\mathbf{P}}}

\newcommand{\gS}{{\mathbf{S}}} 
\newcommand{\gT}{{\mathbf{T}}}

\newcommand{\gX}{{\mathbf{X}}}

\newcommand{\bs}[1]{\boldsymbol{#1}}
\newcommand{\bsX}{{\bs{X}}}

\newcommand{\bsZ}{\bs{Z}}

\newcommand{\gpi}{{\pi}}

\newcommand{\gPsi}{\bs{\Psi}}

\newcommand{\gtheta}{{\bs{\theta}}}
\newcommand{\gvartheta}{\bs{\vartheta}}

\newlist{lista}{enumerate}{1}
\setlist[lista,1]{label=\alph*),ref=\alph*)}

\newlist{listi}{enumerate}{1}
\setlist[listi,1]{label=(\roman*),ref=(\roman*),align=left}

\newcommand{\eref}[1]{(\ref{#1})}
\renewcommand{\ge}{\geqslant}
\renewcommand{\le}{\leqslant}
\newcommand{\1}{1\hskip-2.6pt{\rm l}}
\newcommand{\<}{{\langle}}
\renewcommand{\>}{{\rangle}}

\newcommand{\etc}[1]{#1_1,\ldots,#1_n}

\newcommand{\str}[1]{\rule{0mm}{#1mm}}
\newcommand{\st}{\strut}

\newcommand{\dps}[1]{\displaystyle{#1}}

\newcommand{\on}{^{\otimes n}}

\newcommand{\et}{^{\star}}
\newcommand{\eps}{{\varepsilon}}

\renewcommand{\gX}{{\bs{X}}}
\newcommand\p{{\gpi_{\!\gX}^{L}}}

\endlocaldefs

\usepackage[nokeyprefix]{refstyle}
\usepackage{varioref}
\usepackage{xr-hyper}
\usepackage{hyperref}
\begin{document}

\begin{frontmatter}

\title{Robust Bayes-Like Estimation: Rho-Bayes estimation}
\runtitle{Rho-Bayes estimation}
\author{\fnms{Yannick} \snm{Baraud}\corref{}\ead[label=e1]{yannick.baraud@uni.lu}\thanksref{m1}}
\address{University of Luxembourg,\\
Mathematics Research Unit,\\
Maison du nombre,\\
6 avenue de la Fonte,\\
L-4364 Esch-sur-Alzette,\\
Grand Duchy of Luxembourg.\\ \printead{e1}}
\and
\author{\fnms{Lucien} \snm{Birg\'e}\ead[label=e2]{lucien.birge@upmc.fr}\thanksref{m2}}
\address{Sorbonne Universit\'e\\
CNRS, Laboratoire de Probabilit\'es,\\ \hspace{5mm}Statistique et Mod\'elisation (LPSM)\\
Case courrier 158\\
75252 Paris Cedex 05\\
France.\\ \printead{e2}}
\affiliation{University of Luxembourg\thanksmark{m1} and Sorbonne Universit\'e \thanksmark{m2}}


\maketitle

\begin{abstract} 
We observe $n$ independent random variables with joint distribution $\gP$ and pretend that they are i.i.d. with some common density $s$ (with respect to a known measure $\mu$) that we wish to estimate. We consider a density model $\overline S$ for $s$ that we endow with a prior distribution $\pi$ (with support in $\overline S$) and build a robust alternative to the classical Bayes posterior distribution which possesses similar concentration properties around $s$ whenever the data are truly i.i.d.\ and their density $s$  belongs to the model $\overline S$. Furthermore, in this case, the Hellinger distance between the classical and the robust posterior distributions tends to 0, as the number of observations tends to infinity, under suitable assumptions on the model and the prior. However, unlike what happens with the classical Bayes posterior distribution, we show that the concentration properties of this new posterior distribution are still preserved when the model is misspecified or when the data are not i.i.d. but the marginal densities of their joint distribution are close enough in Hellinger distance to the model $\overline S$. 
\end{abstract}



\end{frontmatter}

\section{Introduction}\label{I}
The purpose of this paper is to define and study a robust substitute to the classical posterior
distribution in the Bayesian framework. It is known that the posterior is not robust with respect to misspecifications of the model. More precisely, if the true distribution $P$ of an $n$-sample $\bsX=(X_{1},\ldots,X_{n})$  does not
belong to the support $\sP$ of the prior and even if it is close to this support in total variation or Hellinger distance, the posterior may concentrate around a point of this support which is quite far from the truth. A simple example is the following one. 

Let $P_t$ be the uniform distribution on $[0,t]$ with $t\in \overline S=(0,+\infty)$ and, given $a>0$ and $\alpha>1$, let $\pi$ be the prior with density $Ct^{-\alpha}\1_{[a,+\infty)}(t)$, $C=(\alpha-1)^{-1}a^{1-\alpha}$, with respect to the Lebesgue measure on $\R_+$. Given an $n$-sample $\bsX=(X_{1},\ldots,X_{n})$ with distribution $P_{t_{0}}$, the posterior distribution function writes as 
\begin{equation}
t\mapsto G^{L}(t|\gX)=\left[1-\left(\frac{a\vee X_{(n)}}{t}\right)^{n+\alpha-1}\right]\1_{[a\vee X_{(n)},+\infty)}(t)
\label{Eq-nonrobust}
\end{equation}
and, for $t_{0}>a$, we see that this posterior  is highly concentrated on intervals of the form $\left[a\vee X_{(n)},\left(1+cn^{-1}\right)\left(a\vee X_{(n)}\right)\right]$ with $c>0$ large enough. Now assume that the true distribution has been contaminated and is rather 
\[
P=\left(1-n^{-1}\right)\cU([0,t_0])+n^{-1}\cU\left([t_0+100,t_{0}+100+n^{-1}]\right).
\] 
Although it is quite close to the initial distribution $P_{t_0}$ in variation distance (their distance is $1/n$), on an event of probability $1-(1-n^{-1})^{n}>1/2$, $t_0+100<X_{(n)}<t_0+100+n^{-1}$ and the posterior distribution is therefore concentrated around $t_0+100$ according to (\ref{Eq-nonrobust}). The same problem would occur if we were using the maximum likelihood estimator (MLE for short) as an estimator of $t$. 

In the literature, most results  about the behaviour of the posterior do not say anything about  misspecification. Some papers like Kleijn and van der Vaart~\citeyearpar{MR2283395},~\citeyearpar{MR2988412} and Panov and Spokoiny~\citeyearpar{MR3420819} address this problem but their results involve the behaviour of the Kullback-Leibler divergence between $P$ and the distributions in $\sP$, as is also often the case when studying the MLE --- see for instance Massart~\citeyearpar{MR2319879} ---. However, two distributions may be very close in Hellinger distance and therefore indistinguishable with our sample $\bsX$, but have a large Kullback-Leibler divergence. 

Even when the model is exact, the Kullback-divergence is used to analyze the properties of the Bayes posterior. It is known mainly from the work of van der Vaart and co-authors --- see in particular Ghosal, Ghosh and van der Vaart~\citeyearpar{MR1790007} --- that the posterior distribution concentrates around $P\in\sP$ as $n$ goes to infinity but those general results require that the prior puts enough mass on neighbourhoods of $P\in\sP$ of the form $\cK(P,\eps)=\{P'\in\sP,\; K(P,P')<\eps\}$ where $\eps$ is a positive number and $K(P,P')$ the Kullback-Leibler divergence between $P$ and $P'$. Unfortunately such neighbourhoods may be empty (and consequently the condition unsatisfied) when the probabilities in $\sP$ are not equivalent, which is for example the case for the translation model of the uniform distribution on $[0,1]$, even though the Bayes method may work well in such cases.

As already mentioned, the lack of robustness is not specific to the Bayesian framework but has also been noticed for the MLE. Alternatives to the MLE that remedy this
lack of robustness have been considered many years ago by Le Cam (\citeyearpar{MR0334381}, \citeyearpar{MR0395005}, \citeyearpar{MR856411}) and Birg\'e (\citeyearpar{MR722129}, \citeyearpar{MR762855}, \citeyearpar{MR2320111}) but have some limitations. A new recent approach leading to what we called {\it $\rho$-estimators} and described in Baraud, Birg\'e and Sart~\citeyearpar{MR3595933}, hereafter BBS for short, and Baraud and Birg\'e~\citeyearpar{BarBir2018}, hereafter BB, corrects a large part of these limitations. It also improves over the previous constructions since it recovers some of the nice properties of the MLE, like efficiency, under suitably strong regularity assumptions.

The aim of this paper is to extend the theory developed in BBS and BB to a Bayesian paradigm in view of designing a robust substitute to the classical Bayes posterior distribution. To be somewhat more precise, let us consider a classical Bayesian framework of density estimation from $n$ i.i.d.\ observations, although other situations could be considered as well. We observe $\gX=(\etc{X})$ where the $X_i$ belong to some measurable space $(\sX,\sA)$ with an unknown distribution $P$ on $\sX$. We have at disposal a family $\sP=\left\{P_t, t\in \overline S\right\}$ of possible distributions on $\sX$, which is dominated by a $\sigma$-finite measure $\mu$ with respective densities $f(x|t)=(dP_t/d\mu)(x)$. We set $f(\gX|t)=\prod_{i=1}^nf(X_i|t)$ for the likelihood of $t$. Assuming that $\overline S$ is a measurable space endowed with a $\sigma$-algebra $\sS$, we choose a prior distribution $\pi$ on $\overline S$ which leads to a posterior $\pi_\gX^{L}$ that is absolutely continuous with respect to $\pi$ with density $g^{L}(t|\gX)=(d\pi_\gX^{L}/d\pi)(t)$. Following these notations, the log-likelihood function and log-likelihood ratios write respectively as $L(\gX|t)=\log \left(f(\gX|t)\st\right)=\sum_{i=1}^n\log \left(f(X_i|t)\st\right)$ and $\gL(\gX,t,t')=L(\gX|t')-L(\gX|t)$
so that the density $g^{L}(t|\gX)$ of the posterior distribution $\p$ with respect to $\pi$ is given by
\[
\frac{\exp\left[L(\gX|t)\right]}{\dps{\int_{\overline S}}\exp\left[L(\gX|t)\right]d\pi(t)}
=\frac{\dps{\exp\left[L(\gX|t)-\sup_{t'\in {\overline S}}L\left(\gX|t'\right)\right]}}{\dps{\int_{\overline S}}
\exp\left[L(\gX|t)-\sup_{t'\in {\overline S}}L\left(\gX|t'\right)\right]d\pi(t)}
\]
and consequently, 
\begin{equation}
g^{L}(t|\gX)=\frac{f(\gX|t)}{\dps{\int_{\overline S}}f(\gX|t)\,d\pi(t)}=\frac{\dps{\exp\left[-\sup_{t'\in {\overline S}}\gL\left(\gX,t,t'\right)\right]}}
{\dps{\int_{\overline S}}\exp\left[-\sup_{t'\in {\overline S}}\gL\left(\gX,t,t'\right)\right]d\pi(t)}.
\label{Eq-post}
\end{equation}
Note that, if the MLE $\widehat{t}(\gX)$ exists,
\[
\sup_{t'\in {\overline S}}\gL\left(\gX,t,t'\right)=L\left(\gX\!\left|\widehat{t}(\gX)\right.\right)-L(\gX|t)
\]
and that we could as well consider, for all $\beta>0$ the distributions
\[
g_\beta^{L}(t|\gX)\cdot\pi\qquad\mbox{with}\qquad g_\beta^{L}(t|\gX)=
\frac{\exp\left[\beta L(\gX|t)\right]}{\dps{\int_{\overline S}}\exp\left[\beta L(\gX|t)\right]d\pi(t)}.
\]
The posterior corresponds to $\beta=1$ and when $\beta$ goes to infinity the distribution $g_\beta^{L}(t|\gX)\cdot\pi$ converges weakly, under mild assumptions, to the Dirac measure located at the MLE. All values of $\beta\in(1,+\infty)$ will then lead to interpolations between the posterior and the Dirac at the MLE.

Most problems connected with the maximum likelihood or Bayes estimators are due to the fact that the log-likelihood ratios $\gL(\gX,t,t')$ involve the logarithmic function which is unbounded. As a result, we may have 
\[
\E_t\left[\gL(\gX,t,t')\right]=-n\E_t\left[\log(dP_t/dP_{t'})(X_1)\right]=-\infty,
\]
the situation being even more delicate when the true distribution of the $X_i$ is different (even slightly) from $P_t$.

In BBS and BB we offered an alternative to the MLE by replacing the logarithmic function in the log-likehood ratios by other ones. One possibility being the function $\varphi(x)$ defined by
\[
\varphi(x)=4\frac{\sqrt{x}-1}{\sqrt{x}+1}\quad\mbox{for all }x\ge0,
\]
so that, for $x>0$,
\[
\varphi'(x)=\frac{4}{(1+\sqrt{x})^{2}\sqrt{x}}>0\qquad\mbox{and}\qquad\varphi''(x)=-\frac{2(1+3\sqrt{x})}{(1+\sqrt{x})^{3}x^{3/2}}<0.
\]
Like the log function, $\varphi(x)$ is increasing, concave and satisfies $\varphi(1/x)=-\varphi(x)$. In fact, these two functions coincide at $x=1$, their first and second derivatives as well and for all $x\in [1/2,2]$
\begin{equation}
0.99<\frac{\varphi(x)}{\log x}\le1\quad\mbox{and}\quad
|\varphi(x)-\log x|\le0.055|x-1|^{3}.
\label{eq-philog}
\end{equation}
The main advantage of the function $\varphi$ as compared to the log function lies in its boundedness. It can also be extended to $[0,+\infty]$ by continuity by setting $\varphi(+\infty)=4$. As a consequence, the quantity $\varphi\left(t'(X)/t(X)\st\right)$
is well-defined (with the convention $a/0=+\infty$ for $a>0$ and $0/0=1$) and bounded and we can use it as a surrogate for $\log\left(t'(X)/t(X)\st\right)$. This suggests the replacement of $\gL\left(\gX,t,t'\right)$ by $4\gPsi(\gX,t,t')$ where the function $\gPsi$ is defined as 
\begin{equation}
\gPsi(\gx,t,t')=\sum_{i=1}^{n}\psi\left(\sqrt{t'(x_i)\over t(x_i)}\right)\quad\text{for all }\gx\in\sX^{n}\;\;\text{and}\;\;(t,t')\in\overline{S}^{2},
\label{eq-psi}
\end{equation}
with the conventions $0/0=1$, $a/0=+\infty$ for $a>0$ and 
\begin{equation}\label{eq-defpsi}
\psi(x)=
\begin{cases} 
\dps{x-1\over x+1}&\text{for $0\le x<+\infty$},\\
1&\text{for $x=+\infty$},
\end{cases}
\end{equation}
so that $\varphi(x)=4\psi(\sqrt{x})$. Note that $\psi$ is Lipschitz with Lipschitz constant $2$. The important point here is that we have already studied in details in BB the behaviour and properties of a process which is closely related to $(t,t')\mapsto \gPsi(\gX,t,t')$.

We get a pseudo-posterior density with respect to $\pi$ by replacing in (\ref{Eq-post}) the quantity $\sup_{t'\in {\overline S}}\gL(\gX,t,t')$ by $4\sup_{t'\in {\overline S}}\gPsi\left(\gX,t,t'\right)$. This pseudo-posterior density can therefore be written\vspace{1mm}
\[
g(t|\gX)=\frac{\dps{\exp\left[-4\sup_{t'\in {\overline S}}\gPsi\left(\gX,t,t'\right)\right]}}{\dps{\int_{\overline S}}\exp\left[-4\sup_{t'\in {\overline S}}\gPsi\left(\gX,t,t'\right)\right]d\pi(t)}.
\]
More generally we may consider, for $\beta>0$, the random distribution $\pi_{\!\gX}$ given by 
\begin{equation}
{d\pi_{\!\gX}\over d\pi}(t)=\frac{\dps{\exp\left[-\beta\sup_{t'\in {\overline S}}
\gPsi\left(\gX,t,t'\right)\right]}}{\dps{\int_{\overline S}}\exp\left[-\beta\sup_{t'\in {\overline S}}
\gPsi\left(\gX,t,t'\right)\right]d\pi(t)}.
\label{Eq-psB}
\end{equation}
This will be the starting point for our study of this {\em Bayes-like} framework with a {\em posterior-like} distribution $\pi_{\!\gX}$ defined by (\ref{Eq-psB}) that will play a similar role as the posterior distribution in the classical Bayesian paradigm except for the fact that a random variable with distribution $\pi_{\!\gX}$ (conditionally to our sample $\gX$) will possess robustness properties with respect to the hypothesis that $P$ belongs to $\sP$. We shall call it {\em $\rho$-posterior} by analogy with our construction of $\rho$-estimators as described in BBS and BB.

To conclude this introduction, let us emphasize the specific properties of our method that distinguish it from classical Bayesian procedures.\vspace{2mm}\\
--- Contrary to the classical Bayesian framework, concentration properties of the $\rho$-Bayes method do not involve the Kullback-Leibler divergence but only the Hellinger distance.\vspace{1mm}\\
--- Our results are non-asymptotic and given in the form of large deviations of the pseudo-posterior distribution from the true density for a given value $n$ of the number of observations.\vspace{1mm}\\
--- The method is robust to Hellinger deviations: even if the true distribution is at some positive Hellinger distance of the support of the prior, the posterior will behave almost as well as if this were not the case provided that this distance is small.\vspace{1mm}\\
--- Due to the just mentioned robustness properties, we may work with an approximate model for the true density. In particular, when the density is assumed to belong to a non-parametric set $\cS$, it is actually enough to apply our $\rho$-Bayes procedure on a parametric set $\overline{S}$ possessing good approximation properties with respect to the elements of $\cS$. Besides, starting from a continuous prior on a continuous model, we can discretize both of them without loosing much provided that our discretization scale is small enough.\vspace{1mm}\\
--- The $\rho$-posterior also possesses robustness properties with respect to the assumption that the data are i.i.d. provided that the densities of the $X_{i}$ are close enough to the model $\overline S$. \vspace{1mm}

Substituting another function to the log-likelihood in the expression of the posterior distribution, as we do here, is not new in the literature. It has often been motivated by the will of replacing the Kullback-Leibler loss, which is naturally associated to the likelihood-function, by other losses that are  more specifically associated to the problem that needs to be solved (estimation of a mean, classification, etc.) or to deal with the problem of misspecification. This approach leads to {\em quasi-posterior distributions}  which properties have been studied by many authors among which Chernozhukov and Hong~\citeyearpar{MR1984779} and Bissiri~{\em et al.}~\citeyearpar{MR3557191} (see also the references therein). These results do not include robustness but Chernozhukov and Hong~\citeyearpar{MR1984779} proved some analogues of the Bernstein-von Mises theorem under suitable assumptions on the model and loss function. The use of {\em fractional likelihoods} by Jiang and Tanner~\citeyearpar{MR2458185} was motivated by the problem of misspecification. In a sparse parametric framework (the true parameter $\theta\in\R^{d}$ has a small number of nonzero components), Atchad\'e~\citeyearpar{MR3718168} replaces the joint density $f_{n,\theta}$ of the observations by a suitable function $q_{n,\theta}$. Together with a prior that forces sparsity, this results in tractable and consistent procedures for high-dimensional parametric problems. All the cited results are of an asymptotic nature contrary to the next one. Bhattacharya, Pati and Yang~\citeyearpar{MR3909926} investigate the replacement, in the definition of the posterior, of the likelihood by a fractional one, also considering the case of misspecified models, but use what they call $\alpha$-divergences instead of the KL one (but which may also be infinite) to evaluate the amount of misspecification.

Closer to our approach is the PAC-Bayesian one that has been developed  by Olivier Catoni~\citeyearpar{MR2483528} and our parameter $\beta$ in the definition of the $\rho$-posterior~\eref{Eq-psB} refers to the (inverse) of the so-called temperature parameter in the definition of the Gibbs measure. This parameter essentially plays no role in our results.

The paper is organized as follows. In Section~\ref{N}, we describe our framework and state our main assumption that allows to solve the measurability issues that are inherent to the construction of the posterior. An account of what can be achieved with a $\rho$-posterior distribution is presented and commented in Section~\ref{sect-flavour} in the density and regression frameworks (with a random design). 
Our main result can be found in Section~\ref{M} where we present the concentration properties of our $\rho$-posterior distribution. These properties involve two quantities, one which depends on the choice of the prior while the other is independent of it but depends on the model and the true density. We show how one can control these quantities in Sections~\ref{Sect-UPBeta} and~\ref{Sect-UPBeps} respectively, giving there illustrative examples as well as general theorems that can be applied to many parametric models of interest. 
Our results on the connection between the classical Bayes posterior and the $\rho$-one are presented in Section~\ref{RB}. We show that under suitable assumptions on the density model and the prior, the Hellinger distance between these two distributions tends to 0 at rate $n^{-1/4}(\log n)^{3/4}$ as the sample size $n$ tends to infinity. In particular, this result shows that under suitable assumptions our $\rho$-Bayes posterior satisfies a Bernstein-von Mises Theorem. The problem of a hierarchical prior or, equivalently, that of model selection is handled in Section~\ref{MM}. The proofs and discussions about measurability issues are to be found in the supplemental article [Baraud and Birg\'e~\citeyearpar{supp-bayes}] while additional results and examples can be found in the original version of this paper, Baraud and Birg\'e~\citeyearpar{Baraud:2017aa}.

\section{Framework, notations and basic assumptions}\label{N}

\subsection{The framework and the basic notations}\label{N1}
We actually want to deal with more general situations than the one we presented in the introduction, namely the case of independent but possibly non-i.i.d.\ observations, even though the Statistician assumes them to be i.i.d. By doing so, our aim is to emphasize the robustness property of our $\rho$-posterior distribution with respect to the assumption that the data are i.i.d. This generalization leads to the following statistical framework. For $n\in\N\et=\N\setminus\{0\}$, we observe a random variable $\gX=(X_{1},\ldots,X_{n})$ defined on $(\Omega,\Xi)$, where the 
$X_{i}$ are independent with values in a measurable space $(\sX,\sA)$ endowed with a $\sigma$-finite mesure $\mu$. We denote by $\sL$ the set of all probability densities $u$ with respect to $\mu$ (which means that $u$ is a nonnegative measurable function on $\sX$ such that $\int_{\sX}u(x)\,d\mu(x)=1$) and by $P_{u}=u\cdot\mu$ the probability on $(\sX,\sA)$ with density $u\in\sL$. We assume that for each $i\in\{1,\dots,n\}$, $X_{i}$ admits a density with respect to $\mu$, i.e.\ has distribution $P_{s_{i}}=s_{i}\cdot \mu$ with $s_{i}\in\sL$. We set $\gs=(s_{1},\ldots,s_{n})$ and denote by $\P_{\gs}$ the probability on $(\Omega,\Xi)$ that gives $\gX$ the distribution $\gP_{\gs}=\bigotimes_{i=1}^nP_{s_{i}}$ on $\sX^{n}$ and by $\E_{\gs}$ the corresponding expectation. We shall abusively refer to $\gs$ as the (true) density of $\bsX$. 

We denote by $|A|$ the cardinality of a finite set $A$ and use the word {\it countable} for {\it finite or countable}. Parametric models will be indexed by some subset $\bs{\Theta}$ of $\R^{d}$ and $|\cdot|$ will denote the Euclidean norm on $\R^{d}$. Finally, we shall often use the inequalities
\begin{equation}
2ab\le\alpha a^{2}+\alpha^{-1}b^{2};\quad(a+b)^{2}\le(1+\alpha)a^{2}+
(1+\alpha^{-1})b^{2}\quad\mbox{for all }\alpha>0.
\label{eq-2ab}
\end{equation}

\subsection{Hellinger type metrics}\label{N2}
For all $t,t'\in\sL$, we shall write $h(t,t')$ and $\rho(t,t')$ for the Hellinger distance and affinity between $P_{t}$ and $P_{t'}$. We recall that the Hellinger distance and affinity between two probabilities $P,Q$ on a measurable space $(\sX,\sA)$ are given respectively by
\[
h(P,Q)=\cro{{1\over 2}\int_{\sX}\left(\sqrt{{dP\over d\nu}}-\sqrt{{dQ\over d\nu}}\right)^{2}d\nu}^{1/2};\quad\rho(P,Q)=\int_\sX\sqrt{{dP\over d\nu} {dQ\over d\nu}}\,d\nu,
\]
where $\nu$ denotes an arbitrary measure which dominates both $P$ and $Q$, the result being independent of the choice of $\nu$. It is well-known since Le Cam~\citeyearpar{MR0334381} that $0\le\rho(P,Q)=1-h^2(P,Q)$ and that the Hellinger distance is related to the total variation distance by the following inequalities:
\begin{equation}\label{eq-h-TV}
h^2(P,Q)\le\sup_{A\in\cA}|P(A)-Q(A)|\le h(P,Q)\sqrt{2-h^2(P,Q)}\le\sqrt{2}h(P,Q).
\end{equation}
Therefore robustness with respect to the Hellinger distance implies robustness with respect to the total variation distance. 

The Hellinger closed ball centred at $t\in\sL$ with radius $r>0$ is denoted $\sB(t,r)$ and, for $\gs\in\sL^{n}$, we define
\[
\sB(\gs,r)=\ac{t\in\sL,\; h^{2}(\gs,t)\le r^{2}}\quad \text{with}\quad h^{2}(\gs,t)=\frac{1}{n}\sum_{i=1}^{n}h^{2}(s_{i},t)\le1.
\]
Then, for $S\subset \sL$, we set $\sB^{S}(t,r)=S\cap \sB(t,r)$ and $\sB^{S}(\gs,r)=S\cap\sB(\gs,r)$.
If the $X_{i}$ are truly i.i.d.\ with  density $s$, $\gs=(s,\ldots,s)$ and $h^{2}(\gs,t)=h^{2}(s,t)$ for all $t\in \sL$, hence $\sB^{S}(\gs,r)=\sB^{S}(s,r)$. 

Note that although $h$ is a genuine distance on the space of all probabilities on $\sX$, therefore on $\{P_{t},\ t\in\sL\}$, it is only a {\em pseudo-distance} on $\sL$ itself since $h(t,t')=0$ if $t\ne t'$ but $t=t'$ $\mu$-a.e. For simplicity, we shall nevertheless still call $h$ a distance on $\sL$ and set $h(t,A)=\inf_{u\in A} h(t,u)$ for the distance of a point $t\in\sL$ to the subset $A$ of $\sL$. Similarly $h(\gs,A)=\inf_{t\in A}h(\gs,t)$. We recall that a pseudo-distance $d$ satisfies the axioms of a distance apart from the fact that one may have $d(x,y)=0$ with $x\ne y$.

\subsection{Models and main assumptions}\label{N4}
We consider a {\em density model} $\overline S$, i.e.\ a subset of $\sL$, acting as if the data were i.i.d., and our aim is to estimate the $n$-uple $\gs=(s_{1},\ldots,s_{n})$ from the observation of $\gX$ on the basis of this model. Adopting the Bayesian paradigm, we endow $\overline S$ with a $\sigma$-algebra $\sS$ as well as a prior $\pi$ on $(\overline S,\sS)$. 
There is no reason for $t\mapsto \gPsi(\gX,t,t')$ defined by~\eref{eq-psi} and $t\mapsto \sup_{t'\in \overline S}\gPsi(\gX,t,t')$ to be measurable functions of $t$ on $(\overline S,\sS)$ and the function $\omega \mapsto \sup_{t'\in \overline S}\gPsi(\gX(\omega),t,t')$ to be a random variable on $(\Omega,\Xi)$. Therefore our $\rho$-posterior distribution $\pi_{\!\gX}$, as given by (\ref{Eq-psB}), might not be well-defined. 
In order to overcome these difficulties, we introduce the following assumption and also slightly modify the definition of our $\rho$-posterior distribution that was originally given by~\eref{Eq-psB} in the density framework. 
The following assumption ensures that the sets and random variables that we shall introduce later are suitably measurable. We refer the reader to the supplemental article [Baraud and Birg\'e~\citeyearpar{supp-bayes}] for a discussion about Assumption~\ref{hypo-mes} and how it can be checked on examples. 
%
\begin{ass}\label{hypo-mes}\mbox{}\vspace{-3mm}
\begin{listi}
\item\label{Hi} The function $(x,t)\to t(x)$ on $\sX\times\overline{S}$ is measurable with respect to the $\sigma$-algebra $\sA\otimes\sS$.
\item\label{Hii} There exists a countable subset $S$ of $\overline S$ and, given $t\in\overline{S}$ and $t'\in S$, one can find a sequence $(t_{k})_{k\ge0}$ in $S$ such that, for all $x\in\sX$,
\begin{equation}
\lim_{k\to +\infty}t_{k}(x)=t(x)\quad\mbox{and}\quad
\lim_{k\to +\infty}\psi\left(\sqrt{t'(x)\over t_{k}(x)}\right)=\psi\left(\sqrt{t'(x)\over t(x)}\right).
\label{eq-t-t_k}
\end{equation}
%
\end{listi}
\end{ass}
Note that it follows from Proposition~\ref{prop-ass1} in the supplemental article [Baraud and Birg\'e~\citeyearpar{supp-bayes}]  that $S$ is dense in $\overline S$ with respect to the distance $h$. 
Of course, when $\overline{S}$ is countable, we shall set $S=\overline{S}$ without further notice and Assumption~\ref{hypo-mes}-$\ref{Hii}$ will be automatically satisfied with the $\sigma$-algebra $\sS$ gathering all the subsets of $\overline S$.  In the sequel, we shall always assume that the set $S$ associated to the model $\overline{S}$ has been fixed once and for all. 

The following proposition (to be proven in the supplemental article) ensures that the measurability properties required for a proper definition of the posterior distribution hold.
%
\begin{prop}\label{prop-mes}
Under Assumption~\ref{hypo-mes}, given $t'\in S$ and $\gPsi\left(\gx,t,t'\right)$ defined by~\eref{eq-psi}, the functions
\[
(\gx,t)\mapsto\gPsi\left(\gx,t,t'\right)\quad\text{and}\quad
(\gx,t)\mapsto\gPsi\left(\gx,t\right)=\sup_{u\in S}\gPsi\left(\gx,t,u\right)
\]
are measurable with respect to the $\sigma$-algebra $\sA\otimes\sS$. Hence the function
\[
\gx\mapsto\int_{\overline S}\exp\left[-\beta\gPsi(\gx,t)\right]d\pi(t)
\]
is measurable with respect to $\sA$ and the function $t\mapsto h(t,s)$ is measurable with respect to $\sS$ whatever $s\in\sL$.
\end{prop}
%

\subsection{The $\rho$-posterior distribution $\pi_{\!\gX}$}
Let $S$ be the countable subset of $\overline{S}$ provided by Assumption~\ref{hypo-mes}. For $\omega\in\Omega$ and $\beta>0$, we define the distribution $\pi_{\!\gX(\omega)}$ on $\overline{S}$ by its density with respect to the prior $\pi$:
\begin{equation}\label{eq-LPS2}
{d\pi_{\!\gX(\omega)}\over d\pi}(t)=g(t|\gX(\omega))=\frac{\exp\left[-\beta\gPsi\left(\gX(\omega),t\right)\right]}{\dps{\int_{\overline{S}}}\exp\left[-\beta\gPsi\left(\gX(\omega),t'\right)\right]d\pi(t')}.
\end{equation}
Proposition~\ref{prop-mes} implies that the function $(\omega,t)\mapsto g(t|\gX(\omega))$ is measurable with respect to the $\sigma$-algebra $\Xi\otimes\sS$. We recall that the choice of $\beta=4$ leads to an analogue of the classical Bayes posterior since the function $x\mapsto4\psi\left(\sqrt{x}\right)$ is close to $\log x$ as soon as $x$ is not far from one. Throughout the paper the parameter $\beta$ will remain fixed and part of our results will depend on it.
%
\begin{df}\label{D-Bayesrho}
The method that leads from the set $ \overline{S}$ and the prior $\pi$ on $\overline{S}$ to the distribution $\pi_{\!\gX}$ (and all related estimators) will be called {\em $\rho$-Bayes estimation} and $\pi_{\!\gX}$ is the {\em $\rho$-posterior} distribution.
\end{df}
\section{A flavour of what a $\rho$-Bayes procedure can achieve\label{sect-flavour}}
Throughout this section we take $\beta=4$, the value for which the $\rho$-posterior distribution is the analogue of the classical Bayes posterior.

\subsection{The density framework}
Let $\overline S$ be a density model for the supposed common density of our observations $X_{1},\ldots,X_{n}$ and consider the following entropy condition.
\begin{ass}\label{hypo-simple-entro}
There exists a non-increasing function $H$ from $(0,1]$ to $[3,+\infty)$ such that, for any $\eps\in (0,1]$, there exists a subset $S_{\eps}$ of $\overline S$ with cardinality not larger than $\exp\cro{H(\eps)}$ and such that $h(t,S_{\eps})\le \eps$ for all $t\in\overline S$.
\end{ass}
%
\begin{prop}\label{prop-densite}
Let $\overline{S}$ satisfy Assumption~\ref{hypo-simple-entro} and $\varepsilon_{n}$ be such that 
\begin{equation}\label{def-epsn}
\eps_{n}\ge 1/\left(2\sqrt{n}\right)\qquad \text{and}\qquad H(\eps_{n})\le\left(4\cdot 10^{-6}\right)n\eps_{n}^{2}.
%
\end{equation}
There exists a prior $\pi$ on $\overline S$ (depending on $\eps_{n}$ only) such that, whatever the true density $\gs=(s_{1},\dots,s_{n})$ and $\xi>0$, there exists a measurable subset $\Omega_{\xi}$ of $\Omega$ satisfying $\P_{\gs}(\Omega_{\xi})\ge1-e^{-\xi}$ and for all $\omega\in \Omega_{\xi}$,
%
\[
\pi_{\!\gX(\omega)}\pa{\ac{t\in\overline S,\; h(\gs,t)\le C\overline r_{n}}}\ge 1-e^{-\xi'}\quad \text{for all $\xi'>0$},
\]
with $C$ a positive universal constant and
\[
\overline r_{n}=h(\gs,\overline S)+\eps_{n}+\sqrt{\frac{\xi+\xi'}{n}}.
\]
In particular, if $X_{1},\ldots,X_{n}$ are truly i.i.d.\ with density $s\in\sL$,
\[
\pi_{\!\gX(\omega)}\pa{\sB^{\overline S}(s,C\overline r_{n})}\ge 1-e^{-\xi'}\quad \text{with}\quad \overline r_{n}=h(s,\overline S)+\eps_{n}+\sqrt{\frac{\xi+\xi'}{n}}.
\]
\end{prop}

This result shows that with probability close to 1, the $\rho$-posterior distribution concentrates around points $t$ in the density model $\overline S$ which satisfy 
\[
\cro{\frac{1}{n}\sum_{i=1}^{n}h^{2}(s_{i},t)}^{1/2}\le C\overline r_{n}\quad \text{with $\overline r_{n}$ of order}\quad h(\gs,\overline S)+\eps_{n}.
\]
The quantity $\eps_{n}$ corresponds to the concentration rate we get when the $X_{i}$ are truly i.i.d.\ with density in $\overline S$. For instance, when $H(\eps)=A\eps^{-V}$ for all $\eps>0$ and some constants $A,V>0$, $\eps_{n}$ is of order $n^{-1/(V+2)}$.   
This concentration rate remains of the same order as long as $h(\gs,\overline S)$ is small enough compared to $\eps_{n}$, which is actually possible even when none of the densities $s_{i}$ belongs to $\overline S$. This stability result accounts for the robustness property of our procedure. 

It is well-known --- see for instance Birg\'e~\citeyearpar{MR722129} and \citeyearpar{MR816706} --- that, in many cases, the smallest value of $\eps_{n}$ which satisfies (\ref{def-epsn}) corresponds to the minimax rate of estimation (with respect to $n$) over $\overline S$. Here are two typical illustrations for densities with respect to the Lebesgue measure. 

i) Assume that $\overline S$ is the set of all non-increasing densities on $[0,1]$ which are bounded by $M<+\infty$. Of course, if $s\in\overline{S}$, $\sqrt{s}$ is also non-increasing and is bounded by $\sqrt{M}$ and the Hellinger entropy of $\overline{S}$ corresponds to the $\L_{2}$-entropy of the set $\left\{\sqrt{s},\;s\in\overline{S}\right\}$ which is known from van de Geer~\citeyearpar{MR1739079} to be bounded by $A\varepsilon^{-1}$ leading to an $\varepsilon_{n}$ of order $n^{-1/3}$ which is known to be the minimax rate for this problem. 

ii) If $\overline{S}$ is the set of $\alpha$-H\"olderian densities on $[0,1]^{d}$ with $\alpha>0$, its Hellinger entropy is known from Birg\'e~\citeyearpar{MR816706} to be of order $\varepsilon^{-2d/\alpha}$ leading to a convergence rate with respect to $n$ of order $n^{-\alpha/2(\alpha+d)}$. All details can be found in Birg\'e~\citeyearpar{MR816706} (see in particular his Corollary~3.2) where it is also proved that this rate is minimax (see his Proposition~4.3).

Assumption~\ref{hypo-simple-entro} can actually be replaced by the more general one that $\overline S$ admits a metric dimension $D$ (according to Definition~\ref{def-MD} below) in which case the same conclusion holds with $\eps_{n}\ge 1/(2\sqrt{n})$ satisfying $D(\eps_{n})\le 10^{-6}n\eps_{n}^{2}$. 

\subsection{The regression framework\label{sect-reg}}
We observe i.i.d.\ pairs $X_{i}=(W_{i},Y_{i})$ with values in $\sW\times \R$ drawn from the regression model 
\[
Y_{i}=f\et(W_{i})+\eps_{i}\quad \text{for $i=1,\ldots,n$.}
\]
We assume that the regression function $f\et$ is bounded in supnorm (denoted $\norm{\cdot}_{\infty}$) by some known number $B>0$, that the $W_{i}$ are i.i.d.\ with unknown distribution $P_{W}$ on $\sW$ and the $\eps_{i}$ are i.i.d.\ with unknown density $p$ with respect to the Lebesgue measure $\lambda$ on $\R$.  

We consider a model $\overline \cF$ for $f\et$ which is a set of functions on $\sW$ satisfying the following property.  
%
\begin{ass}\label{hypo-simple-entro2}
For all $f\in\overline \cF$, $\norm{f}_{\infty}\le B$ and there exists a non-increasing function $H$ on $[3,+\infty)$ such that for all $\eps>0$, one can find a subset $\cF_{\eps}\subset \overline \cF$ with cardinality not larger than $\exp\cro{H(\eps)}$ which satisfies $\inf_{g\in\cF_{\eps}}\norm{f-g}_{\infty}\le \eps$ for all $f\in\overline \cF$.
\end{ass}

The density $p$ being unknown, we consider a candidate density $q$ for $p$. Denoting by $q_{\delta}$ the translated density $q_{\delta}(\cdot)=q(\cdot-\delta)$ for $\delta\in\R$, we assume that $q$ is of order $\alpha\in (-1,1]$, i.e.\ satisfies, for some constant $a\ge 1$,
\begin{equation}
a^{-1}\cro{|\delta|^{1+\alpha}\wedge a^{-1}}\le h^{2}(q_{\delta},q)\le a\cro{|\delta|^{1+\alpha}\wedge a^{-1}}\ \ \mbox{for all}\ \delta\in\R.
\label{regul}
\end{equation}
Note that the mapping $\delta\mapsto q_{\delta}$ is one-to-one. 

For $f\in\overline\cF$, we denote by $q_{f}$ the density of $X_{1}$ (with respect to $\mu=P_{W}\otimes \lambda$) when $p=q$ and $f\et=f$ which is given by $q_{f}(w,y)=q(y-f(w))$. The set $\overline S=\{q_{f},\; f\in\overline\cF\}$ is a density model for the true density $s$ of $X_{1}$. A prior $\pi'$ on $\overline\cF$ induces a prior $\pi$ on $\overline S$ by taking the image of $\pi'$ by the mapping $f\mapsto q_{f}$. In turn, the $\rho$-posterior $\pi_{\bsX}$ on $(\overline S,\pi)$ induces a $\rho$-posterior distribution $\pi_{\bsX}'$ on $\overline\cF$ which is the image of $\pi_{\bsX}$ by the reciprocal mapping $q_{f}\mapsto f$. Let us choose as our loss function on $\overline \cF$  
\[
\norm{f\et-f}_{1+\alpha}= \pa{\int_{\sW}\ab{f\et-f}^{1+\alpha}dP_{W}}^{1/(1+\alpha)}\quad \text{for $f\in\overline \cF$.}
\]
%
\begin{prop}\label{prop-reg}
Let Assumption~\ref{hypo-simple-entro2} hold, $q$ satisfy \eref{regul} for some $a\ge 1$ and $\alpha\in (-1,1]$ and let $\eps_{n}$ satisfy
\begin{equation}\label{eq-epsreg}
\eps_{n}\ge 1/\left(2\sqrt{n}\right)\quad \text{and}\quad H\cro{\left(\eps_{n}^{2}/a\right)^{1/(1+\alpha)}}\le \left(4\cdot10^{-6}\right)n\eps_{n}^{2}.
\end{equation}
There exists a prior $\pi'$ on $\overline \cF$ (which only depends on $\eps_{n}$) such that,  whatever the function $f\et$ bounded by $B$, whatever the distribution $P_{W}$, whatever the density $p$ and the positive number $\xi$,  there exists a measurable subset $\Omega_{\xi}$ of $\Omega$ satisfying $\P_{\gs}(\Omega_{\xi})\ge1-e^{-\xi}$ and for all $\omega\in \Omega_{\xi}$,
%
%
\[
\pi_{\!\gX(\omega)}'\pa{\ac{f\in\overline \cF,\; \norm{f\et-f}_{1+\alpha}\le C\overline r^{2/(1+\alpha)}_{n}}}\ge 1-e^{-\xi'}\quad \text{for all $\xi'>0$}
\]
where 
\begin{equation}\label{def-breg}
\overline r_{n}^{2}=h^{2}(p,q)+\inf_{f\in\overline\cF}\norm{f\et-f}_{\infty}^{1+\alpha}+\eps_{n}^{2}+\frac{\xi+\xi'}{n},
\end{equation}
for some constant $C>0$ depending on $a,B$ and $\alpha$ only.
\end{prop}

Let us first emphasize the fact that neither the prior nor the construction of the posterior requires the knowledge of the distribution of the design $P_{W}$ or any assumption about it. The result shows that with probability close to 1 the $\rho$-posterior on $\overline \cF$ concentrates around functions $f\in\overline \cF$ for which  $\norm{f-f\et}_{1+\alpha}$ is of order $\cro{h^{2}(p,q)+\inf_{f\in\overline\cF}\norm{f\et-f}_{\infty}^{1+\alpha}+\eps_{n}^{2}}^{1/(1+\alpha)}$.  The quantity $\eps_{n}^{2/(1+\alpha)}$ corresponds to the concentration rate we get when $p$ is equal to $q$ and $f\et$ belongs to $\overline \cF$ while the terms $\inf_{f\in\overline\cF}\norm{f\et-f}_{\infty}^{1+\alpha}$ and $h^{2}(p,q)$ account for the robustness of the procedure with respect to a misspecification of the class $\overline \cF$ of the regression functions and the noise distribution respectively. The loss and the quantity $\eps_{n}$ depend on the specific features of the chosen density $q$. 

When $H(\eps)=A\eps^{-V}$ for some constants $A,V>0$, then
\[
\eps_{n}^{2/(1+\alpha)}=C'n^{-1/(V+1+\alpha)}
\]
where $C'>0$ depends on $A,a,\alpha$ and $V$ only. We refer to Ibragimov and Has{'}minski{\u\i}~\citeyearpar{MR620321} Chapter VI p.~281 
for  sufficient conditions on the density $q$ to be of order $\alpha$. For illustration, when $q$ is Gaussian, $\alpha=1$, the loss corresponds to the $\L_{2}(P_{W})$-norm and $\eps_{n}^{2/(1+\alpha)}=\eps_{n}$ is of order $n^{-1/(V+2)}$; when $q$ is the uniform density on an interval, $\alpha=0$, the loss corresponds to the $\L_{1}(P_{W})$-norm and $\eps_{n}^{2/(1+\alpha)}=\eps_{n}^{2}$ is of order $n^{-1/(V+1)}$. 

When $\overline \cF$ is a subset of the $\L_{\infty}(P_{W})$-ball with radius $B$ and center 0 of a linear space with dimension $d\ge 1$,  a classical result on the entropy of balls in a finite dimensional linear space implies that Assumption~\ref{hypo-simple-entro2} is satisfied with $H(\eps)=d\log(1+[2B/\eps])$ which leads to an upper bound for $\eps_{n}^{2/(1+\alpha)}$ of order $[d\log(nB/d)/n]^{1/(1+\alpha)}$. Note that this rate is faster than the usual parametric rate $1/\sqrt{n}$ when $\alpha\in (-1,1)$. 

Choosing a specific density $q$ and a single model $\overline \cF$ for $f\et$ is usually not enough for many  applications. It is however possible to mix up several choices of $q$ and $\overline \cF$ by using a hierarchical prior as we shall show in Section~\ref{MM} and by arguing as in BBS, Sections~7.2 and~7.3.

\section{Our main results}\label{M}
Our main results and definitions involve some numerical constants that we list below for further reference.
\begin{equation}
\left\{
\begin{array}{l}  c_{0}=10^{3};\qquad c_{1}=15;\qquad c_{2}=16;\qquad c_{3}=0.62;
\vspace{2mm}\\ c_{4}=3.5\max\{375;\beta^{-1/2}\};\qquad
c_{5}=16\times10^{-3};\qquad c_{6}=7\times10^{4};\vspace{2mm}\\
c_{7}=4.01;\qquad c_{8}=0.365;\qquad c_{9}=c_{8}^{-1}\left[(2c_{6})\vee \beta^{-1}\right];
\vspace{2mm}\\\overline{c}_{n}=1+[(\log 2)/\log(en)];\qquad\gamma=\beta/8.\end{array}\right.
\label{eq-constantes}
\end{equation}
The properties of $\pi_{\!\gX}$ actually depend on two quantities, namely $\eps_{n}^{\overline{S}}(\gs)$ and $\eta_{n}^{\overline{S},\pi}(t)$ for $t\in\overline S$, that we shall now define. The former only depends on $\overline{S}$ via $S$ and also possibly on $\gs$ while the latter depends on the choice of the prior $\pi$ but not on $\gs$. 

\subsection{The quantity $\eps_{n}^{\overline{S}}(\gs)$}\label{sect-eps}
Given $\bsX$ with distribution $\gP_{\gs}$ and $y>0$, we set
\[
\bsZ(\gX,t,t')=\gPsi(\gX,t,t')-\E_{\gs}\left[\gPsi(\gX,t,t')\right],
\]
and
\[
\gw^{\overline{S}}(\gs,y)=\E_{\gs}\left[\sup_{t,t'\in\sB^{S}(\gs,y)}\left|\bsZ(\gX,t,t')\right|\right]\text{ with the convention }\sup_{\varnothing }=0.
\]
%
Note that $\gw^{\overline{S}}(\gs,y)=\gw^{\overline{S}}(\gs,1)$ for $y>1$. We then define  $\eps_{n}^{\overline{S}}(\gs)$ as
\begin{equation}
\eps_{n}^{\overline{S}}(\gs)=\sup\left\{y>0\,\left|\,\gw^{\overline{S}}(\gs,y)>
6c_{0}^{-1}ny^{2}\right.\right\}\bigvee\frac{1}{\sqrt{n}}\;\mbox{ with }\;\sup\varnothing=0.
\label{def-en}
\end{equation}
Since the function $\psi$ is bounded by 1, $\gw^{\overline{S}}(\gs,y)$ is not larger than $2n$ hence $\eps_{n}^{\overline{S}}(\gs)$ is not larger than $(c_{0}/3)^{1/2}$. The quantity $\eps_{n}^{\overline{S}}(\gs)$ measures in some sense the massiveness of the set $S$. In particular, if $S\subset S'$, $\eps_{n}^{\overline{S}}(\gs)\le \eps_{n}^{\overline{S}'}(\gs)$.

\subsection{The quantity $\eta_{n}^{\overline{S},\gpi}(t)$}\label{sect-eta}
%
\begin{df}\label{def-eta}
Let $\gamma=\beta/8$. Given the prior $\gpi$ on the model $\overline{S}$, we define the function $\eta_{n}^{\overline{S},\gpi}$ on $\overline{S}$ by 
\[
\eta_{n}^{\overline{S},\gpi}(t)=\sup\left\{\eta\in(0,1]\,\left|\,\gpi\left(\sB^{\overline{S}}(t,2\eta)\right)>\exp\left[\gamma n\eta^{2}\right]\gpi\left(\sB^{\overline{S}}(t,\eta)\right)\right.\right\},
\]
with the convention $\sup \varnothing=0$.
\end{df}
Note that $\eta_{n}^{\overline{S},\gpi}(t)\le 1$ since $\gpi\!\left(\sB^{\overline{S}}\left(t,r\right)\right)=\pi(\overline S)=1$ for $r\ge1$ and that
\begin{equation}
\gpi\left(\sB^{\overline{S}}(t,2r)\right)\le\exp\left[\gamma nr^{2}\right]\gpi\left(\sB^{\overline{S}}(t,r)\right)\quad\text{for all }r\in\left[\eta_{n}^{\overline{S},\gpi}(t),1\right].
\label{eq-eta2}
\end{equation}
This inequality indeed holds by definition for $r>\eta_{n}^{\overline{S},\gpi}(t)$, which implies by monotonicity that it also holds for $r=\eta_{n}^{\overline{S},\gpi}(t)$.
%
Then, if $0<\eta\le1$ and
\begin{equation}
\gpi\left(\sB^{\overline{S}}(t,2r)\right)\le \exp\left[\gamma nr^{2}\right]\gpi\left(\sB^{\overline{S}}(t,r)\right)\quad\mbox{for all }r\in[\eta,1],
\label{Eq-eta}
\end{equation}
it follows from (\ref{eq-eta2}) that $\eta_{n}^{\overline{S},\gpi}(t)\le \eta$. 

The quantity $\eta_{n}^{\overline{S},\gpi}(t)$ corresponds to some critical radius over which the $\gpi$-probability of balls centred at $t$ does not increase too quickly. In particular, if the prior puts enough mass on a small neighbourhood of $t$, $\eta_{n}^{\overline{S},\gpi}(t)$ is small. Indeed, since $\gpi\left(\sB^{\overline{S}}(t,2r)\right)\le1$ for all $r>0$, the inequality 
\begin{equation}\label{eq-minmass}
\gpi\!\left(\sB^{\overline{S}}(t,\eta)\right)\ge\exp\left[-\gamma n\eta^{2}\right]\quad \text{for some $\eta\in(0,1]$}
\end{equation}
implies that, for $1\ge r\ge \eta$,
\[
\gpi\!\left(\sB^{\overline{S}}(t,r)\right)\ge\exp\left[-\gamma nr^{2}\right]\ge\gpi\left(\sB^{\overline{S}}(t,2r)\right)\exp\left[-\gamma nr^{2}\right],
\]
hence that $\eta_{n}^{\overline{S},\gpi}(t)\le \eta$. However, the upper bounds on $\eta_{n}^{\overline{S},\gpi}(t)$ that are derived from \eref{eq-minmass} are usually less accurate than those derived from \eref{Eq-eta}.

\subsection{Our main theorem}\label{MThm}
The concentration properties of the $\rho$-posterior distribution $\pi_{\!\gX}$ are given by the following theorem.
\begin{thm}\label{thm-main}
Let Assumption~\ref{hypo-mes} be satisfied. Then, whatever the true density $\gs$ of $\gX$ and $\xi>0$, there exists a measurable subset $\Omega_{\xi}$ of $\Omega$ with $\P_{\gs}(\Omega_{\xi})\ge1-e^{-\xi}$ such that
\begin{equation}
\gpi_{\!\gX(\omega)}\pa{\sB^{\overline S}(\gs,r)}
\ge1-e^{-\xi'}\quad \text{for all }\omega\in \Omega_{\xi},\;\xi'>0\,\text{ and }\,r\ge\overline r_{n}
\label{Eq-main}
\end{equation}
with 
\begin{equation}\label{def-rb}
\overline r_{n}=\inf_{t\in\overline{S}}\left[c_{1}h(\gs,t)+c_{2}\eta_{n}^{\overline{S},\gpi}(t)\right]+c_{3}\eps_{n}^{\overline{S}}(\gs)+c_{4}\sqrt{\frac{\xi+\xi'+2.61}{n}}.
\end{equation}
The constants $c_{j}$, $1\le j\le4$ are given in (\ref{eq-constantes}) and actually universal as soon as  $\beta\ge 7.2\times10^{-6}$. 
%
\end{thm}
In the favorable situation where the observations $X_{1},\ldots,X_{n}$ are truly i.i.d.\ so that $\gs=(s,\ldots,s)$, \eref{Eq-main} can be reformulated equivalently as 
\[
\pi_{\!\gX(\omega)}\left(\sB^{\overline{S}}(s,r)\right)\ge1-e^{-\xi'}\quad \text{for all }\omega\in \Omega_{\xi},\;\xi'>0\,\text{ and }\,r\ge\overline r_{n}
\]
with 
\begin{equation}\label{eq-rb2}
\overline r_{n}=\inf_{t\in \overline{S}}\left[c_{1}h(s,t)+c_{2}\eta_{n}^{\overline{S},\gpi}(t)\right]+c_{3}\eps_{n}^{\overline{S}}(\gs)+c_{4}\sqrt{\frac{\xi+\xi'+2.61}{n}},
\end{equation}
which measures the concentration of the $\rho$-posterior distribution $\pi_{\!\gX}$ around the true density $s$ of our i.i.d.\ observations $\etc{X}$. It involves three main terms: $h(s,t)$, $\eta_{n}^{\overline{S},\gpi}(t)$ and $\eps_{n}^{\overline{S}}(\gs)$. For many models $\overline S$ of interest, as we shall see in Section~\ref{Sect-UPBeps}, it is possible to show an upper bound of the form
\begin{equation}\label{eq-borneEPS}
\eps_{n}^{\overline{S}}(\gs)\le v_{n}(\overline S)\quad \text{for all $\gs\in\sL^{n}$},
\end{equation}
where $v_{n}(\overline S)$ is of the order of the minimax rate of estimation on $\overline S$ (up to possible logarithmic factors), i.e. the rate one would expect by using a frequentist or a classical Bayes estimator provided that the true density $s$ does belong to the model $\overline S$ and the prior distribution puts enough mass around $s$. Under \eref{eq-borneEPS}, if $s$ does belong to $\overline S$, we deduce from~\eref{eq-rb2} that
\begin{equation}\label{eq-rb2b}
\overline r_{n}\le (c_{2}+ c_{3})\max\ac{\eta_{n}^{\overline{S},\gpi}(s);v_{n}(\overline S)}+c_{4}{\sqrt{\xi+\xi'+2.61\over n}}.
\end{equation}
In many cases  the quantity $\eta_{n}^{\overline{S},\gpi}(s)$ turns out to be of the same order or smaller than $v_{n}(\overline S)$ provided that the prior $\pi$ puts enough mass around $s$. In~\eref{eq-rb2}, the term $\inf_{t\in\overline{S}}\left[c_{1}h(s,t)+c_{2}\eta_{n}^{\overline{S},\gpi}(t)\right]$ expresses some robustness with respect to this ideal situation: if $\pi$ puts too little mass around $s$, possibly zero mass when $s$ does not belong to the model, but if $s$ is close enough to some point $t\in\overline{S}$ around which $\pi$ puts enough mass, the previous situation does not deteriorate too much. When $s$ does not belong to the model, one may think of $t$ as a best approximation point $\overline t$ of $s$ in $\overline{S}$ when $\eta_{n}^{\overline{S},\gpi}(\overline t)$ is not too large or alternatively to some point $t$ that may be slightly further away from $s$ but for which $\eta_{n}^{\overline{S},\gpi}(t)$ is smaller than $\eta_{n}^{\overline{S},\gpi}(\overline t)$ in order to minimize the function $t'\mapsto c_{1}h(s,t')+c_{2}\eta_{n}^{\overline{S},\gpi}(t')$ over $\overline{S}$.

If $X_{1},\ldots,X_{n}$ are not truly i.i.d.\ but are independent and close to being drawn from a common density $s_{0}\in\overline S$, i.e.\ $\gs=(s_{1},\ldots,s_{n})$ with $h(s_{i},s_{0})\le \eps$  for some small $\eps>0$ and all $i\in \{1,\ldots,n\}$, then $h(\gs,s_{0})\le\varepsilon$ and $\sB^{\overline{S}}(\gs, r)\subset\sB^{\overline{S}}(s_{0}, \varepsilon+r)$. We therefore deduce from \eref{Eq-main} and \eref{def-rb} with $t=s_{0}$
that, if \eref{eq-borneEPS} holds, the posterior distribution concentrates on Hellinger balls around $s_{0}$ with radius not larger than
\[
\eps+\overline r_{n}\le (1+c_{1})\eps+(c_{2}+ c_{3})\max\ac{{\eta_{n}^{\overline{S},\gpi}(s_{0})};v_{n}(\overline S)}+c_{4}{\sqrt{\xi+\xi'+2.61\over n}},
\]
which is similar to \eref{eq-rb2b} with $s=s_{0}$ except for the additional term $(1+c_{1})\eps$ which expresses the fact that our procedure is robust with respect to a possible departure from the assumption of equidistribution.

\section{Upper bounds for $\eps_{n}^{\overline{S}}(\gs)$}\label{Sect-UPBeps}

\subsection{Case of a finite set $\overline{S}$}
There are many situations for which it is natural, in view of the robustness properties of the
$\rho$-Bayes posterior, to choose for $\overline{S}$ a finite set, in which
case we take $S=\overline{S}$ and the quantity $\eps_{n}^{\overline{S}}(\gs)$ can then be bounded from above as follows.
%
\begin{prop}\label{prop-fc}
If  $\overline{S}$ is a finite set and $S=\overline{S}$,
\[
\eps_{n}^{\overline{S}}(\gs)<\left(\sqrt{c_{0}/3}\right)\min\left\{\sqrt{\sqrt{2}c_{0}n^{-1}\log\left(2|\overline{S}|^{2}\right)},1\right\}.
\]
\end{prop}
An important example of such a finite set $\overline S$ is that of an $\eps$-net for a totally bounded set. We recall that, if $\widetilde{S}$ is a subset of some pseudo-metric space $M$ endowed with a pseudo-distance $d$ and $\eps>0$, a subset $S_{\varepsilon}$ of $M$ is an $\eps$-net for $\widetilde S$ if, for all $t\in\widetilde S$, one can find $t'\in S_{\varepsilon}$ such that $d(t,t')\le \eps$. When $\widetilde{S}$ is totally bounded one can find a finite $\eps$-net for $\widetilde S$ whatever $\varepsilon>0$. This applies in particular to totally bounded subsets $\widetilde{S}$ of $(\sL,h)$.
The smallest possible size of such nets depends on the metric properties of $(\widetilde S,h)$ and the following notion of metric dimension, as introduced in Birg\'e~\citeyearpar{MR2219712} (Definition~6 p.~293) turns out to be a central tool. 
\begin{df}\label{def-MD}
Let $D$ be a function from $(0,1]$ to $[3/4,+\infty)$ which is right-continuous. A model $\widetilde S\subset\sL$ admits a metric dimension bounded by $D$ if, for all $\eps\in(0,1]$, there exists an $\eps$-net $S_{\varepsilon}$ for $\widetilde S$ such that, for any $s$ in $\sL$,
\begin{equation}\label{def-mddim1}
\ab{\ac{t\in S_{\varepsilon},\ h(s,t)\le r}}\le\exp\cro{D(\eps)(r/\varepsilon)^{2}}\quad\mbox{for all }r\ge 2\eps.
\end{equation}
%
\end{df}
Note that this implies that $S_{\varepsilon}$ is finite and that one can always take $D(1)=3/4$ since $h$ is bounded by 1.
The following result shows how a bound $D$ for the metric dimension can be used to bound $\eps_{n}^{\overline{S}}(\gs)$ for a model $\overline{S}$ which is an $\varepsilon$-net for $\widetilde{S}$ which satisfies (\ref{def-mddim1}).
%
\begin{prop}\label{maj-epss}
Let $\widetilde S$ be a totally bounded subset of $(\sL,h)$ with metric dimension bounded by $D$ and let $\eps$ be a positive number satisfying
\begin{equation}\label{eq-eps0}
\eps\ge1/\left(2\sqrt{n}\right)\qquad\mbox{and}\qquad D(\eps)\le n(\eps/c_{0})^{2}.
\end{equation}
If $S_{\varepsilon}$ is an $\eps$-net for $\widetilde S$ satisfying~\eref{def-mddim1} and $S=\overline{S}=S_{\varepsilon}$, then $\eps_{n}^{\overline{S}}(\gs)\le 2\eps$ whatever $\gs\in\sL^{n}$.
\end{prop}
Starting from a classical statistical model $\widetilde S$ with metric dimension bounded by $D$ we may therefore replace it by a suitable $\varepsilon$-net $\overline{S}$ in order to build a $\rho$-Bayes posterior based on some prior distribution on $\overline{S}$. The robustness of the procedure, as shown by Theorem~\ref{thm-main}, implies that the replacement of $\widetilde S$ by $\overline{S}$ will only entail an additional bias term of order $\varepsilon$.w

%
\subsection{Weak VC-major classes}
%
\begin{df}\label{def-wvc}
A class of real-valued functions $\sF$ on a set $\sX$ is said to be weak VC-major  with dimension not larger than $d\in\N$ if, for all $u\in\R$, the class of sets 
\[
\sC_{u}(\sF)=\left\{\{f>u\},\ f\in\sF\st\right\}
\]
is VC on $\sX$, with VC-dimension not larger than $d$. The weak VC-major dimension of $\sF$ is the smallest such integer $d$. 
\end{df}
For details on the definition and properties of VC-classes, we refer to van der Vaart and Wellner~\citeyearpar{MR1385671} and for weak VC-major classes to Baraud~\citeyearpar{Bar2016}. One major point about weak VC-major classes is the fact that if $\sF$ is weak VC-major with dimension not larger than $d\in\N$, the same holds for any subset $\sF'$ of $\sF$.
%
\begin{prop}\label{prop-VC}
Let $\sF$ be the class of  functions on $\sX$ given by
\begin{equation}
\sF=\left\{\psi\left(\sqrt{t'\over t}\right),\ (t,t')\in \overline{S}^{2}\right\}.
\label{def-Fd}
\end{equation}
If it is weak VC-major with dimension not larger than $d\ge 1$, then, whatever the density $\gs\in\sL^{n}$,
\begin{equation}
\eps_{n}^{\overline{S}}(\gs)\le\frac{11c_{0}}{4}\sqrt{\frac{\overline{c}_{n}(d\wedge n)}{n}}\left[\log\pa{en\over d\wedge n}\right]^{3/2}\;\mbox{ with }\;\overline{c}_{n}=1+\frac{\log 2}{\log(en)}.
\label{cas-VC}
\end{equation}
\end{prop}
%

\subsection{Examples}\label{sect-exVC}
We provide below examples of parametric models indexed by some subset $\bs{\Theta}$ of a Euclidean space putting on our models the $\sigma$-algebra induced by the Borel one on $\bs{\Theta}$. Since our results are in terms of VC-dimensions, they hold for all submodels of those described below.

%
\begin{prop}\label{prop-exp0}
Let $(g_{j})_{1\le j\le J}$ with $J\ge 1$ be real-valued functions on a set $\sX$. 

a) If the elements $t$ of the model $\overline{S}$ are of the form
\begin{equation}\label{eq-modelt0}
t(x)=\exp\pa{\theta_{0}+\sum_{j=1}^{J}\theta_{j}g_{j}(x)}\quad\text{for all }x\in\sX
\end{equation}
with $\theta_{0},\ldots,\theta_{J}\in\R$, then $\sF$ defined by~\eref{def-Fd} is weak VC-major with dimension not larger than $d=J+2$. 

b) Let $\sJ=(I_{i})_{i=1,\ldots,k}$ ($k\ge2$) be a partition of $\sX$. If the elements $t$ of the model $\overline{S}$ are of the form
\begin{equation}\label{eq-modelt0m}
t(x)=\sum_{i=1}^{k}\exp\pa{\sum_{j=1}^{J}\theta_{i,j}g_{j}(x)}\1_{I_{i}}(x)\quad\text{for all }x\in\sX
\end{equation}
with $\theta_{i,j}\in\R$ 
for $i=1,\ldots,k$ and $j=1,\dots J$, then $\sF$ defined by~\eref{def-Fd} is weak VC-major with dimension not larger than $d=k(J+2)$.
\end{prop}
If $\sX$ is an interval of $\R$ (possibly $\R$ itself), the second part of the proposition extends to densities based on variable partitions of $\sX$.
%
\begin{prop}\label{prop-exp}
Let $(g_{j})_{1\le j\le J}$ ($J\ge1$) be real-valued functions on an interval $\gI$ of $\R$. Let the elements $t$ of the model $\overline{S}$ be of the form 
\begin{equation}
t(x)=\sum_{I\in\sJ(t)}\exp\pa{\sum_{j=1}^{J}\theta_{I,j}g_{j}(x)}\1_{I}(x)\quad\text{for all }x\in\sX,
\label{eq-modelt}
\end{equation}
where $\sJ(t)$ is a partition of\, $\gI$ which may depend on $t$, into at most $k$ intervals ($k\ge2$), and $(\theta_{I,j})_{j=1,\ldots,J}\in\R^{J}$ for all $I\in\sJ(t)$. Then $\sF$ defined by~\eref{def-Fd} is weak VC-major with dimension not larger than $d=\lceil18.8k(J+2)\rceil$, which means the smallest integer $j\ge18.8k(J+2)$.
\end{prop}
If, for instance, $\overline{S}$ consists of all positive histograms defined on a bounded interval $\gI$ of $\R$ with at most $k$ pieces, then one may take $J=1$, $g_{1}\equiv 1$ and Proposition~\ref{prop-exp} implies that $\sF$ is weak VC-major with dimension not larger than $56.4k$. 

Note that the densities $t$ given by (\ref{eq-modelt0m}) can be viewed as elements of a piecewise exponential family. Let us indeed consider a classical exponential family on the set  $\sX$ with densities (with respect to $\mu$) of the form
\begin{equation}
t_{\bs{\theta}}(x)=\exp\cro{\sum_{j=1}^{J}\theta_{j}T_{j}(x)-A(\bs{\theta})}\quad\text{for all }x\in\sX
\label{eq-expt}
\end{equation}
with $\bs{\theta}=(\theta_{1},\cdots,\theta_{J})\in\bs{\Theta}\subset\R^{J}$. It leads  to a model $\overline S$ of the form (\ref{eq-modelt0}) with $g_{j}=T_{j}$ for $1\le j\le J$ and $\theta_{0}=-A(\bs{\theta})$. In particular, $\sF$ is weak VC-major with dimension not larger than $d=J+2$ and we deduce from Proposition~\ref{prop-exp0} that
\begin{equation}\label{eps-EM}
\eps_{n}^{\overline S}(\gs)\le(11/4)c_{0}\sqrt{\frac{\overline{c}_{n}(J+2)}{n}}\log^{3/2}\pa{en}\qquad \mbox{for all}\ \gs\in \sL^{n}.
\end{equation}
If all elements of $\overline{S}$ are piecewise of the form (\ref{eq-expt}) on some partition 
$\sJ=(I_{i})_{i=1,\ldots,k}$ of $\sX$ into $k$ subsets, $\sF$ is then weak VC-major with dimension not larger than $k(J+3)$ and for some positive universal constant $c'$, 
\begin{equation}\label{eps-EMpm}
\eps_{n}^{\overline S}(\gs)\le c'\sqrt{\frac{kJ}{n}}\log^{3/2}\pa{en}\qquad \mbox{for all}\ \gs\in \sL^{n}.
\end{equation}
When $\sX=[0,1]$, one illustration of case $b)$ is provided by $\Theta_{i}=[-M,M]^{J}$ for $i\in\{1,\ldots,k\}$ and $T_{j}(x)=x^{j-1}$ for $j\in\{1,\ldots,J\}$. We may then apply Proposition~\ref{prop-exp0} and the performance of the $\rho$-posterior distribution will depend on the approximation properties of the family of piecewise polynomials on the partition $\sJ$ with respect to the logarithm of the true density. Numerous results about such approximations can be found in DeVore and Lorentz~\citeyearpar{DeVore}.

\section{Upper bounds for $\eta_{n}^{\overline{S},\pi}(t)$}\label{Sect-UPBeta}

\subsection{Uniform distribution on an $\varepsilon$-net}\label{sect-eta1}
We consider here the situation where $\widetilde{S}$ is a totally bounded subset of $(\sL,h)$ with metric dimension bounded by $D$, $\eps\in (0,1]$, $\overline{S}=S_{\varepsilon}$ is an $\eps$-net for $\widetilde{S}$ which satisfies~\eref{def-mddim1} and we choose $\pi$ as the uniform distribution on $\overline{S}$. 
\begin{prop}\label{def-etaRes}
If $D(\eps)\le (\gamma/4)n\eps^{2}$, then $\eta_{n}^{\overline{S},\pi}(t)\le \eps$ for all $t\in\overline{S}$.
\end{prop}
\begin{proof}
Let $t\in\overline{S}$. For all $r>0$, $\gpi\left(\sB^{\overline{S}}(t,r)\right)\ge \gpi(\{t\})=\left|\overline{S}\right|^{-1}$. Using~\eref{def-mddim1} we derive that
\[
{\gpi\left(\sB^{\overline{S}}(t,2 r)\right)\over \gpi\left(\sB^{\overline{S}}(t,r)\right)}\le \left|\overline{S}\right|\gpi\left(\sB^{\overline{S}}(t,2 r)\right)=\left|\sB^{\overline{S}}(t,2 r)\right|\le \exp\cro{4D(\eps)\left({r\over \eps}\right)^{2}}
\]  
for all $r\ge \eps$. The conclusion follows from the fact that $4D(\eps)/\eps^{2}\le \gamma n$.
\end{proof}

\subsection{Parametric models indexed by a bounded subset of $\R^{d}$}\label{sect-eta2}
In this section we consider the situation where $\overline{S}$ is a parametric model $\{t_{\gtheta},\ \gtheta\in \bs{\Theta}\}$ indexed by a measurable (with respect to the Borel $\sigma$-algebra) bounded subset $\bs{\Theta}\subset \R^{d}$ and we assume that the prior $\gpi$ is the image by the mapping $\gtheta\mapsto t_{\gtheta}$ of some probability $\nu$ on $\bs{\Theta}$. Besides, we assume that the Hellinger distance on $\overline{S}$ is related on $\bs{\Theta}$ to some norm $\ab{\cdot}_{*}$ on $\R^{d}$ in the following way:
\begin{equation}
\underline a\ab{\gtheta-\gtheta'}_{*}^{\alpha}\le h(t_{\gtheta},t_{\gtheta'})\le \overline a\ab{\gtheta-\gtheta'}_{*}^{\alpha}\quad\mbox{for all }\gtheta,\gtheta'\in \bs{\Theta},
\label{eq-conEh}
\end{equation}
where $\underline a,\overline a$ and $\alpha$ are positive numbers. Since $h$ is bounded by $1$,~\eref{eq-conEh} implies that $\bs{\Theta}$ is necessarily bounded. Let us denote by $\cB_{*}(\gtheta,r)$ the closed ball (with respect to the norm $\ab{\cdot}_{*}$) of center $\gtheta$ and radius $r$ in $\R^{d}$.

%
\begin{prop}\label{casconv}
Assume that $\bs{\Theta}$ is measurable and bounded in $\R^{d}$, that \eref{eq-conEh} holds and that $\nu$ satisfies
\begin{equation}
\nu(\cB_{*}(\gtheta,2x))\le\kappa_{\gtheta}(x)\nu(\cB_{*}(\gtheta,x))\quad\mbox{for all }\gtheta\in\bs{\Theta}\mbox{ and }x>0,
\label{eq-hypnu}
\end{equation}
where $\kappa_{\gtheta}(x)$ denotes some positive nonincreasing function on $\R_{+}$.
Then, for all $\gtheta\in\bs{\Theta}$,
\begin{equation}
\eta_{n}^{\overline{S},\gpi}(t_{\gtheta})\le\inf\left\{\eta>0\,\left|\,\eta^{2}\ge\frac{\log\left(\kappa_{\gtheta}\pa{[\eta/\overline a]^{1/\alpha}}\right)}{\gamma n}\left[\frac{\log(2\overline{a}/\underline{a})}{\alpha\log 2}+1\right]\right.\right\}.
\label{eq-eta-param0}
\end{equation}
If $\kappa_{\gtheta}(x)\equiv\kappa_{0}$ for all $\gtheta\in\bs{\Theta}$ and $x>0$, then 
\begin{equation}
\eta_{n}^{\overline{S},\gpi}(t_{\gtheta})\le\sqrt{\frac{\log\kappa_{0}}{\gamma n}\left[\frac{\log(2\overline{a}/\underline{a})}{\alpha\log 2}+1\right]}\quad\mbox{for all }\,\gtheta\in\bs{\Theta}.
\label{eq-eta-param9}
\end{equation}
In particular, if $\bs{\Theta}$ is convex and $\nu$ admits a density $g$ with respect to the Lebesgue measure $\lambda$ on $\R^{d}$ which satisfies 
\begin{equation}
\underline b\le g(\gtheta)\le \overline b\quad\mbox{for $\lambda$-almost all }\gtheta\in \bs{\Theta}
\quad\text{with }0<\underline b\le \overline b,
\label{eq-densborn}
\end{equation}
then (\ref{eq-hypnu}) holds with $\kappa_{\gtheta}(x)\equiv\kappa_{0}=2^{d}(\overline b/\underline b)$, hence, for all $t\in\overline{S}$,
\begin{equation}
\eta_{n}^{\overline{S},\gpi}(t)\le c\sqrt{\frac{d}{n}}\quad\mbox{with}\quad c^{2}=\frac{\log\left(2\left[\overline b/\underline b\right]^{1/d}\right)}{\gamma}\left[\frac{\log(2\overline{a}/\underline{a})}{\alpha\log 2}+1\right].
\label{eq-eta-param}
\end{equation}
\end{prop}
\subsection{Example}
Let us consider, in the density model with $n$ i.i.d.\ observations on $\R$, the following translation family $t_{\theta}(x)=t(x-\theta)$ where $t$ is the density of the Gamma$(2\alpha,1)$ distribution, namely
\[
t(x)=c(\alpha)x^{2\alpha-1}e^{-x}\1_{x\ge 0}\quad \text{with $0<\alpha<1$}
\] 
and $\theta$ belongs to the interval $\Theta=[-1,1]$. It is known from Example~1.3 p.287 of Ibragimov and Has{'}minski{\u\i}~\citeyearpar{MR620321} that, in this situation, (\ref{eq-conEh}) holds for $\ab{\cdot}_{*}$ the absolute value and $\underline a,\overline a$ depending on $\alpha$. Let us now derive upper bounds for $\eta_{n}^{\overline{S},\gpi}(t_{\gtheta})$ when $\nu$ has a density $g$ with respect to the Lebesgue measure.\vspace{2mm}\\
--- If $\nu$ is uniform on $\Theta$, then $\overline b=\underline b$ and \eref{eq-eta-param} is satisfied for some constant $c$ depending on $\alpha$ and $\gamma$ only.\vspace{2mm}\\
--- If $g(z)=(\xi/2)|z|^{\xi-1}\1_{[-1,1]}(z)$ with $0<\xi<1$, in order to compute $\kappa_{0}$ one has to compare the $\nu$-measures of the intervals $I_{1}=[(\theta-x)\vee-1,(\theta+x)\wedge1]$ and $I_{2}=[(\theta-2x)\vee-1,(\theta+2x)\wedge1]$ for $x>0$.
%
\begin{prop}\label{prop-nu1/nu2}
If in this example $g(z)=(\xi/2)|z|^{\xi-1}\1_{[-1,1]}(z)$, (\ref{eq-eta-param9}) holds since
\[
\nu(I_{2})\le\kappa_{0}\nu(I_{1})\quad \text{with}\quad \kappa_{0}=2^{1+\xi}\left(2^{\xi}-1\right)^{-1}.
\] 
\end{prop}
One should therefore note that if (\ref{eq-densborn}) is sufficient for $\kappa_{\gtheta}(r)$ to be constant, it is by no means necessary.\vspace{2mm}\\
--- Let us now set $g(z)=c_{\delta}^{-1}\exp\left[-\left(2|z|^{\delta}\right)^{-1}\right]\1_{[-1,1]}(z)$ for some $\delta>0$, which means that the prior puts very little mass around the point $\theta=0$. Then
%
\begin{prop}\label{prop-cas3}
In this example $\eta_{n}^{\overline{S},\gpi}(t_{0})\le Kn^{-\alpha/[2\alpha+\delta]}$,
for some $K$ depending on $\alpha,\delta,\overline{a},\underline{a}$ and $\gamma$.
\end{prop}
It is not difficult to check that in this situation the family $\sF$ defined by~\eref{def-Fd} consists of elements $f$ for which either $f$ or $-f$ is unimodal. In particular, for $f\in\sF$, the levels sets $\{f>u\}$ with $u\in\R$ consist of a union of at most two disjoint intervals. It follows from Lemma~1 of Baraud and Birg\'e~\citeyearpar{MR3565484} that $\sF$ is then weak-VC major with dimension not larger than 4 so that, as a consequence of Proposition~\ref{prop-VC}, $\eps_{n}^{\overline{S}}(\gs)\le C(\log n)^{3/2}$ for some universal constant $C>0$ and all densities $\gs\in\sL^{n}$. Applying Theorem~\ref{thm-main} when the true parameter $\theta$ is $0$ leads to a bound for \eref{eq-rb2} of the form
\[
\overline r_{n}\le K\left[n^{-\alpha/(2\alpha+\delta)}+\sqrt{\frac{\log^{3}n}{n}}+
\sqrt{\frac{\xi+\xi'+2.61}{n}}\right],
\]
which is of the order of $n^{-\alpha/(2\alpha+\delta)}$ and  clearly depends on the relative values of $\alpha$ and $\delta$. In particular, if $\alpha=1/2$, which corresponds to the exponential density, we get a bound for $\overline{r}_{n}$ of order $n^{-1/[2(1+\delta)]}$.

\section{Connexion with classical Bayes estimators}\label{RB}
Throughout this section we assume that the data $X_{1},\ldots,X_{n}$ are i.i.d.\ with density $s$ on the measured space $(\sX,\sA,\mu)$.

We consider a parametric set of real nonnegative functions $\{t_{\gtheta},\,\gtheta\in\bs{\Theta}\}$ satisfying $\int_{\sX}t_{\gtheta}(x)\,d\mu(x)=1$, indexed by some subset $\bs{\Theta}$ of $\R^{d}$ and such that the mapping $\gtheta\mapsto P_\gtheta=t_\gtheta\cdot\mu$ is one-to-one so that our statistical model be identifiable. Our model for $s$ is $\overline{S}=\{t_{\gtheta},\,\gtheta\in\bs{\Theta}\}$. We set $\norm{t}_{\infty}=\sup_{x\in\sX}\ab{t(x)}$ for any function $t$ on $\sX$. Since the mapping $\gtheta\mapsto t_\gtheta$ is one-to-one, the Hellinger distance can be transfered to $\bs{\Theta}$ and we shall write $h(\gtheta,\gtheta')$ for $h(t_\gtheta,t_{\gtheta'})=h(P_{\gtheta},P_{\gtheta'})$.

We consider on $(\overline{S},h)$ the Borel $\sigma$-algebra $\sS$ and, given a prior $\pi$ on $(\overline{S},\sS)$, we consider both the usual Bayes posterior distribution $\pi^{L}_{\!\gX}$ and our $\rho$-posterior distribution $\pi_{\!\gX}$ given by~\eref{Eq-psB} with $\beta=4$. A natural question is whether these two distributions are similar or not, at least asymptotically when $n$ tends to $+\infty$. This question is suggested by the fact, proven in Section~5.1 of BBS, that, under suitable regularity assumptions, the maximum likelihood estimator is a $\rho$-estimator, at least asymptotically.

In order to show that the two distributions $\pi^{L}_{\!\gX}$ and $\pi_{\!\gX}$ are asymptotically close we shall introduce the following assumptions that are certainly not minimal but at least lead to simpler proofs.
\begin{ass}\label{A-MLE}\mbox{}\vspace{-3mm}
\begin{listi}
	\item\label{H4-mi} The function $(x,\gtheta)\mapsto t_{\gtheta}(x)$ is measurable from $\left(\sX\times \bs{\Theta},\sA\otimes\sG(\bs{\Theta})\st\right)$ to $(\R_{+},\sR)$ where $\sG(\bs{\Theta})$ and $\sR$ denote respectively the Borel $\sigma$-algebras on $\bs{\Theta}\subset \R^{d}$ and $\R_{+}$. 
	\item\label{H4-i} The parameter set $\bs{\Theta}$ is a compact and convex subset of $\bs{\Theta}'\subset\R^d$ and the true density $s=t_{\gvartheta}$ belongs to $\overline{S}$.
	\item\label{H4-ii} There exists a positive function $A_{2}$ on $\bs{\Theta}$ such that the following relationship between the Hellinger and Euclidean distances holds: 
\begin{equation}
\frac{A_2(\gtheta')}{2}\ab{\overline\gtheta-\gtheta}\le h\left(\frac{t_{\overline\gtheta}+t_{\gtheta'}}{2},\frac{t_{\gtheta}+t_{\gtheta'}}{2}\right)\quad\mbox{for all }\overline\gtheta,\,\gtheta,\,\gtheta'\in \bs{\Theta}.
\label{eq-A4ii}
\end{equation}
	%
	\item\label{H4-iii} Whatever $\gtheta\in \bs{\Theta}$, the density $t_{\gtheta}$ is positive on $\sX$ and there exists a constant $A_1$ such that
\[
\norm{\sqrt{t_{\gtheta}\over t_{\gtheta'}}-\sqrt{t_{\overline\gtheta}\over t_{\gtheta'}}}_{\infty}\le  
A_1\left|\overline\gtheta-\gtheta\right|\quad\mbox{for all }\gtheta,\,\overline\gtheta
\mbox{ and }\gtheta'\in\bs{\Theta}.
\]
\end{listi}
\end{ass}
These assumptions imply that $(\overline S,h)$ is a metric space and that the function $t\mapsto t(x)$ from $(\overline S,h)$ to $(0,+\infty)$ is continuous whatever $x\in\sX$. Furthermore, Assumption~\ref{A-MLE}-$\ref{H4-iii}$ implies that the Hellinger distance on $\bs{\Theta}$ is controlled by the Euclidean one in the following way:
\[
h^{2}(\gtheta,\overline{\gtheta})=\frac{1}{2}\int\pa{\sqrt{t_{\gtheta}\over t_{\gtheta'}}-\sqrt{t_{\overline\gtheta}\over t_{\gtheta'}}}^{2}t_{\gtheta'}\,d\mu\le\frac{A_{1}^{2}}{2}\left|\overline\gtheta-\gtheta\right|^{2}.
\]
Since the concavity of the square root implies that
\begin{equation}
h\left(\frac{t_{\overline\gtheta}+t_{\gtheta'}}{2},\frac{t_{\gtheta}+t_{\gtheta'}}{2}\right)\le\frac{1}{2}
h\left(\overline\gtheta,\gtheta\right),
\label{eq-hconc}
\end{equation}
we derive from (\ref{eq-A4ii}) with $\gtheta'=\gvartheta$ that $h\left(\overline\gtheta,\gtheta\right)\ge A_{2}(\gvartheta)\ab{\overline\gtheta-\gtheta}$. The Hellinger and Euclidean distances are therefore equivalent on $\bs{\Theta}$: 
\begin{equation}
A_2\ab{\overline\gtheta-\gtheta}\le h\left(\overline\gtheta,\gtheta\right)=h(t_{\overline \gtheta},t_{\gtheta})\le A_3\ab{\overline\gtheta-\gtheta}
\quad\mbox{for all }\overline\gtheta,\,\gtheta\in \bs{\Theta},
\label{metric-h}
\end{equation}
with $A_2=A_{2}(\gvartheta)<A_3=A_{1}/\sqrt{2}$. 

In particular, the mapping $t_{\gtheta}\mapsto \gtheta$ is continuous from $(\overline S,h)$ to $(\bs{\Theta},\ab{\cdot})$, hence measurable from $(\overline S,\sS)$ to $(\bs{\Theta},\sG(\bs{\Theta}))$ and so are $f:(x,t_{\gtheta})\mapsto (x,\gtheta)$ from $(\sX\times \overline S,\sA\otimes \sS)$ to $(\sX\times \bs{\Theta},\sA\otimes\sG(\bs{\Theta}))$ and $(x,t_{\gtheta})\mapsto t_{\gtheta}(x)$ from $(\sX\times \overline S,\sA\otimes \sS)$ to $(\R_{+},\sR)$ as the composition of $f$ with $(x,\gtheta)\mapsto t_{\gtheta}(x)$ which is measurable under Assumption~\ref{A-MLE}-\ref{H4-i}. Consequently Assumption~\ref{hypo-mes}-\ref{Hi} is satisfied and so is~\eref{eq-t-t_k} if we take for $S$ the image by the mapping $\gtheta\mapsto t_{\gtheta}$ of a countable and dense subset of $(\bs{\Theta},\ab{\cdot})$ and use the fact that for all $x\in \sX$, the function $t\mapsto t(x)$ is continuous and positive on $(\overline S,h)$.

We deduce from (\ref{eq-A4ii}) and (\ref{eq-hconc}) that
\begin{equation}
\frac{A_2}{2}\ab{\overline\gtheta-\gtheta}\le h\left(\frac{t_{\overline\gtheta}+s}{2},\frac{t_{\gtheta}+s}{2}\right)\le \frac{A_3}{2}\ab{\overline\gtheta-\gtheta} \quad\mbox{for all }\overline\gtheta,\,\gtheta\in \bs{\Theta},
\label{metric-h'}
\end{equation}
and since $\psi$ is a Lipschitz function with Lipschitz constant 2, Assumption~\ref{A-MLE}-$\ref{H4-iii}$ implies that
\[
\norm{\psi\left(\sqrt{t_{\gtheta}\over t_{\gtheta'}}\right)-\psi\left(\sqrt{t_{\overline\gtheta}\over t_{\gtheta'}}\right)}_{\infty}\le 2A_1\left|\overline\gtheta-\gtheta\right|\quad\mbox{for all }\gtheta,\,\overline\gtheta\mbox{ and }\gtheta'\in\bs{\Theta}.
\]
If $\bs{\Theta}$ is a compact subset of an open set $\bs{\Theta}'$ and the the parametric family $\{t_{\gtheta},\,\gtheta\in\bs{\Theta}'\}$ is regular with invertible Fisher Information matrix, the same holds for the family $\{[t_{\gtheta}+t_{\gtheta'}]/2,\,\gtheta\in\bs{\Theta}'\}$ for each given $\gtheta'$ in $\bs{\Theta}$, which implies that Assumption~\ref{A-MLE}-$\ref{H4-ii}$ holds.

%
\begin{ass}\label{A-prior}
The prior $\pi$ on $(\overline{S},\sS)$ is the image via the mapping $\gtheta\mapsto t_\gtheta$ of a probability $\nu$ on $(\bs{\Theta},\sG(\bs{\Theta})$ that satisfies the following requirements for suitable constants $B\ge1$ and $\overline{\gamma}\in[1,4)$: if $\cB(\gtheta,r)$ denotes the closed Euclidean ball in $\bs{\Theta}$ with center $\gtheta$ and radius $r$, whatever $\gtheta\in \bs{\Theta}$ and $r>0$, 
\begin{equation}
\nu\!\left[\st\cB\!\left(\gtheta,2^kr\right)\right]\le\exp\left[B\overline{\gamma}^{k}\right]\nu\!\left[\st\cB(\gtheta,r)\right]\quad\mbox{for all }k\in\N\et.
\label{Eq-AS2}
\end{equation}
\end{ass}
The convexity of $\bs{\Theta}$ and the well-known formulas for the volume of Euclidean balls imply that this property holds for all probabilities which are absolutely continuous with respect to the Lebesgue measure with a density which is bounded from above and below but other situations are also possible. One simple example would be $\bs{\Theta}=[-1,1]$ and $\nu$ with density $(1/2)(\alpha+1)|x|^{\alpha}$, $\alpha>0$ with respect to the Lebesgue measure.
%
%
\begin{thm}\label{thm-B}
Under Assumptions~\ref{A-MLE} and~\ref{A-prior}, one can find two functions $C$ and $n_{1}$ on $(0,+\infty)$, also depending on $s$ and all the parameters involved in these assumptions but independent of $n$, such that, for all $n\ge n_{1}(z)$,
\[
\P_{\gs}\left[h^{2}\left(\pi_{\!\gX}^{L},\pi_{\!\gX}\right)\le C(z){(\log n)^{3/2}\over \sqrt{n}}\right]\ge1-e^{-z}
\quad\mbox{for all }z >0.
\]
\end{thm}
This means that, under suitably strong assumptions, the usual posterior and our $\rho$-posterior distributions are asymptotically the same which shows that our construction is a genuine generalization  of the classical Bayesian approach. 
It also implies that the Bernstein-von Mises Theorem also holds for $\pi_{\!\gX}$ as shown by the following result.
%
\begin{cor}\label{cor-BvM}
Let Assumptions~\ref{A-MLE} and~\ref{A-prior} hold, $(\widehat{\gtheta}_{n})$ be an asymptically efficient sequence of estimators of the true parameter $\gvartheta$ and assume that the following version of the Bernstein-von Mises Theorem is true:
\[
\norm{\pi_{\!\gX}^{L}-\cN\left(\widehat{\gtheta}_{n},[nI(\gvartheta]^{-1}\right)}_{TV}\CP{n\rightarrow+\infty}0,
\]
where $I$ denotes the Fisher Information matrix and $\norm{\cdot}_{TV}$ the total variation 
norm. Then the $\rho$-posterior distribution also satisfies the same Bernstein-von Mises Theorem, i.e.
\[
\norm{\pi_{\!\gX}-\cN\left(\widehat{\gtheta}_{n},[nI(\gvartheta]^{-1}\right)}_{TV}\CP{n\rightarrow+\infty}0.
\]
\end{cor}
%
\begin{proof}
It follows from the triangular inequality and the classical relationship between Hellinger and total variation distances given by~\eref{eq-h-TV}.
\end{proof}
%

\section{Combining different models}\label{MM}

\subsection{Priors and models}
In the case of simple parametric problems with parameter set $\bs{\Theta}$, such as those we considered in Section~\ref{RB}, $\overline{S}$ is the image of a subset of some Euclidean space $\R^d$ and one often chooses for $\gpi$ the image of a probability on $\bs{\Theta}$ which has a density with respect to the Lebesgue measure. The choice of a convenient prior $\gpi$ becomes more complex when $\overline{S}$ is a complicated function space which is very inhomogeneous with respect to the Hellinger distance. In such a case it is often useful to introduce ``models'', that is to consider $\overline{S}$ as a countable union of more elementary and homogeneous disjunct subsets $\overline{S}_{m}$, $m\in\cM$, and to choose a prior $\gpi_m$ on each $\overline{S}_{m}$ in such a way that Theorem~\ref{thm-main} applies to each model $\overline{S}_{m}$ and leads to a non-trivial result. It remains to put all models together by choosing some prior $\nu$ on $\cM$ and defining our final prior $\gpi$ on $\overline{S}=\bigcup_{m\in\cM}\overline{S}_{m}$ as $\sum_{m\in\cM}\nu(\{m\})\gpi_m$. This corresponds to a hierarchical prior.

One can as well proceed in the opposite way, starting from a global prior $\gpi$ on $\overline{S}$ and partitioning $\overline{S}$ into subsets $\overline{S}_{m}$, $m\in\cM$, of positive prior probability, then setting $\nu(\{m\})=\gpi(\overline{S}_{m})$ and defining $\gpi_m$ as the conditional distribution of a random element $t\in\overline{S}$ when it belongs to $\overline{S}_{m}$. The two points of view are actually clearly equivalent, the important fact for us being that the pairs $(\overline{S}_{m},\gpi_m)$ are such that Theorem~\ref{thm-main} can be applied to each of them. 

Throughout this section, we work within the following framework. Given a countable sequence of disjunct probability spaces  $(\overline S_{m},\sS_{m},\gpi_{m})_{m\in\cM}$ on $(\sX,\sA)$, we consider $\overline S=\bigcup_{m\in\cM}\overline S_{m}$ endowed with the $\sigma$-algebra $\sS$ defined as
\[
\sS=\{A\subset \overline S,\ A\cap \overline S_{m}\in\sS_{m}\ \mbox{for all}\ m\in\cM\}.
\]

In order to define our prior, we introduce a mapping $\pen$ from $\cM$ into $\R_{+}$ that will also be involved in the definition of our $\rho$-posterior distribution. The prior $\gpi$ on $\overline{S}$ is given by
\begin{equation}\label{def-prior-sm}
\gpi(A)=\Delta\sum_{m\in\cM}\int_{A\cap \overline S_{m}}\exp[-\beta\pen(m)]\,d\gpi_{m}(t)\quad \mbox{for all}\ A\in\sS
\end{equation}
with
\[
\Delta=\left(\sum_{m\in\cM}\int_{\overline S_{m}}\exp[-\beta\pen(m)]\,d\gpi_{m}(t)\right)^{-1},
\]
so that $\gpi$ is a genuine prior. This amounts to put a prior weight proportional to $\exp[-\beta\pen(m)]$ on the model $\overline{S}_{m}$. We shall assume the following.

\begin{ass}\label{H-mod}\text{}\vspace{-1mm}
\begin{listi}
\item\label{H'i}
For all $m\in\cM$ the function $(x,t)\to t(x)$ on $\sX\times\overline{S}_{m}$ is measurable with respect to the $\sigma$-algebra $\sA\otimes\sS_{m}$.
\item\label{H'ii}
For all $m\in\cM$ there exists a countable subset $S_{m}$ of \,$\overline S_{m}$ with the following property:: given $t\in\overline{S}_{m}$ and $t'\in S=\bigcup_{m'\in\cM}S_{m'}$, one can find a sequence $(t_{k})_{k\ge0}$ in $S_{m}$ such that (\ref{eq-t-t_k}) holds for all $x\in\sX$.
\item\label{H'iii} 
There exists a mapping $m\mapsto \overline \eps_{m}^{2}$ from $\cM$ to $\R_{+}$ such that, whatever the density $\gs\in\sL^{n}$,
\begin{equation}
\eps_{n}^{\overline{S}_{m}\cup \overline{S}_{m'}}(\gs)\le \sqrt{\overline \eps_{m}^{2}+\overline \eps_{m'}^{2}}\quad\mbox{for all }m,m'\in\cM.
\label{eq-eps}
\end{equation}
\item\label{H'iv} 
Given a set $\{L_{m},m\in \cM\}$ of nonnegative numbers satisfying 
\begin{equation}
\sum_{m\in\cM}\exp[-L_{m}]=1,
\label{eq-Lm}
\end{equation}
the penalty function pen is lower bounded in the following way:
\begin{equation}
\pen(m)\ge c_{5}n\overline \eps_{m}^{2}+\left(c_{6}+\beta^{-1}\right)L_{m}\quad\mbox{for all }m\in\cM, 
\label{eq-pen}
\end{equation}
with constants $c_{5}$ and $c_{6}$ defined in (\ref{eq-constantes}).
\end{listi}
\end{ass}

%
\subsection{The results}
We define the $\rho$-posterior distribution $\overline{\gpi}_{\!\gX}$ on $\overline{S}$ by its density with respect to the prior $\gpi$ given by~\eref{def-prior-sm} as follows:
\begin{equation}\label{def-post2}
{d\overline{\gpi}_{\!\gX}\over d\gpi}(t)={\exp\left[-\beta\overline \gPsi(\gX,t)\right]\over \int_{\overline S}\exp\left[-\beta\overline \gPsi(\gX,t')\right]d\gpi(t')}\quad\mbox{for all }t\in\overline S,
\end{equation}
with
\[
\overline \gPsi(\gX,t)=\sup_{m\in\cM}\sup_{t'\in S_{m}}\left[\gPsi(\gX,t,t')-\pen(m)\right].
\]
Note that if we choose $\beta=1$ and replace $\gPsi(\gX,t,t')$ by the difference of the log-likelihoods $\sum_{i=1}^{n} \log t'(X_{i})-\sum_{i=1}^{n}\log t(X_{i})$, $\overline \gpi_{\gX}$ is the usual posterior distribution corresponding to the prior $\gpi$.
We finally, introduce a mapping $\overline \eta$ on $\overline S$ which associates to an element $t\in \overline{S}_{m}$ with $m\in\cM$ the quantity $\overline\eta_{n}^{2}(t)$
given by
\begin{equation}\label{eq-eta}
\overline \eta_{n}^{2}(t)=\inf_{r\in(0,1]}\cro{c_{7}r^{2}+{1\over 2n\beta}\log\pa{1\over \gpi_{m}(\sB^{\overline S_{m}}(t,r))}}\quad \mbox{for all $t\in\overline S_{m}$},
\end{equation}
which only depends on the choice of the prior $\gpi_{m}$ on $\overline S_{m}$. Taking $r=1$, we see that $ \eta_{n}^{2}(t)\le c_{7}$ for all $t\in\overline{S}$. Moreover, if, for some $\eta\in(0,1]$ and $\lambda>0$,
\[
\gpi_{m}\!\left(\sB^{\overline{S}_{m}}(t,r)\right)\ge\exp\left[-\lambda n r^{2}\right]\quad\mbox{ for all }r\ge \eta,\;m\in\cM\text{ and }t\in\overline S_{m},
\]
then 
\[
\overline \eta_{n}^{2}(t)\le\inf_{r\ge\eta}\cro{c_{7}r^{2}+{\lambda r^{2}\over 2\beta}}=
\cro{c_{7}+{\lambda\over 2\beta}}\eta^{2},
\]
a result which is similar to the one we derived for $\eta_{n}^{\overline{S},\gpi}(t)$ in Section~\ref{sect-eta} under an analogous assumption.
%
\begin{thm}\label{thm-main2}
Let Assumption~\ref{H-mod} hold. For all $\xi>0$ and  whatever the density $\gs\in\sL^{n}$ of $\gX$, there exists a set $\Omega_{\xi}$ with $\P_{\gs}(\Omega_{\xi})\ge1-e^{-\xi}$ and such that  
\[
\overline{\gpi}_{\gX(\omega)}\left(\sB^{\overline{S}}(\gs,r)\right)\ge 1-e^{-\xi'}
\quad\mbox{for all }\omega\in \Omega_{\xi},\:\xi'>0\mbox{ and }r\ge\overline r_{n}
\]
with
\begin{align*}
\overline r_{n}^{2}=&\inf_{m\in\cM}\inf_{\overline s\in\overline S_{m}}\cro{{3c_{7}\over c_{8}}h^{2}(\gs,\overline s)-h^{2}(\gs,\overline{S})+\frac{2}{c_{8}}\pa{\frac{2\pen(m)}{n}+\overline \eta_{n}^{2}(\overline s)-\frac{L_{m}}{\beta n}}}\\&+c_{9}\frac{\xi+\xi'+2.4}{n}
\end{align*}
and constants $c_{j}, 7\le j\le9$ defined in (\ref{eq-constantes}).
\end{thm}

This result about the concentration of the $\rho$-posterior distribution is analogue to that one can obtain from a frequentist point of view by using a model selection method. Up to possible extra logarithmic terms, the $\rho$-posterior concentrates at a rate which achieves the best compromise between the approximation and complexity terms among the family of models.

\subsection{Model selection among exponential families\label{ex-MS}}
In this section we pretend that the observations $X_{1},\ldots,X_{n}$ are i.i.d. but  keep in mind that the $X_{i}$ might not be equidistributed so that their true joint density $\gs$ might not be of the form $(s,\ldots,s)$. 

Hereafter, $\ell_{2}(\N)$ denotes the Hilbert space of all square-summable sequences $\gtheta=(\theta_{j})_{j\ge 0}$ of real numbers that we endow with the Hilbert norm $|\cdot|$ and the inner product $\<\cdot,\cdot\>$. Let $\cM=\N$, $M$ be some positive number and for $m\in\cM$, let $\bs{\Theta}_{m}'$ be the subset of $\ell_{2}(\N)$ of these sequences $\gtheta=(\theta_{j})_{j\ge 0}$ such that $\theta_{j}\in [-M,M]$ for $0\le j\le m$ and $\theta_{j}=0$ for all $j>m$. 

For a sequence $\gT=(T_{j})_{j\ge 0}$ of linearily independent measurable real-valued functions on $\sX$ with $T_{0}\equiv1$ and $m\in\cM$, we define the density model $\overline{S}_{m}$ as the exponential family 
\[
\overline{S}_{m}=\ac{t_{\gtheta}=\exp\cro{\<\gtheta,\gT\>-A(\gtheta)},\ \gtheta\in\bs{\Theta}_{m}',\ \theta_{m}\neq 0},
\]
where $A$ denotes the mapping from $\bs{\Theta}=\bigcup_{m\in\cM}\bs{\Theta}_{m}'$ to $\R$ defined by 
\[
A(\gtheta)=\log\int_{\sX}\exp\cro{\<\gtheta,\gT(x)\>}d\mu(x),
\]
and $\mu$ is a finite measure on $\sX$.
Note that, whatever $\gtheta\in\bs{\Theta}$, $x\mapsto \<\gtheta,\gT(x)\>$ is well-defined on $\sX$ since only a finite number of coefficients of $\gtheta$ are non-zero.  

For all $m\in\cM$, we endow $\overline S_{m}$ with the Borel $\sigma$-algebra $\sS_{m}$ and the prior $\pi_{m}$ which is the image of the uniform distribution on $\bs{\Theta}_{m}'$ (identified with $[-M,M]^{m+1}$) by the mapping $\gtheta\mapsto t_{\gtheta}$ on $\bs{\Theta}_{m}'$. Throughout this section, we consider the family of (disjunct) measured spaces $(\overline S_{m},\sS_{m},\pi_{m})$ with $m\in\cM$ together with the choice $L_{m}=(m+1)\log 2$ for all $m\in\cM$, so that $\sum_{m\in\cM}e^{-L_{m}}=1$. Then $\overline S=\bigcup_{m\in\cM}\overline S_{m}$, $\pi$ is given by~\eref{def-prior-sm} and for all $m\in\cM$,
\[
\pen(m)=c_{5}n\overline \eps_{m}^{2}+(c_{6}+\beta^{-1})L_{m}\quad\mbox{with}\quad
\overline \eps_{m}=\frac{11c_{0}}{4}\sqrt{\frac{\overline{c}_{n}(m+3)}{n}}\log^{3/2}(en)
\]
and $c_{0}, c_{5}, c_{6}$, $\overline{c}_{n}$ defined in (\ref{eq-constantes}). In such a situation we derive the following result. 
%
\begin{prop}\label{prop-selexp}
Assume that, for all $m\in\cM$, the restriction $A_{m}$ of $A$ to $\bs{\Theta}_{m}'$ 
is convex and twice differentiable on the interior of $\bs{\Theta}_{m}'$ with a Hessian whose eigenvalues lie in $(0,\sigma_{m}]$ for some $\sigma_{m}>0$. Whatever the density $\gs$ of $\gX$, for all $\xi>0$, with $\P_{\gs}$-probability at least $1-e^{-\xi}$, 
\[
\overline{\gpi}_{\gX}\left(\sB^{\overline{S}}(\gs,r)\right)\ge 1-e^{-\xi'}\ \ \mbox{for all $\xi'>0$ and all $r\in[\overline r_{n},1]$}
\]
with
\begin{align*}
\overline r_{n}^{2}\le &\ C(\beta)\inf_{m\ge 1}\cro{h^{2}(\gs,\overline{S}_{m})+ \frac{m+1}{n}\!\cro{\log^{3}(en)+\log\pa{1+n\sigma_{m}^{2}M^{2}}}\str{4}}\\
& +c_{9}\frac{\xi+\xi'+2.4}{n}
\end{align*}
and some constant $C(\beta)>0$ depending on $\beta$ only.
\end{prop}

\section*{Acknowledgements}
The first author has received funding from the European Union's Horizon 2020 research and innovation programme under grant agreement N\textsuperscript{o} 811017.

The second author was supported by the grant ANR-17-CE40-0001-01 of the French National Research Agency ANR (project BASICS) and by Laboratoire J.A. Dieudonn\'e (Nice).

\bibliographystyle{apalike}

\newpage
\part*{Supplement to ``Robust Bayes-Like Estimation: Rho-Bayes estimation"}

\section{Measurability issues}

\subsection{Proof of Proposition~\ref{prop-mes}}
Since $(x,t)\mapsto t(x)$ is measurable, the same holds for $(x,t,t')\mapsto\left(\st t(x),t'(x)\right)$ hence for $(x,t,t')\mapsto \psi\left(\sqrt{t'(x)/t(x)}\right)$ and $(x,t,t')\mapsto \gPsi(\gx,t,t')$ as a sum of measurable functions; the first assertion follows since $S$ is countable. The second assertion immediately follows from the first. As to the last one, the conclusion follows from the fact that the mapping $t\mapsto \rho(s,t)=\int_{\sX}\sqrt{s(x)t(x)}d\mu(x)$ is a measurable function of $t\in\overline S$ for all $s\in\sL$.
%
%

\subsection{Piecewise exponential families}
The following proposition provides sufficient conditions for some exponential families of interest to satisfy our Assumption~\ref{hypo-mes}.

\begin{prop}\label{mes-cond-modele}
Let $(T_{j})_{1\le j\le J}$ be  $J\ge 1$ measurable functions on $(\sX,\sA)$ and $\bs{\Theta}\in\R^{J}$. If the density model $\overline S=\{t_{\gtheta}, \gtheta\in \bs{\Theta}\}$ on $(\sX,\sA,\mu)$ has one of the two following forms a) or b) below, it satisfies Assumption~\ref{hypo-mes} with $\sS$ the Borel $\sigma$-algebra on $(\overline S,h)$ and $S=\{t_{\gtheta}, \gtheta\in \bs{\Theta}'\}$ where $\bs{\Theta}'$ denotes any dense and countable subset of $\bs{\Theta}$.\vspace{-1mm}
\begin{lista}
\item $\bs{\Theta}$ is a convex and compact subset of $\R^{J}$ and $t_{\gtheta}$ is of the form~\eref{eq-expt} where $A$ is strictly convex and continuous on $\bs{\Theta}$;
\item $\sJ=(I_{i})_{i=1,\ldots,k}$ is a partition of $\sX$ into $k\ge2$ measurable subsets of positive measure, $\bs{\Theta}=\prod_{i=1}^{k}\Theta_{i}$ where each $\Theta_{i}$ is a convex and compact subset of $\R^{J}$ and $t_{\gtheta}$ has the form
\[
t_{\gtheta}(x)=\sum_{i=1}^{k}\exp\cro{\sum_{j=1}^{J}\theta_{i,j}T_{j}(x)-A(\gtheta)}\1_{I_{i}}(x)\quad\text{for all }x\in\sX,
\]
where $\gtheta=(\theta_{i,j})_{\!\!\tiny{\begin{array}{l}i=1,\ldots,k\\j=1,\ldots,J\end{array}}}$ and
\[
A(\gtheta)=\log \left(\sum_{i=1}^{k}\int_{I_{i}}\exp\cro{\sum_{j=1}^{J}\theta_{i,j}T_{j}(x)}d\mu(x)\right)
\]
is continuous and strictly  convex on $\bs{\Theta}$.
\end{lista}
\end{prop}
It is well-known that in (\ref{eq-expt}), $A$ is strictly convex and continuous on $\bs{\Theta}$ when $\bs{\Theta}$ is a subset of the interior of the set
\[
\ac{\gtheta=(\theta_{1},\ldots,\theta_{J})\in\R^{J},\int_{\sX}\exp\cro{\sum_{j=1}^{J}\theta_{j}T_{j}(x)}d\mu(x)<+\infty}
\]
and $T_{1},\ldots,T_{J}$ are almost surely affinely independent, which means that, for all $(\lambda_{1},\ldots,\lambda_{J})\in\R^{J}\setminus\{\bf 0\}$, $\sum_{j=1}^{J}\lambda_{j}T_{j}$ is not constant a.s. If $\overline S$ is not of the form $a)$ or $b)$ but is a subset of a density model of one of these forms then $\overline S$ also satisfies Assumption~\ref{hypo-mes} with the Borel $\sigma$-algebra  $\sS$  on $(\overline S,h)$. 

\begin{proof}
In case $a)$, for all $\gtheta,\gtheta'\in \bs{\Theta}$, the Hellinger affinity writes
\[
\rho(t_{\gtheta},t_{\gtheta'})=\exp\left[-\frac{A(\gtheta)+A(\gtheta')}{2}\right]\int_{\sX}\exp\cro{-\sum_{j=1}^{J}\pa{\theta_{j}+\theta_{j}'\over 2}T_{j}(x)}d\mu(x)
\]
and, since $\bs{\Theta}$ is convex, $(\gtheta+\gtheta')/2\in \bs{\Theta}$ so that
\begin{equation}
\rho(t_{\gtheta},t_{\gtheta'})=\exp\cro{A\pa{\gtheta+\gtheta'\over 2}-\pa{A(\gtheta)+A(\gtheta')\over 2}}.
\label{eq-rhott'}
\end{equation}
In case $b)$, the same argument shows that
\begin{align*}
\rho(t_{\gtheta},t_{\gtheta'})&=\exp\left[-\frac{A(\gtheta)+A(\gtheta')}{2}\right]\\
&\quad \times \sum_{i=1}^{k}\int_{I_{i}}\exp\cro{-\sum_{j=1}^{J}\pa{\theta_{i,j}+\theta_{i,j}'\over 2}T_{j}(x)}d\mu(x)
\end{align*}
and (\ref{eq-rhott'}) also holds. The function $A$ being strictly convex, $h(t_{\gtheta},t_{\gtheta'})=0$ implies that $\gtheta=\gtheta'$, hence $t_{\gtheta}=t_{\gtheta'}$ and $(\overline S,h)$ is therefore a metric space. Moreover, the mapping $\gtheta\mapsto t_{\gtheta}$ is continuous from $(\bs{\Theta},\ab{\cdot})$ to the metric space $(\overline S,h)$ and since it is also one-to-one from the compact set $\bs{\Theta}$ onto $\overline S$, it is an homeomorphism from $(\bs{\Theta},\ab{\cdot})$ to $(\overline S,h)$. 
If we denote by $g$ the continuous function from $\overline{S}$ to $\bs{\Theta}$ defined by $g(t)=\gtheta$ if $t=t_{\gtheta}$, the function $(x,t_{\gtheta})\mapsto(x,\gtheta)$ is measurable and since the function $(x,t_{\gtheta})\mapsto t_{\theta}(x)$ is the composition of the maps $(x,t_{\gtheta})\mapsto(x,\gtheta)$ and $(x,\gtheta)\mapsto t_{\gtheta}(x)$, it is enough, in order to check Assumption~\ref{hypo-mes}-$\ref{Hi}$, to show that $(x,\gtheta)\mapsto t_{\gtheta}(x)$ is measurable.
The functions $x\mapsto T_{j}(x)$ are mesurable and the function $\gtheta\mapsto A(\gtheta)$ as well by continuity. Since products and sums of measurable functions are measurable, $(x,\gtheta)\mapsto\sum_{j=1}^{J}\theta_{i,j}T_{j}(x)-A(\gtheta)$ is measurable
hence $(x,\gtheta)\mapsto t_{\gtheta}(x)$ as well.

Since $(\bs{\Theta},\ab{\cdot})$ is compact, it is separable and a countable and dense subset $\bs{\Theta}'$ of $\bs{\Theta}$ does exist. We take $S=\{t_{\gtheta}, \gtheta\in \bs{\Theta}'\}$. Since $A$ is continuous, the mapping $\gtheta\mapsto t_{\gtheta}(x)$ is continuous on $\bs{\Theta}$ for all $x\in\sX$. Moreover, since $t'(x)>0$ for all $t'\in S$ and $x\in\sX$, if $t_{k}(x)\rightarrow t(x)$, $\psi\left(\sqrt{t'(x)/t_{k}(x)}\right)\rightarrow\psi\left(\sqrt{t'(x)/t(x)}\right)$ so that Assumption~\ref{hypo-mes}-$\ref{Hii}$ is satisfied.
\end{proof}
%

\section{Proof of Theorem~\ref{thm-main}}
The proof relies on suitable bounds for the numerator and denominator of the density
\begin{equation}
p_{\!\gX(\omega)}(t)={d\gpi_{\!\gX(\omega)}\over d\gpi}(t)={\exp\left[-\beta \gPsi(\gX(\omega),t)\right]\over 
\dps{\int_{\overline{S}}}\exp\left[-\beta \gPsi(\gX(\omega),t')\right]d\gpi(t')},
\label{eq-th1-a}
\end{equation}
which themselves derive from bounds on $\gPsi(\gX(\omega),t)$. To get such bounds, we shall use the three following results to be proven in Section~\ref{Sect-12}.

\begin{prop}\label{prop-ass1}
Under Assumption~\ref{hypo-mes}, the sequence $(t_{k})_{k\ge0}$ satisfies
\begin{equation}
\lim_{k\to +\infty}h(t,t_{k})=0,\qquad
\lim_{k\to +\infty}\gPsi(\gx,t_{k},t')=\gPsi(\gx,t,t').
\label{eq-ass1}
\end{equation}
\end{prop}
%

\begin{thm}\label{main}
Under Assumption~\ref{hypo-mes}, whatever the density $\gs$ of $\gX$ and $\xi>0$, there exists a measurable subset $\Omega_{\xi}$ of $\Omega$ the $\P_{\gs}$-probability of which is at least $1-e^{-\xi}$ and such that for all $\omega\in\Omega_{\xi}$, $t\mapsto \gPsi(\gX(\omega),t,t')$ is measurable on $(\overline S,\sS)$ for all $t'\in S$ and, for all $t\in\overline{S}$,
\begin{equation}\label{eq-main}
\frac{1}{n}\gPsi(\gX(\omega),t,t')\le c_{7} h^{2}(\gs,t) -c_{8}h^{2}(\gs,t') + c_{5}\left(\eps_{n}^{\overline{S}}(\gs)\right)^2+c_{6}\frac{\xi+2.4}{n},
\end{equation}
where $\eps_{n}^{\overline{S}}(\gs)$ is given by~\eref{def-en} and the constants $c_j$ have been defined in (\ref{eq-constantes}).
\end{thm}
%
\begin{prop}
\label{Prop-int}
Let $r$, $a$ and $b$ satisfy $0<(\sqrt{n}r)^{-2}\le a\le b$. If $J\in\N\et$ is such that $4^{J-1}\ge b/a$, $r_{0}=2^{-J}r$ and $t\in\overline{S}$ is such that
\begin{equation}
\gpi\left(\sB^{\overline{S}}(t,2r')\right)\le \exp\left[(3a/8)nr'^{2}\right]\gpi\left(\sB^{\overline{S}}(t,r')\right)\quad\mbox{for all }r'\ge r_{0},
\label{Eq-Gamma}
\end{equation}
then
\begin{equation}
{\dps{\int_{\overline{S}\setminus \sB^{\overline{S}}(t,r)}}\exp\left[-anh^{2}(t,t')\right]d\gpi(t')\over
\dps{\int_{\overline{S}}}\exp\left[-bnh^{2}(t,t')\right]d\gpi(t')}\le\exp\left[{1-anr^{2}\over4}\right].
\label{eq-RapInt}
\end{equation}
\end{prop}
Let us now fix $\xi>0$, $\omega$ in the set $\Omega_{\xi}$ provided by Theorem~\ref{main} and set $\gx$ for $\gX(\omega)$, $\eps(\gs)$ for $\eps_{n}^{\overline{S}}(\gs)$ and $\eta(t)$ for $\eta_{n}^{\overline{S},\gpi}(t)$ in order to keep our formulae as simple as possible. For $u\in\overline{S}$ and $u'\in S$, (\ref{eq-main}) can be written
\begin{equation}
\gPsi(\gx,u,u')\le n\left[c_{7}h^{2}(\gs,u) -c_{8}h^{2}(\gs,u') +c_{5}\eps^{2}(\gs)\right]+c_{6}(\xi+2.4).
\label{eq-imp}
\end{equation}

To bound the denominator of $p_{\gx}(t)$ from below we apply~\eref{eq-imp} with $u=t'\in \overline{S}$ and, since $h(\gs,u')\ge h(\gs,S)=h(\gs,\overline{S})$, derive that, whatever $t'\in\overline{S}$,
\[
\frac{\gPsi(\gx,t')}{n}=\sup_{u'\in S}\frac{\gPsi(\gx,t',u')}{n}\le c_{7} h^{2}(\gs,t')-c_{8}h^{2}(\gs,\overline{S})+ c_{5}\eps^{2}(\gs)+c_{6}\frac{\xi+2.4}{n}.
\]
Therefore, 
\begin{align*}
\int_{\overline{S}}\exp\left[-\beta \gPsi(\gx,t')\right]d\gpi(t')\ge&\,\exp\left[\beta\left(c_{8}nh^{2}(\gs,\overline{S})- c_{5}n\eps^{2}(\gs)-c_{6}(\xi+2.4)\right)\right]\\&\,\times\int_{\overline{S}}\exp\left[-\beta c_{7}nh^{2}(\gs,t')\right]d\gpi(t').
\end{align*}

Let us now bound the numerator of $p_{\gx}(t)$ from above for $t\in\overline{S}$. Assumption~\ref{hypo-mes} implies (\ref{eq-ass1}), hence, for $u\in S$ there exists a sequence $(t_{k})_{k\ge0}\in S$ such that
\[
\gPsi(\gx,t_{k},u)\CV{k\to +\infty}\gPsi(\gx,t,u)\qquad\mbox{and}\qquad h^{2}(t,t_{k})\CV{k\to +\infty}0.
\]
Applying \eref{eq-imp} with $u'=t_{k}$ leads to
\begin{align*}
-\gPsi(\gx,t_{k},u)&=\gPsi(\gx,u,t_{k})\\&\le n\left[c_{7}h^{2}(\gs,u)-c_{8}h^{2}(\gs,t_{k})+ c_{5}\eps^{2}(\gs)\right]+c_{6}(\xi+2.4)
\end{align*}
and letting $k$ tend to infinity shows that, for $t\in\overline{S}$ and $u\in S$,
\[
-\gPsi(\gx,t,u)\le n\left[c_{7}h^{2}(\gs,u)-c_{8}h^{2}(\gs,t)+ c_{5}\eps^{2}(\gs)\right]+c_{6}(\xi+2.4).
\]
Consequently,
\begin{align*}
-\gPsi(\gx,t)&=\inf_{u\in S}\left[-\gPsi(\gx,t,u)\right]\\&\le c_{7}
\left[\inf_{u\in S}nh^{2}(\gs,u)\right]-c_{8}nh^{2}(\gs,t)+ c_{5}n\eps^{2}(\gs)+c_{6}(\xi+2.4)
\end{align*}
and, since $\inf_{u\in S}h^{2}(\gs,u)=h^{2}(\gs,S)=h^{2}(\gs,\overline{S})$, for all $t\in\overline{S}$,
\begin{align*}
\lefteqn{\exp\left[-\beta \gPsi(\gx,t)\right]}\hspace{12mm}\\
&\le\exp\left[\beta\left(c_{5}n\eps^{2}(\gs)+c_{6}(\xi+2.4)+c_{7}nh^{2}(\gs,\overline{S})\right)-\beta c_{8}nh^{2}(\gs,t)\right].
\end{align*}
Putting the bounds for the numerator and denominator of (\ref{eq-th1-a}) together, we derive that, for all $t\in\overline{S}$,
\begin{align}
p_{\gx}(t)\le&\,\exp\left[\beta\left((c_{7}-c_{8})nh^{2}(\gs,\overline{S})+2c_{5}n\eps^{2}(\gs)+2c_{6}(\xi+2.4)\right)\right]\nonumber\\
&\,\times{\exp\left[-\beta c_{8}nh^{2}(\gs,t)\right]\over \dps{\int_{\overline{S}}}\exp\left[-\beta c_{7}nh^{2}(\gs,t')\right]d\gpi(t')}\nonumber\\
\le&\,\exp\left[\beta\left((c_{7}-c_{8})nh^{2}(\gs,u)+2c_{5}n\eps^{2}(\gs)+2c_{6}(\xi+2.4)\right)\right]\label{eq-bpx}
\\&\,\times{\exp\left[-\beta c_{8}nh^{2}(\gs,t)\right]\over\dps{\int_{\overline{S}}}\exp\left[-\beta c_{7}nh^{2}(\gs,t')\right]d\gpi(t')}\qquad\mbox{for all }u\in\overline{S}\nonumber,
\end{align}
since $c_{7}-c_{8}>0$.
It follows from the inequalities $\gh(t,u)\le\gh(t,\gs)+\gh(\gs,u)$, $\gh(\gs,t')\le\gh(\gs,u)+\gh(u,t')$ and (\ref{eq-2ab}) that, whatever the positive numbers $a,b$, for all $t,t',u \in\overline{S}$, $h^{2}(\gs,t)\ge a(1+a)^{-1}h^{2}(t,u)-a\gh^{2}(\gs,u)$ and $h^{2}(\gs,t')\le(1+b)h^{2}(u,t')+(1+b^{-1})h^{2}(\gs,u)$.
Together with \eref{eq-bpx} this leads to
\begin{align*}
p_{\gx}(t)\le&\,\exp\left[\beta\left\{\left[c_{7}\left(2+b^{-1}\right)+c_{8}(a-1)\right]nh^{2}(\gs,u)+2c_{5}n\eps^{2}(\gs)\right\}\right]\\&\,\times\exp[2\beta c_{6}(\xi+2.4)]
{\exp\left[-\beta[(ac_{8})/(a+1)]nh^{2}(t,u)\right]\over\dps{\int_{\overline{S}}}\exp\left[-\beta c_{7}(1+b)nh^{2}(u,t')\right]d\gpi(t')},
\end{align*}
for all $t,u\in\overline{S}$. Choosing $a=11$ so that $c_{8}(a-1)<c_{7}$ and $b=1$, then setting 
$c_{7}'=(1+b) c_{7}=8.02$ and $c_{8}'=1/3<a(a+1)^{-1}c_{8}$, we get
\begin{align}
p_{\gx}(t)\le&\,\exp\left[\beta\left(4c_{7}nh^{2}(\gs,u)+2c_{5}n\eps^{2}(\gs)+2c_{6}(\xi+2.4)\right)\right]\nonumber\\
&\,\times{\exp\left[-\beta c_{8}'nh^{2}(t,u)\right]\over\dps{\int_{\overline{S}}}\exp\left[-\beta c_{7}'nh^{2}(u,t')\right]d\gpi(t')}\qquad\mbox{for all }t,u\in\overline{S}.
\label{main-eq1}
\end{align}
%
For $u\in\overline{S}$ and $r>0$, let $\sB_{1}=\sB^{\overline{S}}(u,r)$ and $\sB_{2}= \sB^{\overline{S}}(\gs,r+h(\gs,u))$ so that $\sB_{1}\subset\sB_{2}$, hence
%
\begin{align*}
\gpi_{\gx}\left(\overline{S}\setminus \sB_{2}\right)
\le&\,\gpi_{\gx}\left(\overline{S}\setminus \sB_{1}\right)=\int_{\overline{S}\setminus\sB_{1}}p_{\gx}(t)\,d\gpi(t)\\
\le &\,\exp\left[\beta\left(4c_{7}nh^{2}(\gs,u)+2c_{5}n\eps^{2}(\gs)+2c_{6}(\xi+2.4)\right)\right]\\
&\,\times\int_{\overline{S}\setminus\sB_{1}}{\exp\left[-\beta c_{8}'nh^{2}(t,u)\right]\over \dps{\int_{\overline{S}}}\exp\left[-\beta c_{7}'h^{2}(t',u)\right]d\gpi(t')}\,d\gpi(t)\\= &\,\exp\left[\beta\left(4c_{7}nh^{2}(\gs,u)+2c_{5}n\eps^{2}(\gs)+2c_{6}(\xi+2.4)\right)\right]\\
&\,\times\frac{\dps{\int_{\overline{S}\setminus\sB_{1}}}\exp\left[-\beta c_{8}'nh^{2}(t',u)\right]d\gpi(t')}{\dps{\int_{\overline{S}}}\exp\left[-\beta c_{7}'nh^{2}(t',u)\right]d\gpi(t')}.
\end{align*}
We may now apply Proposition~\ref{Prop-int} to the last fraction with $a=\beta c_{8}'<b=\beta c_{7}'$, so that $b/a=c'_{7}/c'_{8}<25$, $J=4$ and
\begin{equation}
\sqrt{n}r\ge\left[2^{J}\sqrt{n}\eta(u)\right]\bigvee(\beta c_{8}')^{-1/2}=\left[16\sqrt{n}\eta(u)\right]\vee a^{-1/2}.
\label{eq-rr0}
\end{equation}
Since $\gamma= \beta/8=3\beta c_{8}'/8=3a/8$ and $r_{0}=2^{-J}r\ge\eta(u)$, (\ref{Eq-Gamma}) holds and all the conditions needed for applying Proposition~\ref{Prop-int} are satisfied, leading to
\[
{\dps{\int_{\overline{S}\setminus\sB_{1}}}\exp\left[-\beta c_{8}'nh^{2}(t',u)\right]d\gpi(t')\over \dps{\int_{\overline{S}}}\exp\left[-\beta c_{7}'nh^{2}(t',u)\right]d\gpi(t')}\le\exp\left[\frac{1-\beta c_{8}'nr^{2}}{4}\right]=\exp\left[\frac{1-\beta c_{8}'nr^{2}}{4}\right].
\]
Since $c_{8}'=1/3$, it follows that
\begin{align}
\lefteqn{\gpi_{\gx}\left(\overline{S}\setminus \sB_{2}\right)}\hspace{10mm}\nonumber\\
&\le \exp\left[\beta\left(4c_{7}nh^{2}(\gs,u)+2c_{5}n\eps^{2}(\gs)+2c_{6}(\xi+2.4)+{1\over {4\beta}}-{nr^{2}\over 12}\right)\right]
\label{eq-pix}
\end{align}
provided that $r$ satisfies (\ref{eq-rr0}). Let $r_{0}(u)=\left[16\eta(u)\right]\vee r'_{0}(u)$ with 
\[
r_{0}'(u)=\sqrt{12}\left[2\sqrt{c_{7}}h(\gs,u)+\sqrt{2c_{5}}\eps(\gs)+\sqrt{2c_{6}\frac{\xi+2.4}{n}+\frac{1+4\xi'}{4\beta n}}\right]>\sqrt{\frac{3}{\beta n}}.
\]
Then, if $r=r_{0}(u)$, (\ref{eq-rr0}) holds and the left-hand side of (\ref{eq-pix}) is not larger than $e^{-\xi'}$. It follows that $\gpi_{\gx}\left(\overline{S}\setminus \sB^{\overline{S}}(\gs,r')\right)\le e^{-\xi'}$ provided that $r'\ge h(\gs,u)+r_{0}(u)$. Since, 
\[
h(\gs,u)+r_{0}(u)< c_{1}h(\gs,u)+c_{2}\eta(u)+c_{3}\eps(\gs)+c_{4}n^{-1/2}\sqrt{\xi+\xi'+2.61}
\]
with constants $c_{j}$ given by (\ref{eq-constantes}) so that $c_{4}>2(c_{8}')^{-1/2}\max\{\sqrt{2c_{6}};
\beta^{-1/2}\}$, it is actually enough that $r$ be not smaller than $\overline r_{n}$ given by~\eref{def-rb}, as claimed.

\section{Proof of Theorem~\ref{thm-main2}}
In order to simplify the presentation of the proof, we shall occasionally consider the penalty function as a function on $\overline{S}$ writing $\pen(t)$ instead of $\pen(m)$ when $t\in\overline{S}_{m}$, which implies that
\begin{equation}
\overline \gPsi(\gX,t)=\sup_{t'\in S}\left[\gPsi(\gX,t,t')-\pen(t')\right].
\label{eq-psibar}
\end{equation}

For $m,m'\in\cM$, we consider the set  $\overline S_{m,m'}=\overline{S}_{m}\cup \overline{S}_{m'}$ that we endow with the $\sigma$-algebra 
\[
\sS_{m,m'}=\left\{A\subset \overline S_{m,m'}, A\cap \overline S_{i}\in\sS_{i}\ \mbox{for all}\ i\in\{m,m'\}\right\}.
\]
For all $z\in\R_{+}$ and $i\in\{m,m'\}$, the set
\[
\left\{\left.(x,t)\in\sX\times \overline{S}_{m,m'}\,\right|\, t(x)\le z\right\}\cap\overline S_{i}
=\left\{\left.(x,t)\in \sX\times\overline{S}_{i}\,\right|\, t(x)\le z\right\}
\]
belongs to $\in\sA\otimes\sS_{i}$ by Assumption~\ref{H-mod}-$\ref{H'i}$, hence the function $(x,t)\mapsto t(x)$ is measurable from $(\overline S_{m,m'},\sS_{m,m'})$ to $\R_{+}$. Besides, Assumption~\ref{H-mod}-$\ref{H'ii}$ implies that $\overline S_{m,m'}$ satisfies Assumption~\ref{hypo-mes}-$\ref{Hii}$ with $S_{m,m'}=S_{m}\cup S_{m'}$. Consequently, $(\overline S_{m,m'},\sS_{m,m'})$ satisfies Assumption~$\ref{hypo-mes}$. We may therefore apply Theorem~\ref{main} to this model with $\xi_{m,m'}=L_{m}+L_{m'}+\xi$ in place of $\xi$ and deduce that there exists a measurable subset $\Omega_{m,m'}$ of $\Omega$ with
\begin{equation}
\P_{\gs}[\Omega_{m,m'}]\ge1-e^{-\xi_{m,m'}}=1-e^{-(L_{m}+L_{m'}+\xi)}
\label{eq-omm'}
\end{equation}
such that the function $t\mapsto \gPsi(\gX(\omega), t,t')$ is measurable from $(\overline S_{m,m'},\sS_{m,m'})$ to $\R_{+}$ for all $\omega\in \Omega_{m,m'}$ and $t'\in S_{m,m'}$, hence, by the definition of $\sS_{m,m'}$, the restriction of the function $t\mapsto \gPsi(\gX(\omega), t,t')$ to $\overline S_{m}$ is measurable from $(\overline S_{m},\sS_{m})$ to $\R_{+}$. Moreover, by \eref{eq-eps},  (\ref{eq-main}) and \eref{eq-pen}, for all $\omega\in \Omega_{m,m'}$, $u\in \overline{S}_{m}$ and $u'\in S_{m'}$,
\begin{align*}
\lefteqn{\gPsi(\gX(\omega),u,u')}\hspace{5mm}\\
&\le c_{7}n h^{2}(\gs,u) -c_{8}nh^{2}(\gs,u') + c_{5}n\left[\overline \eps_{m}^{2}+\overline \eps_{m'}^{2}\right]+c_{6}\left[L_{m}+L_{m'}+\xi+2.4\right]\\ &\le c_{7} nh^{2}(\gs,u) -c_{8}nh^{2}(\gs,u')+ c_{6}(\xi+2.4)-\frac{L_{m}+L_{m'}}{\beta}\\&\hspace{5mm}
+\pen(u)+\pen(u')
\end{align*}
and this bound holds simultaneously for all $m,m'\in \cM$, $u\in \overline{S}_{m}$, $u'\in S_{m'}$ and $\omega\in\Omega_{\xi}=\bigcap_{m,m'\in\cM}\Omega_{m,m'}$. Since $\gPsi(\gX(\omega),u,u')=-\gPsi(\gX(\omega),u',u)$, we finally conclude that, for all $\omega\in\Omega_{\xi}$, 
$u\in \overline{S}_{m}$ and $u'\in S_{m'}$,
\begin{align}
\gPsi(\gX(\omega),u,u')-\pen(u')=&-\gPsi(\gX(\omega),u',u)-\pen(u')\nonumber\\
\le&\;c_{7} nh^{2}(\gs,u) -c_{8}nh^{2}(\gs,u')+c_{6}(\xi+2.4)
\label{eq-00}
\\&\;+\pen(m)-\beta^{-1}(L_{m}+L_{m'}).\nonumber
\end{align}
Let us assume from now on that $\omega\in\Omega_{\xi}$ and fix some $\overline m\in\cM$ and some $\overline s\in\overline{S}_{\overline m}$. For all $t\in\overline{S}_{\overline{m}}$ we deduce from~\eref{eq-psibar} and \eref{eq-00}, since $L_{m'}\ge0$, that
\begin{align}
\lefteqn{\overline \gPsi(\gX,t)+\pen(\overline m)}\hspace{17mm}\nonumber\\
=&\; \sup_{t'\in S}\left[\gPsi(\gX,t,t')-\pen(t')\right]+\pen(\overline m)\nonumber\\
\le&\; c_{7} nh^{2}(\gs,t) -c_{8}nh^{2}(\gs,S)+c_{6}(\xi+2.4)+2\pen(\overline m)-\beta^{-1}L_{\overline{m}}\nonumber
\\ \le&\; 2c_{7}n\cro{h^{2}(\gs,\overline s)+h^{2}(\overline \gs,t)} -c_{8}nh^{2}(\gs,\overline{S})+2\pen(\overline m)-\beta^{-1}L_{\overline{m}}\label{eq-thm22}\\
&\;+c_{6}(\xi+2.4).\nonumber 
\end{align}
Let now $t\in\overline{S}_{m}$. It follows from Assumption~\ref{H-mod} that one can find a sequence $(\overline{s}_{k})_{k\ge0}$ in $S_{\overline{m}}$ converging to $\overline s$ and for each $k\in\N$ a sequence $(t_{k,j})_{j\ge0}$ in $S_{m}$ converging to $t$ such that $\lim_{j\to +\infty}\gPsi(\gx,t_{k,j},\overline{s}_{k})=\gPsi(\gx,t,\overline{s}_{k})$ for all $\gx\in\sX^{n}$. Then (\ref{eq-00}) implies that, for all $k\in\N$,
\begin{align*}
\lefteqn{-\overline \gPsi(\gX,t)-\pen(m)}\hspace{10mm}\\
=&\;\inf_{t'\in S}\!\left[\pen(t')-\gPsi(\gX,t,t')\right]-\pen(m)\\\le&\;\pen(\overline s_{k})-\gPsi(\gX,t,\overline s_{k})-\pen(m)\\\le&\;\pen(\overline m)+\lim_{j\to +\infty}\left[-\gPsi(\gX,t_{k,j},\overline{s}_{k})-\pen(t_{k,j})\right]\\
\le&\;\pen(\overline m)+c_{7} h^{2}(\gs,\overline s_{k})-\lim_{j\to +\infty}c_{8}nh^{2}(\gs,t_{k,j})+c_{6}(\xi+2.4)+\pen(\overline m)\\
&\;-\beta^{-1}L_{m}\\
=&\;c_{7} nh^{2}(\gs,\overline s_{k}) -c_{8}nh^{2}(\gs,t)+c_{6}(\xi+2.4)+2\pen(\overline m)-\beta^{-1}L_{m}.
\end{align*}
Letting $k$ tend to infinity, we conclude that
\begin{align*}
\lefteqn{-\left[\overline \gPsi(\gX,t)+\pen(m)\right]}\hspace{15mm}\\
&\le c_{7} nh^{2}(\gs,\overline s) -c_{8}nh^{2}(\gs,t)+c_{6}(\xi+2.4)+2\pen(\overline m)-\beta^{-1}L_{m}.
\end{align*}
Given $r>0$ and $t\in \overline{S}_{m}$ such that $\sqrt{n}h(\gs,t)\ge r$, we finally derive that
\[
-\!\left[\overline \gPsi(\gX,t)+\pen(m)\right]\le c_{7} nh^{2}(\gs,\overline s) -c_{8}nr^{2}+c_{6}(\xi+2.4)+2\pen(\overline m)-\beta^{-1}L_{m}.
\]
Therefore, for all $m\in\cM$ and $t\in \overline{S}_{m}$ such that $\sqrt{n}h(\gs,t)\ge r$,
\begin{align}
\lefteqn{\int_{\overline{S}_{m}\setminus \sB^{\overline{S}_{m}}(\gs,r)}\exp\cro{-\beta\left(\overline \gPsi(\gX,t)+\pen(m)\right)}d\pi_{m}(t)}\hspace{10mm}\nonumber\\
&\le \exp\left[\beta\left(c_{7} nh^{2}(\gs,\overline s)+2\pen(\overline m)-\beta^{-1}
L_{m}+c_{6}(\xi+2.4)-c_{8}nr^{2}\right)\right],
\label{eq-02}
\end{align} 
this integral being well-defined with $\sB^{\overline{S}_{m}}(\gs,r)\in\sbS_{m}$ by Proposition~\ref{prop-mes}.

Let us now define the function $I$ on $\overline S$ as follows: for $m\in\cM$ and $t\in\overline S_{m}$, let
\[
I(t)=\int_{\overline{S}_{m}}\exp\cro{-2\beta \pa{c_{7} nh^{2}(t,t')+\pen(m)}}d\pi_{m}(t').
\]
It follows from the definition~\eref{eq-eta} of $\overline \eta_{n}^{2}(t)$ that, whatever $\varepsilon>0$, one can find $r_{\varepsilon}>0$ such that
\[
2n\beta\overline \eta_{n}^{2}(t)\ge2n\beta c_{7}r_{\varepsilon}^{2}-\log\left(\gpi_{m}\left(\sB^{\overline S_{m}}(t,r_{\varepsilon})\right)\right)-\varepsilon,
\]
or, equivalently,
\[
\pi_{m}\left(\sB^{\overline S_{m}}(t,r_{\varepsilon})\right)\ge\exp\left[2n\beta\left(c_{7}r_{\varepsilon}^{2}-\overline \eta_{n}^{2}(t)\right)-\varepsilon\right].
\]
It then follows that
\begin{align*}
I(t)&\ge\int_{\sB^{\overline S_{m}}(t,r_{\varepsilon})}\exp\cro{-2\beta \pa{c_{7} nh^{2}(t,t')+\pen(m)}}d\pi_{m}(t')\\
&\ge\exp\left[-2\beta\pa{\pen(t)+c_{7}nr_{\varepsilon}^{2}}\right]\gpi_{m}\left(\sB^{\overline S_{m}}(t,r_{\varepsilon})\right)\\&\ge
\exp\left[-2\beta\pa{\pen(t)+c_{7}nr_{\varepsilon}^{2}}+2n\beta\left(c_{7}r_{\varepsilon}^{2}-\overline \eta_{n}^{2}(t)\right)-\varepsilon\right]
\end{align*}
and, since $\varepsilon$ is arbitrary, 
\begin{equation}
I(t)\ge\exp\left[-2\beta \pa{\pen(t)+ n\overline \eta_{n}^{2}(t)}\right]\quad\mbox{for all }t\in\overline{S}.
\label{eq-It}
\end{equation}
As a consequence of~\eref{eq-thm22} and~\eref{eq-It}, 
\begin{align}
\lefteqn{\log \int_{\overline{S}_{\overline m}}\exp\cro{-\beta\left(\overline \gPsi(\gX,t)+\pen(\overline{m})\right)}d\pi_{\overline m}(t)}\hspace{13mm}\nonumber \\
\ge&\; \beta\left[c_{8}nh^{2}(\gs,\overline{S})-2c_{7}nh^{2}(\gs,\overline s)-c_{6}(\xi+2.4)+\beta^{-1}L_{\overline{m}}\right]\nonumber\\ 
&\;+\log \int_{\overline{S}_{\overline m}}\exp\cro{-2\beta \left(c_{7}nh^{2}(\overline s,t)+\pen(\overline{m})\right)}d\pi_{\overline m}(t)\nonumber\\
=&\; L_{\overline{m}}+\beta\left[c_{8}nh^{2}(\gs,\overline{S})-2c_{7}nh^{2}(\gs,\overline s)-c_{6}(\xi+2.4)\right]+\log I(\overline s)\nonumber\\
\ge&\;\beta\left[c_{8}nh^{2}(\gs,\overline{S})-2c_{7}nh^{2}(\gs,\overline s)-2\left(\pen(\overline s)+n\overline \eta_{n}^{2}(\overline s)\right)-c_{6}(\xi+2.4)\right]\label{eq-03}\\&
\:+ L_{\overline{m}}.\nonumber
\end{align}
Combining~\eref{eq-02} and~\eref{eq-03} and using (\ref{eq-Lm}), we derive from (\ref{def-prior-sm}) that
\begin{align*}
\lefteqn{\overline{\gpi}_{\!\gX}\left(\overline{S}\setminus \sB^{\overline{S}}(\gs,r)\right)}\hspace{12mm}\\
=\:&
{\dps{\Delta\sum_{m\in\cM}\int_{\overline{S}_{m}\setminus \sB^{\overline{S}_{m}}(\gs,r)}}\exp\cro{-\beta\left(\overline \gPsi(\gX,t)+\pen(m)\right)}d\gpi_{m}(t)\over\dps{\Delta\sum_{m'\in\cM}\int_{\overline{S}_{m'}}}\exp\cro{-\beta\left(\overline \gPsi(\gX,t)+\pen(m')\right)}d\pi_{m'}(t)}
\\ \le\:&\frac{\dps{\sum_{m\in\cM}{\int_{\overline{S}_{m}\setminus \sB^{\overline{S}_{m}}(\gs,r)}}\exp\cro{-\beta\left(\overline \gPsi(\gX,t)+\pen(m)\right)}d\pi_{m}(t)}}{\dps{\int_{\overline{S}_{\overline m}}}\exp\cro{-\beta\left(\overline \gPsi(\gX,t)+\pen(\overline{m})\right)}d\pi_{\overline m}(t)}\\
\le\:&\sum_{m\in\cM}e^{-L_{m}}\exp\left[\beta\left(c_{7} nh^{2}(\gs,\overline s)+2\pen(\overline s)+c_{6}(\xi+2.4)-c_{8}nr^{2}\right)\right]\\&\times
\exp\left[-\ L_{\overline m}-\beta\left(c_{8}nh^{2}(\gs,\overline{S})-2c_{7}nh^{2}(\gs,\overline s)-2[\pen(\overline s)+n\overline \eta_{n}^{2}(\overline s)]\right)\right]\\&\times\exp\left[\beta c_{6}(\xi+2.4)\right]\\
\le\:&\exp\left[\beta\left(c_{7} nh^{2}(\gs,\overline s)+2\pen(\overline s)+c_{6}(\xi+2.4)-c_{8}nr^{2}\right)\right]\\
&\times
\exp\left[-L_{\overline m}-\beta\left(c_{8}nh^{2}(\gs,\overline{S})-2c_{7}nh^{2}(\gs,\overline s)-2[\pen(\overline s)+n\overline \eta_{n}^{2}(\overline s)]\right)\right]\\&\times\exp\left[\beta c_{6}(\xi+2.4)\right]=\exp[-\beta H]
\end{align*}
with
\begin{align*}
H=&\:c_{8}n\!\left[h^{2}\!\left(\gs,\overline{S}\right)+r^{2}\right]-3c_{7}nh^{2}(\gs,\overline s)-2\!\cro{2\pen(\overline s)+n\overline \eta_{n}^{2}(\overline s)}\\
&\:-2c_{6}(\xi+2.4)+\beta^{-1}L_{\overline m}.
\end{align*}
Consequently, $\overline{\gpi}_{\!\gX}\left(\overline{S}\setminus \sB^{\overline{S}}(\gs,r)\right)\le e^{-\xi'}$ provided that  $r\ge r_{0}$ with
\begin{align*}
c_{8}nr_{0}^{2}&=3c_{7}nh^{2}(\gs,\overline s)-c_{8}nh^{2}(\gs,\overline{S})+2\cro{2\pen(\overline s)+n\overline \eta_{n}^{2}(\overline s)}\\
&\quad -\beta^{-1}(L_{\overline m}-\xi')+2c_{6}(\xi+2.4).
\end{align*}
The value of $\overline{r}_{n}^{2}$ follows from the fact that $\overline s$ is arbitrary in $\overline{S}$ and by (\ref{eq-omm'}),
\[
\P_{\gs}\left[\Omega_{\xi}^{c}\right]\le
\sum_{m,m'\in\cM}e^{-(L_{m}+L_{m'}+\xi)}=e^{-\xi}\left[\sum_{m\in\cM}e^{-L_{m}}\right]^{2}=e^{-\xi}.
\]

\section{Proofs of Proposition~\ref{prop-ass1}, Theorem \ref{main} and Proposition~\ref{Prop-int}}\label{Sect-12}

\subsection{Proof of Proposition~\ref{prop-ass1}}
The pointwise convergence of $t_{k}$ to $t$ when $k\rightarrow+\infty$ and Scheff\'e's Lemma imply the convergence of $t_{k}$ to $t$ in $\L_{1}(\mu)$-distance, hence in Hellinger distance. As to the limit of $\gPsi(\gx,t_{k},t')$ when $k\rightarrow+\infty$, it follows trivially from~\eref{eq-t-t_k}. 

\subsection{Preliminary results}\label{Sect-ProofPR}
We shall use Proposition~45 of BBS together with the remark following it and extending the result to a countable set $T$. We recall this result for further reference.
%
\begin{prop}\label{talagrand}
Let $T$ be some countable set, $U_{1},\ldots,U_{n}$ be independent centred random vectors with values 
in $\R^{T}$ and $Z=\sup_{t\in T}\left|\sum_{i=1}^{n}U_{i,t}\right|$. If for some positive numbers $b$ and $v$, 
\[
\max_{i=1,\ldots,n}|U_{i,t}|\le b\qquad\mbox{and}\qquad
\sum_{i=1}^{n}\E\left[U^2_{i,t}\right]\le v^{2}\ \quad\mbox{for all }t\in T,
\] 
then, for all positive $c$ and $x$,
\begin{equation}\label{massart3}
\P\left[Z\le(1+c){\mathbb E}(Z)+(8b)^{-1}cv^2+2\left(1+8c^{-1}\right)bx\right]\ge1-e^{-x}.
\end{equation}
\end{prop}
We shall also use the following properties of the function $\psi$ that have been established in BB (Proposition~3 with $\psi=\psi_{2}$).
\begin{prop}
Whatever $\gs=(s_{1},\ldots,s_{n})\in\sL^{n}$  and $t,t'\in\sL$, 
\begin{equation}\label{eq-esp-psi}
\frac{1}{n}\sum_{i=1}^{n}\int_{\sX}\psi\pa{\sqrt{t'\over t}}dP_{s_{i}}\le4h^{2}(\gs,t)-0.375h^{2}(\gs,t');
\end{equation}
\begin{equation}\label{eq-var-psi}
\frac{1}{n}\sum_{i=1}^{n}\int_{\sX}\psi^{2}\pa{\sqrt{t'\over t}}dP_{s_{i}}\le3\sqrt{2}\cro{h^{2}(\gs,t)+h^{2}(\gs,t')}.
\end{equation}
\end{prop}
\subsection{Proof of Theorem~\ref{main}}
The mapping $t\mapsto \gPsi(\gX(\omega),t,t')$ is measurable on $(\overline S,\sS)$ for all $t'\in S$ by Proposition~\ref{prop-mes}. By the same argument based on (\ref{eq-ass1}) that we already used for the proof of Theorem~\ref{thm-main}, it is enough to show that (\ref{eq-main}) holds for all $t\in S$ and all $\omega$ belonging to some measurable set $\Omega_{\xi}$ with $\P_{\gs}(\Omega_{\xi})\ge1-e^{-\xi}$, which we shall now do.

In order to do this, we fix $\xi>0$, $\epsilon>0$  and set $\eps(\gs)$ for $\eps_{n}^{\overline{S}}(\gs)$. Let $\tau>1,\alpha>0,c>0,q>1$ be numbers to be chosen later on and set for all $j,j'\in\N\cup\{-1\}$,
\[
r_j^{2}=q^{j+1}\left[\eps^{2}(\gs)+{\tau\over 2n}\left(\xi+\alpha\right)\right],\quad
x_{j,j'}=\frac{r_{j}^{2}+r_{j'}^{2}}{\tau}>\frac{\xi}{n}+{\alpha q\over 2n} \left(q^{j}+q^{j'}\right),
\]
\[
\sB^{	S}_{j}=\sB^{S}(\gs,r_{j})\quad\text{ and }\quad Z_{j,j'}(\gX)=\sup_{(t,t')\in\sB^{\gS}_{j}\times\sB^{\gS}_{j'}}\left|\bsZ(\gX,t,t')\right|.
\]
Since $t$ and $t'$ are measurable functions on $\sX$, $\gx\mapsto  \gPsi(\gx,t,t')$ is measurable on $(\sX^{n},\sA\on)$ for all $t,t'\in S\subset \sL$ and $\bsZ(\gX,t,t')$ as well. Since the set $\sB^{S}_{j}\times \sB^{S}_{j'}\subset S^{2}$ is countable for all $j,j'\in \N\cup\{-1\}$, $Z_{j,j'}(\gX)$ is measurable on $(\Omega,\Xi)$ as the supremum of countably many random variables. Moreover, we may apply Proposition~\ref{talagrand} to this supremum, taking  for $T$ the countable set $\sB^{S}_{j}\times \sB^{S}_{j'}$ and for $(t,t')\in\sB^{S}_{j}\times \sB^{S}_{j'}$,
\begin{equation}\label{def-U}
U_{i,t,t'}=\psi\left(\sqrt{t'\over t}(X_{i})\right)-
\E_{\gs}\left[\psi\left(\sqrt{t'\over t}(X_{i})\right)\right]\quad\mbox{for }i=1,\ldots,n.
\end{equation}
For such a choice, the assumptions of Proposition~\ref{talagrand} are met with $b=2$ (since $\psi$ is bounded by $1$) and $v^{2}=\left(3\sqrt{2}\right)n\left(r_{j}^{2}+r_{j'}^{2}\right)$ (by the definition of the sets $\sB^{S}_{j}$ and~\eref{eq-var-psi}). We derive from~\eref{massart3} with $x=nx_{j,j'}$ that, on a measurable set $\Omega_{j,j',\xi}\subset \Omega$ whose $\P_{\gs}$-probability is at least $1-e^{-nx_{j,j'}}$,
\begin{equation}
Z_{j,j'}(\gX)\le(1+c)\E_{\gs}\left[Z_{j,j'}(\gX)\right]+ \frac{3\sqrt{2}}{16}cn\left(r_{j}^{2}+r_{j'}^{2}\right)+4n\left(1+\frac{8}{c}\right)x_{j,j'}.\label{Eq-pr1}
\end{equation}
For $k\in\{j,j'\}$, $\sB^{S}_{k}\subset \sB^{S}(\gs,r_{j\vee j'})$ and, since $r_{j\vee j'}>\eps(\gs)$, it follows from  the definition of $\eps(\gs)$ that, 
\begin{align*}
\E_{\gs}\left[Z_{j,j'}(\gX)\right]&\le \E_{\gs}\left[\sup_{t,t'\in \sB^{S}_{j}(\gs,r_{j\vee j'})}\left|\bsZ(\gX,t,t')\right|\right]=\gw^{\overline{S}}(\gs,r_{j\vee j'})
\\ &\le 6c_0^{-1}nr_{j\vee j'}^{2}\le 6c_{0}^{-1}n\left(r_{j}^{2}+r_{j'}^{2}\right),
\end{align*}
and (\ref{Eq-pr1}) becomes
\begin{equation}\label{eq-pr2}
\bsZ(\gX,t,t')\le Z_{j,j'}(\gX)\le n\varpi \left(r_{j}^{2}+r_{j'}^{2}\right)\ \ \mbox{for all}\ (t,t')\in \sB^{S}_{j}\times \sB^{S}_{j'}
\end{equation}
with 
\[
\varpi=\varpi(c,c_{0},\tau)=c_{0}^{-1}\left[6(1+c)+{3\sqrt{2}c_{0}c\over16}+{4c_{0}(1+8c^{-1})\over \tau}\right].
\]
Let us now set
\[
B^{S}_{-1}=\sB^{S}_{-1}\qquad\mbox{and}\qquad B^{S}_{j}=\sB^{S}_{j}\setminus\sB^{S}_{j-1}
\quad\mbox{for }j\ge0.
\]
If $j\ge0$ and $t\in B^{S}_{j}$, then $h^{2}(\gs,t)>r_{j-1}^{2}=q^{-1}r_{j}^{2}$ and it follows that
$r_{j}^{2}\le r_{-1}^{2}+ q h^{2}(\gs,t)$ for all $j\ge-1$. Therefore (\ref{eq-pr2}) implies that, on $\Omega_{j,j',\xi}$,
\[
n^{-1}\bsZ(\gX,t,t')\le\varpi\left[2r_{0}^{2}+q\left(h^{2}(\gs,t)+h^{2}(\gs,t')\right)\right]\quad
\mbox{for all }(t,t')\in B^{S}_{j}\times B^{S}_{j'}.
\]
This bound, together with~\eref{eq-esp-psi}, leads for all $(t,t')\in B^{S}_{j}\times B^{S}_{j'}$ to
\begin{align*}
\frac{1}{n}\gPsi(\gX,t,t')&\le \frac{1}{n}\E_{\gs}\left[\gPsi(\gX,t,t')\right] + \varpi \left[2r_{0}^{2}+q\left(h^{2}(\gs,t)+h^{2}(\gs,t')\right)\right]\\
&= \frac{1}{n}\sum_{i=1}^{n}\int_{\sX}\psi\left(\sqrt{t'\over t}\right)dP_{s_{i}}+ \varpi \left[2r_{0}^{2}+qh^{2}(\gs,t)+qh^{2}(\gs,t')\right]\\
&\le (4+\varpi q)h^{2}(\gs,t) -(0.375-\varpi q)h^{2}(\gs,t') + 2\varpi r_{0}^{2}.
\end{align*}
We recall that $c_{0}=10^{3}$ by (\ref{eq-constantes}) and choose $q=5/4$, $\alpha=2.4$, 
\[
c^{-1}=6\pa{1+\frac{c_{0}\sqrt{2}}{32}}\quad\mbox{ and }\quad
\tau=4c_{0}(1+8c^{-1}),
\]
so that
\[
6c+{3\sqrt{2}c_{0}c\over 16}={4c_{0}(1+8c^{-1})\over\tau}=1,\qquad\varpi={1\over 5^{3}}\qquad\mbox{and}\qquad\varpi q={1\over 100}.
\]
It follows that on $\Omega_{j,j',\xi}$, the probability of which is at least $1-e^{-nx_{j,j'}}$, for all $(t,t')\in B^{S}_{j}\times B^{S}_{j'}$
\begin{equation}\label{eq-concl}
\frac{1}{n}\gPsi(\gX,t,t')\le c_{7}h^{2}(\gs,t) -c_{8}h^{2}(\gs,t') + c_{5}\eps^{2}(\gs)+ c_{6}\frac{\xi+2.4}{n},
\end{equation}
with $c_{7}=4.01$, $c_{8}=0.365$, $c_{5}=2\times 5^{-3}$ and $c_{6}=7.10^{4}$ as indicated in (\ref{eq-constantes}), since $\tau\varpi<7.10^{4}$. Since the sets $B^{S}_{j}\times B^{S}_{j'}$ with $j,j'\in\N\cup\{-1\}$ provide a partition of $S^{2}$ and, for all $j,j'\in\N\cup\{-1\}$, $nx_{j,j'}\ge\xi+(\alpha q/2)(q^{j}+q^{j'})$, inequality~\eref{eq-concl} holds for all $(t,t')\in S^{2}$ on the measurable set $\Omega_{\xi}$ with
\[
\Omega_{\xi}^{c}=\bigcup_{j,j'\in \N\cup\{-1\}}\Omega_{j,j',\xi}^{c}\qquad\text{and}\qquad
\P_{\gs}\left[\Omega_{j,j',\xi}^{c}\right]\le e^{-nx_{j,j'}}.
\]
We conclude that $\P_{\gs}\left[\Omega_{\xi}\right]>1-e^{-\xi}$ since the probability of $\Omega_{\xi}^{c}$ is at most
\begin{align*}
\sum_{j\ge -1}\sum_{j'\ge -1}e^{-nx_{j,j'}}\le e^{-\xi}\!\left(\sum_{j\ge 0}e^{-(\alpha/2)q^{j}}\right)^{2}=e^{-\xi}\!\left(\sum_{j\ge 0}e^{-1.2\times(5/4)^{j}}\right)^{2}<e^{-\xi}.
\end{align*}
%

\subsection{Proof of Proposition~\ref{Prop-int}}
It follows from Proposition~\ref{prop-mes} that the two integrals appearing in the left-hand side of~\eref{eq-RapInt} are well-defined. Applying (\ref{Eq-Gamma}) iteratively, we get, for all $k\in\N$, 
\begin{align*}
\gpi\left(\sB^{\overline{S}}(t,2^{k+1}r_{0})\right)&\le\prod_{j=0}^k\exp\left[(3a/8)4^{j}nr^{2}_{0}\right]\gpi\left(\sB^{\overline{S}}(t,r_{0})\right)\\&\le \exp\left[\left(4^{k+1}/3\right)(3a/8)nr^{2}_{0}\right]\gpi\left(\sB^{\overline{S}}(t,r_{0})\right).
\end{align*}
For $k=J+j$ with $j\in\N$, this leads, since $2^{J}r_{0}=r$, to
\[
\gpi\left(\sB^{\overline{S}}(t,2^{j+1}r)\right)\le\exp\left[4^{j}\left(a n r^{2}/2\right)\right]\gpi\left(\sB^{\overline{S}}(t,r_{0})\right).
\]
Therefore
\begin{align*}
\lefteqn{\int_{\overline{S}\setminus \sB^{\overline{S}}(t,r)}\exp\!\left[-an h^{2}(t,t')\right]\!d\gpi(t')}\:\\
&= \sum_{j\ge 0}\int_{\sB^{\overline{S}}(t,2^{j+1}r)\setminus \sB^{\overline{S}}(t,2^{j}r)}
\exp\!\left[-an h^{2}(t,t')\right]\!d\gpi(t')\\
&\le \sum_{j\ge 0}\exp\!\left[-a 4^{j}n r^{2}\right]\gpi\!\left(\sB^{\overline{S}}(t,2^{j+1}r)\right)
\le \sum_{j\ge 0}\exp\!\left[-\frac{4^{j}an r^{2}}{2}\right]\gpi\!\left(\sB^{\overline{S}}(t,r_0)\right)
\end{align*}
and, since $anr^2\ge1$,
\[
\sum_{j\ge 0}\exp\!\left[-\frac{4^{j}anr^{2}}{2}\right]=\exp\!\left[-\frac{anr^{2}}{2}\right]
\sum_{j\ge 0}\exp\!\left[-\frac{4^{j}-1}{2}\right]<e^{1/4}\exp\left[-\frac{anr^{2}}{2}\right].
\]
Besides 
\begin{align*}
\int_{\overline{S}}\exp\left[-bnh^{2}(t,t')\right]d\gpi(t')&\ge\int_{\sB^{\overline{S}}(t,r_0)}\exp\left[-bnh^{2}(t,t')\right]d\gpi(t')\\&\ge\exp\left[-bnr_{0}^{2}\right]
\gpi\left(\sB^{\overline{S}}(t,r_0)\right)\\&=\exp\left[-b4^{-J}nr^{2}\right]\gpi\left(\sB^{\overline{S}}(t,r_0)\right).
\end{align*}
Putting everything together we finally get, since $4^{J-1}\ge b/a$,
\begin{align*}
{\dps{\int_{\overline{S}\setminus \sB^{\overline{S}}(t,r)}}\exp\left[-anh^{2}(t,t')\right]d\gpi(t')\over\dps{\int_{\overline{S}}}\exp\left[-bnh^{2}(t,t')\right]d\gpi(t')}&\le e^{1/4}\exp\left[-nr^{2}\left({a\over2}-4^{-J}b\right)\right]\\
&\le e^{1/4}\exp\left[-{anr^{2}/4}\right].
\end{align*}
%

\section{Proof of Theorem~\ref{thm-B}}
Let us fix $\xi> \log4$ (so that $4e^{-\xi}< 1$), set $\lambda_{n}=(\xi+\log n)^{1/2}$ and introduce in the course of the proof various functions ($n_{j}$, $\kappa_j$, $k_j$, $C_{j}$, $K$, $\Gamma$ and $J_{0}$) of $\xi$, of the parameters involved in our assumptions, namely $A_1$, $A_2=A_{2}(\gvartheta)$, $B$, $\overline{\gamma}$ and $d$ and of the various universal constants $c_{j}$ and $\gamma$ defined in (\ref{eq-constantes}), but not on $n$. For simplicity we shall call them ``constants". Throughout the proof we shall assume that $n\ge n_{1}$, starting with $n_{1}\ge e^{\xi}>4$ (so that $\log n>1$) and increasing $n_{1}$ whenever necessary in the course of this proof.

\subsection{Some consequences of our assumptions}\label{PV-1}
\paragraph{\bf Consequence 1} 
It follows from (\ref{metric-h}) that the Euclidean and Hellinger distances on $\bs{\Theta}$ are equivalent which means that, for all $r>0$ and $\gtheta\in \bs{\Theta}$
\begin{equation}
\cB(\gtheta,A_{3}^{-1}r)\subset\cB_{h}(\gtheta,r)= \{\gtheta'\in \bs{\Theta}\,|\, h(\gtheta,\gtheta')\le r\}\subset \cB(\gtheta,A_{2}^{-1}r),
\label{eq-eqdis}
\end{equation}
and the definition of $\pi$ implies that 
\[
\pi\pa{\sB^{\overline{S}}\left(t_\gtheta,r\right)}=\pi\pa{\left\{\left.t_{\gtheta'}\in \overline{S}\,\right|\,h(t_{\gtheta},t_{\gtheta'})\le r\right\}}=\nu\pa{\cB_{h}(\gtheta,r)}.
\]
Let $k_{1}\in\N\et$ such that $2^{k_{1}}\ge A_{3}A_{2}^{-1}>2^{k_{1}-1}$. It follows from  Assumption~\ref{A-prior} and (\ref{eq-eqdis}) that, for all $\gtheta\in\bs{\Theta}$, $r>0$ and $k\in\N$, hence $k+k_{1}\in\N\et$,
\begin{align*}
\pi\left(\sB^{\overline{S}}\left(t_\gtheta, 2^{k}r\right)\right)&= \nu\pa{\cB_{h}\left(\gtheta,2^{k}r\right)}\le\nu\pa{\cB\pa{\gtheta,2^{k+k_{1}}\left[2^{k_{1}}A_{2}\right]^{-1}r}}\\
&\le \exp\left[B\overline{\gamma}^{k+k_{1}}\right]\nu\pa{\cB\pa{\gtheta,\left[2^{k_{1}}A_{2}\right]^{-1}r}}\\
&\le \exp\left[\left(B\overline \gamma^{k_{1}}\right)\overline{\gamma}^{k}\right]
\nu\pa{\cB\pa{\gtheta,A_{3}^{-1}r}}
\\&\le \exp\left[\left(B\overline \gamma^{k_{1}}\right)\overline{\gamma}^{k}\right]\nu\pa{ \cB_{h}(\gtheta,r)\st}.
\end{align*}
Setting $\kappa_{1}=B\overline \gamma^{k_{1}}\ge1$ we deduce that, whatever $\gtheta\in \bs{\Theta}$ and $r>0$,
\begin{equation}\label{eq-typeta}
\pi\left(\sB^{\overline{S}}\left(t_\gtheta, 2^{k}r\right)\right)\le \exp\left[\kappa_{1}\overline{\gamma}^{k}\right]\pi\left(\sB^{\overline{S}}\left(t_\gtheta,r\right)\right)\quad\text{for all }k\in\N\et.
\end{equation}
%

\paragraph{\bf Consequence 2} 
We now want to apply Theorem~2 of Birg\'e~\citeyearpar{Lucien-Bayes} to the posterior distribution $\pi_{\!\gX}^{L}$ and we therefore have to check Assumptions~1 and 2 of that paper. 
The equivalence of the Euclidean and Hellinger distances and the classical metric properties of the Euclidean space $\R^{d}$ imply that Assumption~1 is satisfied for some constant function $D(x)=\overline{D}$ depending only on $A_{2}$, $A_{3}$ and $d$. As to Assumption~2, it follows from (\ref{eq-typeta}) with $\gamma=\overline{\gamma}$ and $\beta(j)=\kappa_{1}$ for all $j\ge3$ as required. We derive from this Theorem~2 with $\overline{\beta}=\kappa_{1}$, $\kappa=0$ and $c=e^{-\xi}<1/4$ the following result. There exist two constants $n_{0}$ and $k_0$ (with $2^{k_{0}-4}\ge e^{\xi}$) which are positive integers and, for all $n\ge n_{0}$, a set $\Omega(n,\xi)$ of probability $\P_{\gs}[\Omega(n,\xi)]\ge1-e^{-\xi}$, such that, if $\omega\in\Omega(n,\xi)$ and $k\ge k_{0}$,
\begin{equation}
\pi_{\!\gX(\omega)}^{L}\!\left(\sB^{\overline{S}}\left(s,2^{k}e^{-\xi}/\sqrt{n}\right)\right)
\ge1-1.05\exp\left[-4^{k-4}e^{-2\xi}\right].
\label{eq-n0k0}
\end{equation}
Increasing $n_{1}$ if necessary, we shall assume that $n_{1}\ge n_{0}$ and $\log n_{1}\ge4^{k_{0}-4}e^{-2\xi}-\log(1.05)$. For $n\ge n_{1}$ let $k_{n}$ be the largest integer such that $4^{k_{n}-4}e^{-2\xi}-\log(1.05)\le\log n$. Then $k_{n}\ge k_{0}$, $1.05\exp\left[-4^{k_{n}-4}e^{-2\xi}\right]\le n^{-1}$ and $2^{k_{n}}e^{-\xi}<16\sqrt{\log n}<16\lambda_{n}$. Therefore, by (\ref{eq-n0k0}), for $\omega\in\Omega(n,\xi)$,
\[
\pi_{\!\gX(\omega)}^{L}\!\left(\sB^{\overline{S}}\left(s,16\lambda_{n}/\sqrt{n}\right)\right)
\ge\pi_{\!\gX(\omega)}^{L}\!\left(\sB^{\overline{S}}\left(s,2^{k_{n}}e^{-\xi}/\sqrt{n}\right)\right)\ge1-n^{-1},
\]
hence, for $n\ge n_{1}$,
\begin{equation}
\P_{s}\left[\pi_{\!\gX}^{L}\!\left(\sB^{\overline{S}}\left(s,16\lambda_{n}/\sqrt{n}\right)\right)
\ge1-n^{-1}\right]\ge1-e^{-\xi}.
\label{Eq-J1}
\end{equation}
%
\paragraph{\bf Consequence 3} 
The quantities $\eta_{n}^{\overline{S},\pi}(t_{\gtheta})$ and $\eps_{n}^{\overline{S}}(\gs)$ involved in the performance of our $\rho$-posterior distribution as described in Theorem~\ref{thm-main} are controlled as follows (with $\gamma$ given by (\ref{eq-constantes})). 

\begin{prop}\label{prop-eta-eps}
Under Assumptions~\ref{A-MLE} and~\ref{A-prior}, 
\[
\eta_{n}^{\overline{S},\gpi}(t_{\gtheta})\le\kappa_{2}/ \sqrt{n}\mbox{ for all }\gtheta\in \bs{\Theta}\qquad\mbox{and}\qquad \eps_{n}^{\overline{S}}(\gs)\le \kappa_{3}/\sqrt{n}.
\]
\end{prop}
\begin{proof}
It follows from \eref{eq-typeta} with $k=1$ that
\[
\gpi\left(\sB^{\overline{S}}\left(t_\gtheta, 2r\right)\right)\le \exp\left[\kappa_{1}\overline{\gamma}\right]\pi\left(\sB^{\overline{S}}\left(t_\gtheta,r\right)\right)\quad\mbox{for all }r>0.
\]
Therefore, for all $\gtheta\in \bs{\Theta}$ and $r\ge \sqrt{\kappa_{1}\overline \gamma/(\gamma n)}$,
\[
\pi\left(\sB^{\overline{S}}\left(t_\gtheta, 2r\right)\right)\le \exp\left[\gamma n r^{2}\right]\pi\left(\sB^{\overline{S}}\left(t_\gtheta,r\right)\right)\;
\]
and (\ref{Eq-eta}) implies that $\eta_{n}^{\overline{S},\gpi}(t_{\gtheta})\le\sqrt{\kappa_{1}\overline \gamma/(\gamma n)}$ for all $\gtheta\in \bs{\Theta}$.

Let us now turn to $\eps_{n}^{\overline{S}}(\gs)$. We set, for $0<y\le1$,
\[
\sF_{y}=\left\{\psi\left(\sqrt{t_{\gtheta'}\over t_{\gtheta}}\right),\ t_{\gtheta},t_{\gtheta'}\in\sB^{\overline{S}}(s,y)\right\}.
\]
Our aim is to control the entropy of $\sF_{y}$ in view of bounding 
\begin{equation}
\gw^{\overline{S}}(\gs,y)=\E\cro{\sup_{f\in\sF_{y}}\ab{\sum_{i=1}^{n}f(X_{i})-\E\cro{f(X_{i})}}}
\label{eq-wsbar}
\end{equation}
from above. Since $h$ is bounded by 1, $\sB^{\overline{S}}(s,y)=\sB^{\overline{S}}(s,1)$ for all $y\ge 1$ and it is therefore enough to bound this quantity for $y\in (0,1]$ only, which we shall now do. 

It follows from~\eref{metric-h} that, if $t_{\gtheta}$ belongs to $\sB^{\overline{S}}(s,y)$ or equivalently if $h(s,t_{\gtheta})=h(t_{\gvartheta},t_{\gtheta})\le y$,  then $\ab{\gvartheta-\gtheta}\le A_{2}^{-1}y$. Therefore 
\[
\left\{(\gtheta,\gtheta')\in \bs{\Theta}^{2}\,\left|\,\gt_{\gtheta},\gt_{\gtheta'}\in\sB^{\overline{S}}(s,y)\right.\right\}\subset \cB_{2}(y)
\]
where $\cB_{2}(y)$ denotes the Euclidean ball in $\bs{\Theta}^{2}\subset \R^{2d}$ centred at $(\gvartheta,\gvartheta)$ with radius $\sqrt{2}A_{2}^{-1}y$.

Let $x_{1},\ldots,x_{n}$ belong to $\sX$, $P_{n}=n^{-1}\sum_{i=1}^{n}\delta_{x_{i}}$ be the corresponding empirical measure and $\|\cdot\|_{2}$ be the norm in $\L_{2}(\sX,P_{n})$. 
Since $\psi(1/u)=-\psi(u)$ for all $u\ge 0$ and $\psi$ is Lipschitz on $[0,+\infty)$ (with Lipschitz constant $2$), we deduce from Assumption~\ref{A-MLE}-$\ref{H4-iii}$ that, for all $\gtheta,\gtheta',\overline \gtheta,\overline \gtheta'\in \bs{\Theta}$,
\begin{align*}
\lefteqn{\norm{\psi\left(\sqrt{t_{\gtheta'}\over t_{\gtheta}}\right)-\psi\left(\sqrt{t_{\overline \gtheta'}\over t_{\overline \gtheta}}\right)}_{2}^{2}}\hspace{35mm}\\
=\:&{1\over n}\sum_{i=1}^{n}\ab{\psi\left(\sqrt{t_{\gtheta'}\over t_{\gtheta}}(x_{i})\right)-\psi\left(\sqrt{t_{\overline \gtheta'}\over t_{\overline \gtheta}}(x_{i})\right)}^{2}\\
\le\:&{2\over n}\left[\sum_{i=1}^{n}\ab{\psi\left(\sqrt{t_{\gtheta'}\over t_{\gtheta}}(x_{i})\right)-\psi\left(\sqrt{t_{\overline \gtheta'}\over t_{\gtheta}}(x_{i})\right)}^{2}\right.\\
&\hspace{6mm}\left. +\sum_{i=1}^{n}\ab{\psi\left(\sqrt{t_{\overline \gtheta'}\over t_{\gtheta}}(x_{i})\right)-\psi\left(\sqrt{t_{\overline \gtheta'}\over t_{\overline \gtheta}}(x_{i})\right)}^{2}\right]\\
=\:&{2\over n}\left[\sum_{i=1}^{n}\ab{\psi\left(\sqrt{t_{\gtheta'}\over t_{\gtheta}}(x_{i})\right)-\psi\left(\sqrt{t_{\overline \gtheta'}\over t_{\gtheta}}(x_{i})\right)}^{2}\right.\\
&\hspace{6mm}\left. +\sum_{i=1}^{n}\ab{\psi\left(\sqrt{t_{\gtheta}\over t_{\overline \gtheta'}}(x_{i})\right)-\psi\left(\sqrt{t_{\overline \gtheta}\over t_{\overline \gtheta'}}(x_{i})\right)}^{2}\right]\\
\le\:&8\cro{\norm{\sqrt{t_{\gtheta'}\over t_{\gtheta}}-\sqrt{t_{\overline \gtheta'}\over t_{\gtheta}}}_{\infty}^{2}+\norm{\sqrt{t_{\gtheta}\over t_{\overline \gtheta'}}-\sqrt{t_{\overline \gtheta}\over t_{\overline \gtheta'}}}_{\infty}^{2}}\\
\le\:& 8A_{1}^{2}\cro{\ab{\gtheta'-\overline \gtheta'}^{2}+\ab{\gtheta-\overline \gtheta}^{2}}.
\end{align*}
This implies that the minimal number of closed balls of radius $\eps\in (0,1)$ for the $\L_{2}(\sX,P_{n})$-distance that are necessary to cover $\sF_{y}$ is not larger than the minimal number $N$ of closed balls of radius $\eps/(2\sqrt{2}A_{1})$ for the Euclidean distance on $\R^{2d}$ that are necessary to cover $\cB_{2}(y)$. 
Let us now recall the following classical result :
%
\begin{lem}\label{lem-Euclidball}
The minimal number of closed balls of radius $x$ that are necessary to cover a ball of radius $r>x$ in the Euclidean space $\R^{k}$ is not larger than $(1+2rx^{-1})^{k}$.
\end{lem}
It follows from this lemma that
\[
N\le\cro{1+8A_{1}A_{2}^{-1}y/\eps}^{2d}<\cro{e+8A_{1}A_{2}^{-1}y/\eps}^{2d}.
\]
This remains true for $\varepsilon\ge1$ since
\[
\norm{\psi\left(\sqrt{t_{\gtheta'}\over t_{\gtheta}}\right)-\psi\left(\sqrt{s\over s}\right)}_{2}^{2}=
{1\over n}\sum_{i=1}^{n}\psi^{2}\left(\sqrt{t_{\gtheta'}\over t_{\gtheta}}(x_{i})\right)\le1,
\]
which shows that $\sF_{y}$ is included in a ball of radius one centered at $0=\psi\left(\sqrt{s/s}\right)$.
Besides, the function $\psi$ satisfies~\eref{eq-var-psi}. Therefore, since $h(s,t_{\gtheta})\le y$ for all $t_{\gtheta}\in\sB^{\overline{S}}(s,y)$, $\E\cro{f^{2}(X_{1})}\le 6\sqrt{2}y^{2}<9y^{2}$ for all $f\in\sF_{y}$. This implies that the inequalities (109) and (110) of BBS are satisfied with
\[
v^{2}= 9ny^{2}\qquad\mbox{and}\qquad\overline{\sH}(x)=2d\log(e+Kx)\quad\mbox{with }K=8A_{1}A_{2}^{-1}y.
\]
To apply Lemma~49 of BBS, we bound, using an integration by parts and the fact that $\log(e+Kx)>1$,
\begin{align*}
\frac{\dps{\int_{x}^{+\infty}}u^{-2}\sqrt{\overline{\sH}(u)}\,du}{x^{-1}\sqrt{\overline{\sH}(x)}}&=
\frac{\dps{\int_{x}^{+\infty}}u^{-2}\sqrt{\log(e+Ku)}\,du}{x^{-1}\sqrt{\log(e+Kx)}}\\&=
1+\frac{x}{\sqrt{\log(e+Kx)}}\int_{x}^{+\infty}\frac{K\,du}{2u(e+Ku)\sqrt{\log(e+Ku)}}\\&<1+x\int_{x}^{+\infty}\frac{du}{2u^2}=\frac{3}{2}.
\end{align*}
Since Assumption~10 of BBS is satisfied with $L=3/2$, we derive from Lemma~49 of BBS that $\gw^{\overline{S}}(\gs,y)\le C_{1}\cro{vL\sqrt{H}+L^{2}H}$ with 
\begin{align*}
H&=2d\log\left(e+4A_{1}A_{2}^{-1}y\left[\frac{1}{3y}\bigvee1\right]\right)= 2d\log\left(e+4A_{1}A_{2}^{-1}\left[\frac{1}{3}\bigvee {y}\right]\right)\\
&\le 2d\log\left(e+4A_{1}A_{2}^{-1}\right)\quad\mbox{ for }0<y\le 1.
\end{align*}
This implies that for all $y>0$, $\gw^{\overline{S}}(\gs,y)\le C_{2}\left(y\sqrt{nd}+d\right)$ and it then follows from \eref{def-en} that $\eps_{n}^{\overline{S}}(\gs)\le C_{3}\sqrt{d/n}$ . 
\end{proof}
%

\paragraph{\bf Consequence 4} 
Applying Theorem~\ref{thm-main} with the bounds of Proposition~\ref{prop-eta-eps} leads, since $\xi>\log4$, to
\[
\overline r_{n}\sqrt{n}\le c_{2}\kappa_{2}+c_{3}\kappa_{3}+c_{4}\sqrt{\xi+\log n+2.61}\le\kappa_{0}\lambda_{n},
\]
so that
\begin{equation}
\P_{s}\left[\pi_{\!\gX}\pa{\sB^{\overline{S}}\pa{s,\kappa_{0}\lambda_{n}/\sqrt{n}}}\ge 1-n^{-1}\right]
\ge1-e^{-\xi}.
\label{Eq-J2}
\end{equation}
Putting (\ref{Eq-J1}) and (\ref{Eq-J2}) together, we see that, for $n\ge n_{1}$,
\begin{equation}\label{eq-defB}
\P_{\gs}\left[\min\ac{\pi_{\!\gX}^{L}\pa{\sB_{n}},\pi_{\!\gX}(\sB_{n})}\ge 1-n^{-1}\right]\ge1-2e^{-\xi}
\end{equation}
with
\begin{equation}
\sB_{n}=\sB^{\overline{S}}\pa{s,\kappa_{4}\lambda_{n}/\sqrt{n}}\quad\text{and}\quad\kappa_{4}=16\vee\kappa_{0},
\label{def-Bn}
\end{equation} 

\paragraph{\bf Consequence 5} 
Let us now focus on the behaviour of the MLE $\widehat{\gtheta}_{n}$ on $\bs{\Theta}$. 
%
\begin{prop}\label{prop-mle}
Under Assumption~\ref{A-MLE}, a maximum likelihood estimator exists and any such sequence of estimators $(\widehat{\gtheta}_{n})$ is $\sqrt{n}$-consistent.
More precisely, for some $\kappa_{5}=\kappa_{5}(d,A_{1},A_{2})>0$ and all $n\ge1$,
\begin{equation}
\P_s\left[h\left(s,t_{\widehat{\gtheta}_{n}}\right)\le \kappa_{5}\sqrt{(1+\xi)/n}\right]\ge1-e^{-\xi}\quad\mbox{for all }\xi>0.
\label{Eq-J3}
\end{equation}
\end{prop}
\begin{proof}
Since the proof will be based on Theorem~7.4 page 99 of van de Geer~\citeyearpar{MR1739079}, we shall follow here her notations and denote by $\kappa$ a generic function of the two parameters $A_{1}$ and $A_{2}=A_{2}(\gvartheta)$ (recalling that $A_{3}=A_{1}/\sqrt{2}$ and $s=t_{\gvartheta})$ that may change from line to line. The above mentioned theorem relies on bounds for the Hellinger bracketing entropy
of the set of densities $\overline{U}=\{u_{\gtheta},\,\gtheta\in\bs{\Theta}\}$ with $u_{\gtheta}=(t_{\gtheta}+s)/2$. Following her notations, we shall set, for $\delta\in (0,1]$,
\[
\overline \cP^{1/2}(\delta)=\left\{\st\sqrt{u_{\gtheta}}\text{ such that }\gtheta\in\bs{\Theta} \ \mbox{and}\ h(u_{\gtheta},s)\le \delta\right\}
\]
and consider the quantity $H_{B}\left(\eps,\overline \cP^{1/2}(\delta),\mu\right)$ for $\eps>0$ which is the (local) entropy with bracketing, i.e.\ the logarithm of the smallest number of brackets of $\L_{2}(\mu)$-length not larger than $\eps$ that are necessary to cover $\overline \cP^{1/2}(\delta)$. Finally, let 
\begin{align*}
J_{B}\!\left(\overline \cP^{1/2}(\delta),\mu\right)&=\int_{\delta^{2}2^{-13}}^{\delta}
\sqrt{H_{B}\!\left(z,\overline \cP^{1/2}(\delta),\mu\right)}\,dz\\
&\le \int_{0}^{\delta}
\sqrt{H_{B}\!\left(z,\overline \cP^{1/2}(\delta),\mu\right)}\,dz.
\end{align*}
The concavity of the square root and Assumption~\ref{A-MLE}-$\ref{H4-iii}$ imply that
\[
\norm{\sqrt{u_{\gtheta}\over s}-\sqrt{u_{\gtheta'}\over s}}_{\infty}\le\frac{1}{\sqrt{2}}  
\norm{\sqrt{t_{\gtheta}\over s}-\sqrt{t_{\gtheta'}\over s}}_{\infty}\le\frac{A_{1}}{\sqrt{2}}|\gtheta-\gtheta'|
=A_{3}|\gtheta-\gtheta'|.
\]
Hence, for $z>0$, the inequality $|\gtheta-\gtheta'|\le z$ implies that, for all $x\in\sX$,
\[
\sqrt{u_{\gtheta}(x)}-A_{3}\sqrt{s(x)}\,z\le\sqrt{u_{\gtheta'}(x)}\le\sqrt{u_{\gtheta}(x)}+A_{3}\sqrt{s(x)}\,z
\]
and the square $\L_{2}(\mu)$-length of this bracket is
\[
\int\left[\left(\sqrt{u_{\gtheta}(x)}+A_{3}\sqrt{s(x)}z\right)-\left(\sqrt{u_{\gtheta}(x)}-A_{3}\sqrt{s(x)}z\right)\right]^{2}d\mu(x)=2A_{1}^{2}z^{2}.
\]
Taking $\overline \gtheta=\gvartheta$ in \eref{metric-h'}, we derive that 
\[
\overline \cP^{1/2}(\delta)\subset\left\{\st\sqrt{u_{\gtheta}}\text{ such that }\gtheta\in\bs{\Theta} \ \mbox{and}\ |\gtheta-\gvartheta|\le2\delta/A_{2}\right\},
\]
so that the previous computations imply that any covering of $\cB(\gvartheta,2\delta/A_{2})$ with closed balls of radii not larger than $z=\eps/(A_{1}\sqrt{2})$ leads to a covering of $\overline \cP^{1/2}(\delta)$ with brackets the $\L_{2}(\mu)$-lengths of which are not larger than $\eps$. Since, by Lemma~\ref{lem-Euclidball}, it is possible to find such a covering of cardinality $N$ satisfying
\[
\log N\le d\log\left(1+\frac{4\delta}{zA_{2}}\right)<\frac{4d\delta}{zA_{2}}=
\frac{4\sqrt{2}dA_{1}\delta}{\varepsilon A_{2}}={C_{4}d\delta\over \eps},
\]
we derive that
\[
J_{B}\left(\overline \cP^{1/2}(\delta),\mu\right)\le
\sqrt{C_{4}d \delta} \int_{0}^{\delta}\varepsilon^{-1/2}\,d\eps=2\delta\sqrt{C_{4}d}.
\]
We can therefore take $\Psi(\delta)=2\delta\sqrt{C_{4}d}$ in the statement of Theorem~7.4  of van de Geer~\citeyearpar{MR1739079}, hence $\delta_{n}=2c\sqrt{C_{4}d/n}$ for some universal constant $c\ge e$ and
\[
\P_s\left[h\left(s,t_{\widehat{\gtheta}_{n}}\right)>\delta\right]\le c\exp\left[-n(\delta/c)^{2}\right] \quad\mbox{for
}\delta\ge\delta_{n}.
\]
The result follows by setting $\kappa_{5}=c\sqrt{4\kappa d+\log c}$ and
\[
\delta=cn^{-1/2}\max\ac{\sqrt{4C_{4}d},\sqrt{\log c+\xi}}\le \kappa_{5}n^{-1/2}\sqrt{1+\xi}.\qedhere
\]
\end{proof}
%

\subsection{An essential intermediate result}\label{PV-2}
Let us set, for $J\in\N$, $n\ge d$, $\Gamma\in[1,n/d]$ and $\delta_{n}=\sqrt{\Gamma d/n}\le1$,
\begin{equation}
B=\ac{\gtheta\in \bs{\Theta}\,|\, h(t_{\gvartheta},t_{\gtheta})\le \delta_{n}}\quad\mbox{and}\quad B'_{J}=\ac{\gtheta\in \bs{\Theta}\,|\, h(t_{\gvartheta},t_{\gtheta})\le 2^{J/2}\delta_{n}}.
\label{def-BAB}
\end{equation}
We want to establish the following result.
%
\begin{prop}\label{etape0001}
Under Assumption~\ref{A-MLE}, there exist a positive constant $C_{5}$ and a  positive integer $J_{0}$, both independent of $n$, such that, for all $J\ge J_{0}$,
\begin{align*}
\lefteqn{\P_{\gs}\left[\sup_{\gtheta'\in \bs{\Theta}}\gPsi(\gX,t_{\gtheta},t_{\gtheta'})=\sup_{\gtheta'\in B'_{J}}\gPsi(\gX,t_{\gtheta},t_{\gtheta'})\quad\mbox{for all }\gtheta\in B\right]}\hspace{80mm}\\
&\ge1-\exp[-C_{5}2^{J}\Gamma d].
\end{align*}
\end{prop}
\begin{proof}
Since $B\subset B'_{J}$, $\sup_{\gtheta'\in B'_{J}}\gPsi(\gX,t_{\gtheta},t_{\gtheta'})\ge \gPsi(\gX,t_{\gtheta},t_{\gtheta})=0$, for all $\gtheta\in B$. Therefore it suffices to prove that
\[
\P_{\gs}\left[\sup_{\gtheta'\in \bs{\Theta}\setminus B'_{J}}\gPsi(\gX,t_{\gtheta},t_{\gtheta'})<0\quad\mbox{for all }\gtheta\in B\right]\ge1-\exp[-C_{5}2^{J}\Gamma d].
\]
Let $\bs{\Theta}'$ be a countable and dense subset of $\bs{\Theta}$ and take $\gS=\{\gt_{\gtheta}, \gtheta\in \bs{\Theta}'\}$. Since, by Assumption~\ref{A-MLE}-$\ref{H4-iii}$, for all $x\in\sX$, the mapping $\gtheta\mapsto t_{\gtheta}(x)$ is positive and continuous on $\bs{\Theta}$, to prove that $\sup_{\gtheta'\in \bs{\Theta}\setminus B'_{J}}\gPsi(\gX,t_{\gtheta},t_{\gtheta'})<0$, it suffices to find some $c>0$ such that
\[
\gPsi(\gX,t_{\gtheta},t_{\gtheta'})\le -c\quad\mbox{ for all }
\gtheta'\in \bs{\Theta}'\;\mbox{ with }\;h(\gvartheta,\gtheta')>  2^{J/2}\delta_{n}.
\]
Assumption~\ref{hypo-mes} being fulfilled, we may apply Theorem~\ref{main} with Proposition~\ref{prop-eta-eps} and use that $\min\{\Gamma,d\}\ge1$ to get the following result: 
for $z=C_{5}2^{J}\Gamma d$, there exists a set of probability at least $1-e^{-z}$ on which for all $\gtheta\in B$ and all $\gtheta'\in\bs{\Theta}'\setminus B'_{J}$
\begin{align*}
\gPsi(\gX,t_{\gtheta},t_{\gtheta'})&\le c_{7}nh^{2}(s,t_{\gtheta}) -c_{8}nh^{2}(s,t_{\gtheta'}) + c_{5}\kappa_{3}^{2}+c_{6}(z+2.4)\\
&\le c_{7}n\delta^{2}-c_{8}n2^{J}\delta^{2}+c_{5}\kappa_{3}^{2}+c_{6}(C_{5}2^{J}\Gamma d+2.4)\\
&= 2^{J}\Gamma d\cro{c_{6}C_{5}+2^{-J}\pa{c_{7}+c_{5}{\kappa_{3}^{2}\over \Gamma d}+{2.4c_{6}\over \Gamma d}}-c_{8}}\\
&\le 2^{J}\Gamma d\cro{c_{6}C_{5}+2^{-J}\pa{c_{7}+c_{5}\kappa_{3}^{2}+2.4c_{6}}-c_{8}}.
\end{align*}
The right-hand side is negative if we set $C_{5}=c_{8}/(2c_{6})$ and take $J$ larger than some constant $J_{0}$ only depending on $\kappa_{3}$, which proves the result.
\end{proof}

\subsection{Controlling $h^{2}(\pi_{\!\gX}^{L},\pi_{\!\gX})$}\label{PV-3}
Increasing $n_{1}$ again if necessary, we assume that $\sqrt{n_{1}}\ge\kappa_{4}\lambda_{n_{1}}$
with $\kappa_{4}$ given by (\ref{def-Bn}) and define $\Gamma$ and $J$ by
\[
\Gamma=(\kappa_{4}\lambda_{n})^2/d\qquad\mbox{and}\qquad
J=\inf\left\{j\ge J_0\,\left|\,C_{5}\kappa_{4}^{2}2^{j}\ge1\right.\right\}
\]
with $J_{0}$ provided by Proposition~\ref{etape0001}. It follows that $C_{5}2^{J}\Gamma d>\xi$ and that $\sB_{n}=\left\{t_\gtheta,\ \gtheta\in B\right\}$ where $\sB_{n}$ and $B$ are defined by~\eref{def-Bn} and~\eref{def-BAB} respectively. If $n\ge n_{1}$, $\Gamma d\le n$ and Proposition~\ref{etape0001} leads to
\[
\P_{\gs}\left[\gPsi(\gX,t_{\gtheta})=\sup_{\gtheta'\in B'_J}\gPsi(\gX,t_{\gtheta},t_{\gtheta'})\quad\mbox{for all }\gtheta\in {B}\right]\ge1-e^{-\xi}.
\]
For $\gtheta,\gtheta'\in \bs{\Theta}$, let us now consider the log-likelihood ratio 
\[
\gL(\gX,t_{\gtheta},t_{\gtheta'})=\log\pa{\prod_{i=1}^{n}{t_{\gtheta'}(X_{i})\over t_{\gtheta}(X_{i})}}=\sum_{i=1}^{n}\cro{\st\log t_{\gtheta'}(X_{i})-\log t_{\gtheta}(X_{i})}
\]
and set 
\[
\gL(\gX,t_{\gtheta})=\sup_{\gtheta'\in \bs{\Theta}}\gL(\gX,t_{\gtheta},t_{\gtheta'})=\cro{\sup_{\gtheta'\in \bs{\Theta}}\sum_{i=1}^{n}\log t_{\gtheta'}(X_{i})}-\sum_{i=1}^{n}\log t_{\gtheta}(X_{i}).
\]
It also follows from (\ref{Eq-J3}) that, with probability at least $1-e^{-\xi}$,
\[
\sup_{\gtheta'\in \bs{\Theta}}\sum_{i=1}^{n}\log t_{\gtheta'}(X_{i})=\sup_{\gtheta'\in B''}\sum_{i=1}^{n}\log t_{\gtheta'}(X_{i})
\]
where $B''=\left\{\gtheta, h(t_{\gvartheta},t_{\gtheta})\le \kappa_{5}\sqrt{(1+\xi)/n}\right\}$, hence
\[
\P_{\gs}\left[\gL(\gX,t_{\gtheta})=\sup_{\gtheta'\in B''}\gL(\gX,t_{\gtheta},t_{\gtheta'})\quad\mbox{for all }\gtheta\in {B}\right]\ge1-e^{-\xi}.
\]
Setting
\begin{equation}
\kappa_{6}=\left(2^{J/2}\kappa_{4}\right)\vee \kappa_{5}\qquad\text{and}\qquad
B'=\ac{\gtheta,\; h(t_{\gvartheta},t_{\gtheta})\le\kappa_{6}\lambda_{n}/\sqrt{n}}
\label{eq-rn}
\end{equation}
so that $B'\supset B''\cup B$, we conclude that, with a probability at least $1-2e^{-\xi}$, we simultaneously have for all $\gtheta\in B$,
\begin{equation}
\left\{\begin{array}{l}  
\!\!\gPsi(\gX,t_{\gtheta})=\dps{\sup_{\gtheta'\in B'}\gPsi(\gX,t_{\gtheta},t_{\gtheta'})=
\sup_{\gtheta'\in B'}\sum_{i=1}^{n}\psi\pa{\sqrt{t_{\gtheta'}\over t_{\gtheta}}(X_{i})}};\\
\!\!\dps{\gL\left(\gX,t_{\gtheta}\right)=\sup_{\gtheta'\in B'}\gL(\gX,t_{\gtheta},t_{\gtheta'})=\sup_{\gtheta'\in B'}\sum_{i=1}^{n}\cro{\st\log t_{\gtheta'}(X_{i})-\log t_{\gtheta}(X_{i})}}.
\end{array}\right.\hspace{-4mm}
\label{Eq-sup}
\end{equation}
From now on we assume that $\omega\in\widetilde \Omega_\xi$, the set of probability at least $1-4e^{-\xi}$ on which both (\ref{eq-defB}) and (\ref{Eq-sup}) hold. In order to simplify the notations let us set, for all $t\in\overline{S}$,
\[
L_{\!\gX}(t)=\exp\left[-\gL\left(\gX,t\right)\right],\ R_{\!\gX}(t)=\exp\left[-\beta\gPsi\left(\gX,t\right)\right]\text{ and }\Delta_{\!\gX}(t)=\frac{L_{\!\gX}(t)}{R_{\!\gX}(t)}.
\]
Then, according to \eref{eq-LPS2}, the respective densities of the classical Bayes and the $\rho$-Bayes posterior distributions, for $t\in\overline S$, are
\[
\frac{d\pi_{\!\gX}^{L}}{d\pi}(t)=\frac{L_{\!\gX}(t)}{\dps{\int_{\overline S}}L_{\!\gX}(t')\,d\pi(t')}\qquad
\text{and}\qquad\frac{d\pi_{\!\gX}}{d\pi}(t)=\frac{R_{\!\gX}(t)}{\dps{\int_{\overline S}}R_{\!\gX}(t')\,d\pi(t')}.
\]
It follows from the definitions of $\sB_{n}$ and $\pi_{\!\gX}^{L}$ and (\ref{eq-defB}) that
\[
\int_{\overline S}L_{\!\gX}(t)\,d\pi(t)=\frac{1}{\pi_{\!\gX}^{L}\pa{\sB_{n}}}\int_{\sB_{n}}L_{\!\gX}(t)\,d\pi(t)
\le{1\over 1-n^{-1}}\int_{\sB_{n}}L_{\!\gX}(t)\,d\pi(t)
\]
and similarly, 
\[
\int_{\overline S}R_{\!\gX}(t)\,d\pi(t)\le{1\over 1-n^{-1}}\int_{\sB_{n}}R_{\!\gX}(t)\,d\pi(t)
\]
so that the Hellinger affinity $\rho\pa{\pi_{\!\gX}^{L}, \pi_{\!\gX}}$ between $\pi_{\!\gX}^{L}$ and $\pi_{\!\gX}$ satisfies 
\begin{align*}
\rho\pa{\pi_{\!\gX}^{L}, \pi_{\!\gX}}&={\dps{\int_{\overline S}\sqrt{L_{\!\gX}(t)R_{\!\gX}(t)}\,d\pi(t)}\over \cro{\dps{\int_{\overline S}}L_{\!\gX}(t)\,d\pi(t)\dps{\int_{\overline S}}R_{\!\gX}(t)\,d\pi(t)}^{1/2}}\\
&\ge{(1-n^{-1})\dps{\int_{\sB_{n}}\sqrt{L_{\!\gX}(t)R_{\!\gX}(t)}\,d\pi(t)}\over \cro{\dps{\int_{\sB_{n}}}
L_{\!\gX}(t)\,d\pi(t)\dps{\int_{\sB_{n}}}R_{\!\gX}(t)\,d\pi(t)}^{1/2}},
\end{align*}
hence
\begin{equation}
\rho\pa{\pi_{\!\gX}^{L}, \pi_{\!\gX}}\ge{(1-n^{-1})\dps{\int_{\sB_{n}}R_{\!\gX}(t)}\sqrt{\Delta_{\!\gX}(t)}\,
d\pi(t)\over \cro{\dps{\int_{\sB_{n}}}R_{\!\gX}(t)\Delta_{\!\gX}(t)\,d\pi(t)\dps{\int_{\sB_{n}}}R_{\!\gX}(t)\,d\pi(t)}^{1/2}}.
\label{eq-rho1}
\end{equation}
By H\"older's inequality,
\[
\int_{\sB_{n}}R_{\!\gX}(t)\,d\pi(t)
\le\cro{\int_{\sB_{n}}R_{\!\gX}(t)\sqrt{\Delta_{\!\gX}(t)}\,d\pi(t)}^{2/3}
\cro{\int_{\sB_{n}}\frac{R_{\!\gX}(t)}{\Delta_{\!\gX}(t)}\,d\pi(t)}^{1/3},
\]
or equivalently,
\[
\int_{\sB_{n}}R_{\!\gX}(t)\sqrt{\Delta_{\!\gX}(t)}\,d\pi(t)\ge\left[\int_{\sB_{n}}R_{\!\gX}(t)\,d\pi(t)\right]^{3/2}\cro{\int_{\sB_{n}}\frac{R_{\!\gX}(t)}{\Delta_{\!\gX}(t)}\,d\pi(t)}^{-1/2}.
\]
Therefore, setting $\Lambda_{\!\gX}(t)=\Delta_{\!\gX}(t)\vee\left[\Delta_{\!\gX}(t)\right]^{-1}\ge1$, we derive from (\ref{eq-rho1}) that
\begin{align*}
\rho\pa{\pi_{\!\gX}^{L}, \pi_{\!\gX}}&\ge\frac{\left(1-n^{-1}\right)\dps{\int_{\sB_{n}}R_{\!\gX}(t)}\,d\pi(t)}{\cro{\dps{\int_{\sB_{n}}}R_{\!\gX}(t)\Delta_{\!\gX}(t)\,d\pi(t)\dps{\int_{\sB_{n}}}
\frac{R_{\!\gX}(t)}{\Delta_{\!\gX}(t)}\,d\pi(t)}^{1/2}}\\&\ge\frac{\left(1-n^{-1}\right)\dps{\int_{\sB_{n}}R_{\!\gX}(t)}\,d\pi(t)}{\dps{\int_{\sB_{n}}}R_{\!\gX}(t)\Lambda_{\!\gX}(t)\,d\pi(t)}\\&=\left(1-n^{-1}\right)
\left[\int_{\sB_{n}}\Lambda_{\!\gX}(t)\frac{R_{\!\gX}(t)}{\dps{\int_{\sB_{n}}R_{\!\gX}(t)}\,d\pi(t)}\,d\pi(t)\right]^{-1}\\&=\left(1-n^{-1}\right)\left[1+\int_{\sB_{n}}\left(\Lambda_{\!\gX}(t)-1\right)\frac{R_{\!\gX}(t)}{\dps{\int_{\sB_{n}}\!R_{\!\gX}(t)}\,d\pi(t)}\,d\pi(t)\right]^{-1}\\&\ge
\left(1-n^{-1}\right)\left[1-\int_{\sB_{n}}\left(\Lambda_{\!\gX}(t)-1\right)\frac{R_{\!\gX}(t)}{\dps{\int_{\sB_{n}}\!R_{\!\gX}(t)}\,d\pi(t)}\,d\pi(t)\right]\\&\ge1-n^{-1}-\int_{\sB_{n}}\left(\Lambda_{\!\gX}(t)-1\right)\frac{R_{\!\gX}(t)}{\dps{\int_{\sB_{n}}\!R_{\!\gX}(t)}\,d\pi(t)}\,d\pi(t),
\end{align*}
since $\Lambda_{\!\gX}(t)\ge1$. Finally,
\begin{equation}
h^{2}\pa{\pi_{\!\gX}^{L}, \pi_{\!\gX}}\le n^{-1}+\int_{\sB_{n}}\left(\Lambda_{\!\gX}(t)-1\right)\frac{R_{\!\gX}(t)}{\dps{\int_{\sB_{n}}\!R_{\!\gX}(t)}\,d\pi(t)}\,d\pi(t).
\label{eq-maj010}
\end{equation}
It follows from the triangle inequality and the definition~\eref{eq-rn} of $B'$ that, for $\gtheta\in B\subset B'$ and $\gtheta'\in B'$, $h(t_{\gtheta},t_{\gtheta'})\le 2\kappa_{6}\lambda_{n}/\sqrt{n}$, hence by Assumption~\ref{A-MLE}-$iii)$ and (\ref{metric-h}),
\[
\norm{\sqrt{t_{\gtheta'}\over t_{\gtheta}}-1}_{\infty}=\norm{\sqrt{t_{\gtheta'}\over t_{\gtheta}}-\sqrt{t_{\gtheta}\over t_{\gtheta}}}_{\infty}\le A_{1}|\gtheta-\gtheta'|\le \frac{A_{1}}{A_{2}}h(t_{\gtheta},t_{\gtheta'})\le\frac{2A_{1}\kappa_{6}\lambda_{n}}{A_{2}\sqrt{n}}.
\]
This implies that, if $n_{1}$ is large enough, $1/2\le\norm{\sqrt{t_{\gtheta'}/t_{\gtheta}}}_{\infty}\le2$. Moreover, since $\beta=4$ and, by (\ref{eq-philog}), 
\[
\left|\beta\psi\pa{\sqrt{x}}-2\log\pa{\sqrt{x}}\right|=|\varphi(x)-\log x|\le 0.055 |x-1|^{3}\,\mbox{ for }1/2\le x\le2,
\]
we derive from (\ref{Eq-sup}) that, for all $\gtheta\in B$,
\begin{align*}
\ab{\log\left(\st\Delta_{\!\gX}(t_{\gtheta})\right)}&=\ab{\beta \sup_{\gtheta'\in B'}\gPsi(\gX,t_{\gtheta},t_{\gtheta'})- \sup_{\gtheta'\in B'}\gL(\gX,t_{\gtheta},t_{\gtheta'})}\\&\le
\sup_{\gtheta'\in B'}\ab{\beta\gPsi(\gX,t_{\gtheta},t_{\gtheta'})-\gL(\gX,t_{\gtheta},t_{\gtheta'})}\\&=\sup_{\gtheta'\in B'}\ab{\sum_{i=1}^{n}\left[\beta\psi\pa{\sqrt{t_{\gtheta'}\over t_{\gtheta}}(X_{i})}-2\log\pa{\sqrt{t_{\gtheta'}\over t_{\gtheta}}(X_{i})}\right]}\\&\le 0.055n\norm{\sqrt{t_{\gtheta'}\over t_{\gtheta}}-1}_{\infty}^{3}\le0.055n\left(\frac{2A_{1}\kappa_{6}\lambda_{n}}{A_{2}\sqrt{n}}\right)^{3}.
\end{align*}
This implies that $\log\left(\st\Lambda_{\!\gX}(t)\right)\le\kappa_{7}\lambda^{3}_{n}/\sqrt{n}$ for some constant $\kappa_{7}\ge1$. If $n_{1}$ is large enough, $\kappa_{7}\lambda_{n}^{3}/\sqrt{n}\le5/4$, hence, since $e^{x}-1<2x$ for $0\le x\le5/4$, $\Lambda_{\!\gX}(t)-1\le2\kappa_{7}\lambda^{3}_{n}/\sqrt{n}$. Using this bound, \eref{eq-maj010} implies that, for $n\ge n_{1}$ and $\omega\in\widetilde{\Omega}_{\xi}$ with $\P_{\gs}\left[\widetilde{\Omega}_{\xi}\right]1-4e^{-\xi}$,
\[
h^{2}\!\pa{\pi_{\!\gX(\omega)}^{L}, \gpi_{\!\gX(\omega)}}\!\le\frac{1}{n}+2\kappa_{7}\frac{\lambda_{n}^{3}}{\sqrt{n}}=\frac{1}{n}+\frac{2\kappa_{7}[\xi+\log n]^{3/2}}{\sqrt{n}}\le(6\kappa_{7}+1)\frac{[\log n]^{3/2}}{\sqrt{n}}
\]
since $\log n\ge\log n_{1}\ge\xi$. The conclusion follows by setting $\xi=2\log 2+z$ and $C(z)=6\kappa_{7}+1$.

\section{Other proofs}

\subsection{Proof of Propositions~\ref{prop-fc} and \ref{maj-epss}}
let us begin with the following elementary result:
%
\begin{lem}\label{lem-trivial}
If $H$, $\alpha$ and $\gamma$ are positive numbers and
\begin{equation}
\gw^{\overline \gS}(\gs,y)\le\alpha H+\gamma y\sqrt{nH}\quad\text{for all }y\ge y_{0},
\label{eq-}
\end{equation}
then $\eps_{n}^{\overline{S}}(\gs)\le\left[(c_{0}\gamma/6)\sqrt{H/n}\right]\bigvee\left[1/\sqrt{n}\right]\bigvee y_{0}$.
\end{lem}
%
\begin{proof}
In view of (\ref{def-en}), $\eps_{n}^{\overline{S}}(\gs)\le \overline{y}$ if, for all $y\ge\overline{y}\ge n^{-1/2}$, $\gw^{\overline \gS}(\gs,y)\le6c_{0}^{-1}ny^{2}$ which holds if $y\ge\left[1/\sqrt{n}\right]\bigvee y_{0}$ and $\alpha H+\gamma y\sqrt{H}\le6c_{0}^{-1}ny^{2}$. Finally this last inequality is clearly satisfied if 
\[
y\ge\frac{c_{0}\gamma}{12}\sqrt{\frac{H}{n}}\left[1+\sqrt{1+\frac{24\alpha}{c_{0}\gamma^{2}H}}\right]>\frac{c_{0}\gamma}{6}\sqrt{\frac{H}{n}}.\qedhere
\]
\end{proof}
Setting $\sB_{y}=\sB^{\overline{S}}(\gs,y)$, let us apply Proposition~50 of BBS with $T=\sB_{y}\times\sB_{y}$ for some positive $y$, so that $\log_{+}(2|T|)=H_{y}=\log\pa{2\left|\sB_{y}\right|^{2}}$ and $U_{i,(t,t')}=\psi\pa{\sqrt{\left(t_{i}'/t_{i}\right)(X_{i})}}\in [-1,1]$. One may therefore take, in Proposition~50 of BBS, $b=1$ and $v^{2}=2a_{2}^{2}ny^{2}$ with $a_{2}^{2}=3\sqrt{2}$ by~\eref{eq-var-psi}, which leads to $\gw^{\overline \gS}(\gs,y)\le bH_{y}+v\sqrt{2H_{y}}=H_{y}+2a_{2}y\sqrt{nH_{y}}$ for all $y>0$. 

For Proposition~\ref{prop-fc}, $\gw^{\overline \gS}(\gs,y)\le H+2a_{2}y\sqrt{nH}$ with $H=\log\pa{2\left|\overline S\right|^{2}}$ and Lemma~\ref{lem-trivial} leads to $\eps_{n}^{\overline{S}}(\gs)\le\left[(c_{0}\gamma/6)\sqrt{H/n}\right]$ since $H=\log\pa{2\left|\overline S\right|^{2}}\ge\log2$. The conclusion of Proposition~\ref{prop-fc} follows since $\eps_{n}^{\overline{S}}(\gs)\le\sqrt{c_{0}/3}$.

Let us now turn to the proof of Proposition~\ref{maj-epss}. In this case, it follows from (\ref{def-mddim1}) and (\ref{eq-eps0}) that, for $y\ge2\varepsilon$
\[
\left|\sB_{y}\right|\le\exp\left[(y/\varepsilon)^{2}D(\varepsilon)\right]\le\exp\left[n(y/c_{0})^{2}\right]
\]
so that $H_{y}\le2n(y/c_{0})^{2}+\log2$. Since $n(\varepsilon/c_{0})^{2}\ge D(\varepsilon)\ge3/4$, $2n(y/c_{0})^{2}\ge8n(\varepsilon/c_{0})^{2}\ge6$ and $H(y)\le2n(y/c_{0})^{2}[1+(1/6)\log2]<2.116n(y/c_{0})^{2}$. It follows that 
\[
H_{y}+2a_{2}y\sqrt{nH_{y}}<ny^{2}c_{0}^{-1}\left[(2.116/c_{0})+2a_{2}\sqrt{2.116}\right]<6c_{0}^{-1}ny^{2}
\]
for all $y\ge2\varepsilon$, hence the result.

\subsection{Proof of Proposition~\ref{prop-densite}}
Let $S_{n}=S_{\eps_{n}}$ and let $\pi$ be the uniform distribution on $S_{n}$. For all $\eps>0$, $S_{n}$ is obviously an $\eps$-net for itself and 
\[
\log \ab{\ac{t\in S_{n},\ h(s,t)\le r}}\le \log\ab{S_{n}}\le H(\eps_{n})\le D_{n}(r/\varepsilon)^{2} \text{ for all }r\ge2\varepsilon,
\]
with $D_{n}=H(\eps_{n})/4$. Hence $S_{n}$ has metric dimension bounded by $D_{n}$. Since $S_{n}$ is finite, we may endow $S_{n}$ with the $\sigma$-algebra generated by all its parts so that Assumption~\ref{hypo-mes} is clearly satisfied with $\overline S=S=S_{n}$. It follows from~\eref{def-epsn} that  $\sqrt{n}\eps_{n}>1/2$ and $D_{n}\le10^{-6}n\eps_{n}^{2}=\min\{c_{0}^{-2},\gamma/4\}n\eps_{n}^{2}$. Hence $\eps_{n}^{S_{n}}(\gs)\le 2\eps_{n}$ by Proposition~\ref{maj-epss} and $\eta_{n}^{S_{n},\pi}(t)\le \eps_{n}$ for all $t\in S_{n}$ by Proposition~\ref{def-etaRes}. Finally, since $S_{n}$ is an $\eps_{n}$-net for $\overline S$, $\inf_{t\in S_{n}}h(\gs,t)=\inf_{t\in \overline S}h(\gs,t)+\eps_{n}$ and the conclusion follows from Theorem~\ref{thm-main} applied to the density model $S_{n}$ endowed with the prior $\pi$. 
%

\subsection{Proof of Proposition~\ref{prop-reg}}
In this framework, the data $X_{1},\ldots,X_{n}$ are i.i.d.\ with density $s=p_{f\et}$ with respect to $\mu$ and $p_{f\et}(w,y)=p(y-f\et(w))$. 

Since the Lebesgue measure is translation invariant we derive that
\begin{equation}
h^{2}(s,q_{f\et})=\int_{\sW}h^{2}\pa{p_{f\et(w)},q_{f\et(w)}}dP_{W}(w)=h^{2}\pa{p,q}
\label{eq-hsq}
\end{equation}
and, for all measurable functions $f$ and $g$ on $\sW$,  
\[
h^{2}(q_{f},q_{g})=\int_{\sW}h^{2}\pa{q_{f(w)},q_{g(w)}}dP_{W}(w)=\int_{\sW}h^{2}\pa{q_{f(w)-g(w)},q}dP_{W}(w).
\]
It then follows from \eref{regul} by integration with respect to $P_{W}$ that
\begin{equation}\label{eq-hell-Lp}
h^{2}(q_{f},q_{g})\le a\norm{f-g}_{1+\alpha}^{1+\alpha}\le a\norm{f-g}_{\infty}^{1+\alpha}
\end{equation}
and, if $f$ and $g$ are bounded by $B>0$, $\ab{f(w)-g(w)}/(2B)\le1$, hence
\begin{align*}
\frac{a^{-1}\norm{f-g}_{1+\alpha}^{1+\alpha}}{[a(2B)^{1+\alpha}]\vee 1}&=a^{-1}\int_{\sW}\frac{\ab{f(w)-g(w)}^{1+\alpha}}{[a(2B)^{1+\alpha}]\vee 1}dP_{W}(w)\\
&\le\int_{\sW}a^{-1}\cro{\ab{f(w)-g(w)}^{1+\alpha}\wedge a^{-1}}dP_{W}(w)\\
&\le \int_{\sW}h^{2}\pa{q_{f(w)},q_{g(w)}}dP_{W}(w)=h^{2}(q_{f},q_{g}).
\end{align*}
Finally
\begin{equation}
h(q_{f},q_{g})\ge C_{0}\norm{f-g}_{1+\alpha}^{(1+\alpha)/2}\quad\text{with}\quad
C_{0}=a^{-1}\left(\left[a(2B)^{1+\alpha}\right]\vee 1\right)^{-1}.
\label{eq-hsq'}
\end{equation}

Let $\delta_{n}=\left(\eps_{n}^{2}/a\right)^{1/(1+\alpha)}$ with $\eps_{n}$ satisfying \eref{eq-epsreg}, $\cF_{n}$ be a $\delta_{n}$-net for $\overline \cF$ (in $\norm{\cdot}_{\infty}$-norm, therefore in $\norm{\cdot}_{1+\alpha}$-norm) of cardinality bounded by $\exp[H(\delta_{n})]$, $\pi'$ the uniform distribution on $\cF_{n}$, $S_{n}=\{q_{f},\; f\in \cF_{n}\}$ and $\pi$ be the uniform distribution on $S_{n}$. Inequality \eref{eq-hell-Lp} implies that $S_{n}$ is an $\eps_{n}$-net for $\overline S$ with respect to the Hellinger distance which satisfies 
\[
\log\ab{S_{n}}=\log\ab{\cF_{n}}\le H(\delta_{n})=H\cro{\left(\eps_{n}^{2}/a\right)^{1/(1+\alpha)}}
\le \left(4\cdot 10^{-6}\right)n\eps_{n}^{2}.
\]
Arguing as in the proof of  Proposition~\ref{prop-densite} we derive that the $\rho$-posterior $\pi_{\bsX}$ satisfies with a probability at least $1-e^{-\xi}$,
\begin{equation}
\pi_{\!\gX(\omega)}\pa{\sB^{\overline S}(s,C\overline r_{n})}\ge 1-e^{-\xi'}\quad \text{with}\quad \overline r_{n}=h(s,\overline{S})+\eps_{n}+\sqrt{\frac{\xi+\xi'}{n}}
\label{eq-P2a}
\end{equation}
and $C$ depending on $c_{j}$ for $1\le j\le4$. By (\ref{eq-hsq}) and \eref{eq-hell-Lp},
\[
h(s,q_{f})\le h(s,q_{f\et})+h(q_{f\et},q_{f})\le h(p,q)+a\norm{f\et-f}_{1+\alpha}^{(1+\alpha)/2},
\]
hence
\begin{equation}
h(s,\overline S)=\inf_{f\in\overline{\cF}}h(s,q_{f})\le h(p,q)+a\inf_{f\in\overline{\cF}}\norm{f\et-f}_{1+\alpha}^{(1+\alpha)/2}.
\label{eq-P2b}
\end{equation}
Similarly, by (\ref{eq-hsq}) and (\ref{eq-hsq'}),
\[
h(s,q_{f})\ge h(q_{f\et},q_{f})-h(s,q_{f\et})\ge C_{0}\norm{f\et-f}_{1+\alpha}^{(1+\alpha)/2}-h(p,q),
\]
hence, for all $r>0$,
\begin{equation}
\ac{f\in \overline \cF,\; h(s,q_{f})\le r }\subset \ac{f\in \overline \cF,\; \norm{f\et-f}_{1+\alpha}^{(1+\alpha)/2}\le \frac{h(p,q)+r}{C_{0}}}
\label{eq-P2c}
\end{equation}
and the conclusion follows by putting (\ref{eq-P2a}), (\ref{eq-P2b}) and (\ref{eq-P2c}) together.

\subsection{Proof of Proposition~\ref{prop-VC}}
Since $\sF$ is weak VC-major with dimension $d$, 
\[
\sF_{y}=\left\{\left.\psi\left(\sqrt{t'\over t}\right)\,\right|\, (t,t')\in\sB^{\overline{S}}(\gs,y)\times\sB^{\overline{S}}(\gs,y)\right\}
\]
is also weak VC-major with dimension not larger than $d$ as a subset of $\sF$. Applying Corollary~1 in Baraud~\citeyearpar{Bar2016} to the class of functions $\sF_{y}$ and the (independent) random variables $X_{i}$, we may take $b=1$ (since $\psi$ is bounded by 1 and $n\sigma^{2}=(2a_{2}^{2}y^{2})\wedge n$ with $a_{2}^{2}=3\sqrt{2}$ by~\eref{eq-var-psi} and the definition of $\sB^{\overline{S}}(\gs,y)$. We derive from this corollary that
\[
\gw^{\overline{S}}(\gs,y)\le 8\left[a_{2}\sqrt{n}y\log\!\pa{e\left[{1\over a_{2}\sqrt{2}y}\vee 1 \right]}\sqrt{\overline \Gamma(d)}+2\overline \Gamma(d)\right]\quad\mbox{for all }y>0,
\]
with
\[
\overline \Gamma(d)=\log 2+(d\wedge n)L(d)\qquad\mbox{and}\qquad L(d)=\log\pa{\st(en)/(d\wedge n)}\ge1.
\]
Since $(d\wedge n)L(d)$ is a nondecreasing function of $d$ it is not smaller than $L(1)=\log(en)$ so that $\overline \Gamma(d)\le \overline{c}_{n}(d\wedge n)L(d)$, with $\overline{c}_{n}$ given by (\ref{eq-constantes}). Moreover, if $y\ge y_{0}=(d\wedge n)\left(na_{2}\sqrt{2}\right)^{-1}$, which we shall now assume, then
\[
\log\!\left(e\left[{1\over a_{2}\sqrt{2}y}\vee 1 \right]\right)\le L(d)\mbox{ hence }
\gw^{\overline{S}}(\gs,y)\le 8\!\left[a_{2}yL(d)\sqrt{n\overline \Gamma(d)}+2\overline \Gamma(d)\right].
\]
It then follows from Lemma~\ref{lem-trivial} with $\overline{H}=\overline \Gamma(d)$, $\alpha=16$ and $\gamma=8a_{2}L(d)$ that 
\[
\eps_{n}^{\overline{S}}(\gs)\le(4/3)c_{0}a_{2}L(d)\sqrt{\overline \Gamma(d)/n}\le
(4/3)c_{0}a_{2}L(d)^{3/2}\sqrt{\overline{c}_{n}(d\wedge n)/n},
\]
an upper bound which is larger than $\left[1/\sqrt{n}\right]\bigvee y_{0}$ since $L(d)\ge\log(en)$
and $y_{0}<\sqrt{\overline{c}_{n}(d\wedge n)/n}$. The conclusion follows.

\subsection{Proof of Proposition~\ref{prop-exp0}}
For all $t$ and $t'$ in $S$, the ratio $\sqrt{t'/t}$ is still of the form~\eref{eq-modelt0} and, since $\psi$ is monotone, by Proposition~3 of Baraud~\citeyearpar{Bar2016}, it suffices to prove that the class of functions of the form~\eref{eq-modelt0} is weak VC-major with dimension not larger than $d=J+1$ (or equivalently of index not larger than $J+2$). It is in fact VC with dimension not larger than $J+1$ by Proposition~12-$(i)$ of BB . Hence, it is weak VC-major with dimension not larger than $d=J+1$ by Proposition~1 in Baraud~\citeyearpar{Bar2016}). This proves the first part of the proposition. For the second part, the arguments are the same except that we now use  Proposition~12-$(ii)$ of BB.

\subsection{Proof of Proposition~\ref{prop-exp}}
For all $t$ and $t'$ in $S$, the ratio $\sqrt{t'/t}$ is still of the form~\eref{eq-modelt} where 
\[
\sJ=\{I\cap I'\,|\, I\in\sJ(t),\ I'\in \sJ(t')\}
\]
is now a partition of $\gI$ into at most $2k$ intervals. Since $\psi$ is monotone, by Proposition~3 of Baraud~\citeyearpar{Bar2016}, it suffices to prove that the class of such functions is weak VC-major with dimension not larger than $d=\lceil18.8k(J+2)\rceil$. This follows from Proposition~12-$(ii)$ of BB applied with $k$ replaced by $2k$.

\subsection{Proof of Proposition~\ref{casconv}}
Let $r>0$ and $\gtheta\in \bs{\Theta}$. It follows from \eref{eq-conEh} that, on the one hand,
\begin{align*}
\gpi\pa{\sB^{\overline{S}}(t_{\gtheta},r)}&=\nu\pa{\ac{\gtheta'\in \bs{\Theta}\,\left| \,h(t_{\gtheta},t_{\gtheta'})\le r\right.}}\\
&\ge \nu\pa{\ac{\gtheta'\in \bs{\Theta}\,\left| \,\overline a\ab{\gtheta-\gtheta'}_{*}^{\alpha}\le r\right.}}\\
&= \nu\pa{\ac{\gtheta+\gv\,\left| \,\gv\in \R^{d}\mbox{ and }\ab{\gv}_{*}\le(r/\overline{a})^{1/\alpha}\right.}\bigcap \bs{\Theta}}\\
&= \nu\pa{\cB_{*}\pa{\gtheta,({r/\overline{a}})^{1/\alpha}}\bigcap \bs{\Theta}}=\nu\pa{\cB_{*}\pa{\gtheta,(r/\overline{a})^{1/\alpha}}},
\end{align*}
since $\nu$ is supported by $\bs{\Theta}$. On the other hand, by the same arguments,
\begin{align*}
\gpi\pa{\sB^{\overline{S}}(t_{\gtheta},2r)}
&=\nu\pa{\ac{\gtheta'\in \bs{\Theta}\,\left| \,h(t_{\gtheta},t_{\gtheta'})\le 2r\right.}}\\&\le \nu\pa{\ac{\gtheta'\in \bs{\Theta}\,\left| \,\underline a\ab{\gtheta-\gtheta'}_{*}^{\alpha}\le2r\right.}}\\
&\le \nu\pa{\ac{\gtheta+\gv\,\left| \,\gv\in \R^{d}\mbox{ and }\ab{\gv}_{*}\le\frac{2r}{\underline{a}})^{1/\alpha}\right.}}\\&=\nu\pa{\cB_{*}\pa{\gtheta,(2r/\underline{a})^{1/\alpha}}}.
\end{align*}
Let $k\in\N$ be such that
\[
2^{k-1}<\left(\frac{2\overline{a}}{\underline{a}}\right)^{1/\alpha}\le2^{k}\quad\mbox{so that}\quad k<\delta=\frac{1}{\alpha\log 2}\log\left(\frac{2\overline{a}}{\underline{a}}\right)+1.
\]
It then follows from an iterative application of (\ref{eq-hypnu}) that 
\begin{align*}
{\gpi\pa{\sB^{\overline{S}}(t_{\gtheta},2r)}\over \gpi\pa{\sB^{\overline{S}}(t_{\gtheta},r)}}
&\le \frac{\nu\pa{\cB_{*}\pa{\gtheta,(2r/\underline{a})^{1/\alpha}}}}
{\nu\pa{\cB_{*}\pa{\gtheta,(r/\overline{a})^{1/\alpha}}}}
\le\frac{\nu\pa{\cB_{*}\pa{\gtheta,2^{k}(r/\overline{a})^{1/\alpha}}}}
{\nu\pa{\cB_{*}\pa{\gtheta,(r/\overline{a})^{1/\alpha}}}}\\
&\le\kappa_{\gtheta}^{k}[(r/\overline a)^{1/\alpha}]
\le\exp\left[\delta\log\left[\kappa_{\gtheta}\pa{(r/\overline a)^{1/\alpha}}\right]\right]
\end{align*}
since the function $\kappa_{\gtheta}$ is non-increasing. In order to majorize
the value of $\eta_{n}^{\overline{S},\gpi}(t_{\gtheta})$ we have to find, according to (\ref{Eq-eta}), the minimal value of $\eta$ such that $\gamma n r^{2}\ge\delta\log\left[\kappa_{\gtheta}\pa{(r/\overline a)^{1/\alpha}}\right]$ for all $r\ge\eta$, or equivalently
\[
\gamma n\eta^{2}\ge\log\left[\kappa_{\gtheta}\pa{(\eta/\overline a)^{1/\alpha}}\right]\left[\frac{1}{\alpha\log 2}\log\left(\frac{2\overline{a}}{\underline{a}}\right)+1\right],
\]
since $\kappa_{\gtheta}$ is non-increasing. This leads to (\ref{eq-eta-param0}) and (\ref{eq-eta-param9}). 
Finally (\ref{eq-eta-param}) follows from an application of the next lemma with $\cB=\cB_{*}(0,1)$, $r_{2}=2r$ and $r_{1}=r$ which implies that $\kappa_{0}=2^{d}(\overline b/\underline b)$.
%
\begin{lem}\label{lem-star}
If $\bs{\Theta}$ is convex and (\ref{eq-densborn}) holds, then, for all $\gtheta\in\bs{\Theta}$ and all measurable subsets $\cB$ of $\R^d$,
\[
\nu(\gtheta+r_2\cB)\le\left(r_2/r_1\right)^d(\overline b/\underline b)\,\nu(\gtheta+r_1\cB)\quad\mbox{for all }r_2>r_1>0.
\]
\end{lem}
%
\begin{proof}
Let us fix some $\gtheta\in \bs{\Theta}$ and set
\[
\gB_1=\left\{\left.\gu\in r_1\cB\,\right|\,\gtheta+\gu\in\bs{\Theta}\right\}\qquad\mbox{and}\qquad\gB_2=\left\{\left.\gu\in r_2\cB\,\right|\,\gtheta+\gu\in\bs{\Theta}\right\}.
\]
We have to prove that $\nu(\gtheta+\gB_2)\le(r_2/r_1)^{d}(\overline b/\underline b) \nu(\gtheta+\gB_1)$.
Since  $\gtheta+\gB_1$ and $\gtheta+\gB_2$ are two subsets of $\bs{\Theta}$,
\[
\nu(\gtheta+\gB_2)\le \overline b\lambda(\gtheta+\gB_2)\qquad\mbox{and}\qquad
\nu(\gtheta+\gB_1)\ge \underline b\lambda(\gtheta+\gB_1).
\]
It is therefore enough to show that $\lambda(\gtheta+\gB_2)\le(r_2/r_1)^d\lambda(\gtheta+\gB_1)$ or, by translation invariance of the Lebesgue measure, that $\lambda(\gB_2)\le(r_2/r_1)^d\lambda(\gB_1)$. If $\gu\in \gB_2$, then $\gv=(r_1/r_2)\gu\in r_1\cB$ and, since $\bs{\Theta}$ is convex,
\[
\gtheta+\gv=\left(1-\frac{r_1}{r_2}\right)\gtheta+\frac{r_1}{r_2}(\gtheta+\gu)\in\bs{\Theta},
\]
so that $\gv\in \gB_1$. This implies that $\gB_2\subset\{\gu\,|\,(r_1/r_2)\gu\in \gB_1\}$ so that
\begin{align*}
\lambda(\gB_2)&=\int\1_{\gB_2}(u)\,d\lambda(\gu)\le\int\1_{\gB_1}\left(\frac{r_1}{r_2}\gu\st\right)d\lambda(\gu)\\&=\left(\frac{r_2}{r_1}\right)^d\!\int\1_{\gB_1}(\gv)\,d\lambda(\gv)=\left(\frac{r_2}{r_1}\right)^d\!\lambda(\gB_1),
\end{align*}
which concludes our proof.
\end{proof}
%

\subsection{Proof of Proposition~\ref{prop-nu1/nu2}}
By symmetry, we may assume that $\theta\ge0$ and note that, since $\nu$ is supported by $[-1,1]$, $\nu([\theta-x,\theta+x])=1$ if $x\ge\theta+1$, in which case $\nu(I_{2})=\nu(I_{1})=1$. We therefore only consider the case of $0<x<\theta+1$. Since $g$ is decreasing on $[0,1]$,
\[
\nu([\theta,\theta+2x])=\int_{\theta}^{(\theta+2x)\wedge1}g(t)\,dt\le2\int_{\theta}^{(\theta+x)\wedge1}g(t)\,dt=2\nu([\theta,\theta+x]).
\]
Now observe that 
\[
2\nu([\theta-z,\theta]) =
\begin{cases}
\dps{\int_{(\theta-z)}^{0}\xi(-t)^{\xi-1}\,dt+\int_{0}^{\theta}\xi t^{\xi-1}\,dt}=(z-\theta)^{\xi}+\theta^{\xi}&\!\!\text{if }z\ge\theta\\
\dps{\int_{(\theta-z)}^{\theta}\xi t^{\xi-1}\,dt}=\theta^{\xi}-(\theta-z)^{\xi}&\!\!\mbox{if }z\le\theta
\end{cases}
\]
so that,
\[
2\theta^{-\xi}\nu([\theta-z,\theta])=\left\{
\begin{array}{ll}  
1+[(z/\theta)-1]^{\xi}&\;\mbox{ if }z\ge\theta\\1-[1-(z/\theta)]^{\xi}&\;\mbox{ if }z\le\theta
\end{array}\right..
\]
Let now set $y=x/\theta$.\\
--- If $x\ge\theta$, then $y\ge1$ and, since the function $y\mapsto(y-1)^{\xi}$ is concave,
\[
\frac{\nu([\theta-2x,\theta])}{\nu([\theta-x,\theta])}-1=\frac{1+(2y-1)^{\xi}}{1+(y-1)^{\xi}}-1=
\frac{(2y-1)^{\xi}-(y-1)^{\xi}}{1+(y-1)^{\xi}}\le\frac{y^{\xi}}{1+(y-1)^{\xi}}.
\]
If $y\le2$, the right-hand side is bounded by $2^{\xi}$ and when $y>2$, $y/(y-1)<2$ so that the same bound holds, which allows us to conclude that in this case 
\[
\frac{\nu([\theta-2x,\theta])}{\nu([\theta-x,\theta])}\le 1+2^{\xi}.
\]
--- If $0<x<\theta/2$, then $0<y<1/2$
\[
\frac{\nu([\theta-2x,\theta])}{\nu([\theta-x,\theta])}=\frac{1-(1-2y)^{\xi}}{1-(1-y)^{\xi}}=1+
f(y)\ \text{with}\ f(y)=\frac{(1-y)^{\xi}-(1-2y)^{\xi}}{1-(1-y)^{\xi}}
\]
and $f'(y)=A(y)\xi\left(1-(1-y)^{\xi}\right)^{-2}$ with
\begin{align*}
A(y)&=2(1-2y)^{\xi-1}-(1-y)^{\xi-1}-2(1-y)^{\xi}(1-2y)^{\xi-1}\\
&\quad +(1-y)^{\xi-1}(1-2y)^{\xi}\\&=
2(1-y)^{\xi-1}(1-2y)^{\xi-1}\left[(1-y)^{1-\xi}-\frac{(1-2y)^{1-\xi}+1}{2}\right].
\end{align*}
The strict concavity of the function $z\mapsto z^{1-\xi}$ implies that the bracketed factor is
positive, hence $A$ and $f'$ as well so that
\[
\sup_{0<y<1/2}f(y)=\lim_{y\rightarrow1/2}f(y)=\left(2^{\xi}-1\right)^{-1}\quad\mbox{and}\quad
\frac{\nu([\theta-2x,\theta])}{\nu([\theta-x,\theta])}\le1+\frac{1}{2^{\xi}-1}.
\]
--- Finally, if $\theta/2\le x<\theta$, then $1/2\le y<1$ and
\begin{equation}
\frac{\nu([\theta-2x,\theta])}{\nu([\theta-x,\theta])}=\frac{1+(2y-1)^{\xi}}{1-(1-y)^{\xi}}<\frac{2}{1-2^{-\xi}}=\frac{2^{1+\xi}}{2^{\xi}-1}.
\label{eq-2xi}
\end{equation}
In all three cases one can check that (\ref{eq-2xi}) holds which concludes our proof since this bound is larger than 2.

\subsection{Proof of Proposition~\ref{prop-cas3}}
Let
\[
J=c_{\delta}\nu([-x,x])=2\int_{0}^{x}\exp\left[-\frac{1}{2\theta^{\delta}}\right]d\theta\quad
\mbox{for }0<x\le1.
\]
A change of variable followed by an integration by parts leads to
\begin{align*}
\frac{\delta J}{4}&=\int_{x^{-\delta/2}}^{+\infty}u\exp\left[\frac{-u^{2}}{2}\right]u^{-(2/\delta)-2}\,du
\\&=\left[-u^{-(2/\delta)-2}\exp\left[\frac{-u^{2}}{2}\right]\right]_{x^{-\delta/2}}^{+\infty}\\
&\quad-\frac{2(\delta+1)}{\delta}\int_{x^{-\delta/2}}^{+\infty}u^{-(2/\delta)-3}\exp\left[\frac{-u^{2}}{2}\right]du\\&=x^{\delta+1}\exp\left[\frac{-x^{-\delta}}{2}\right]-\frac{2(\delta+1)}{\delta}
\int_{x^{-\delta/2}}^{+\infty}u^{-(2/\delta)-3}\exp\left[\frac{-u^{2}}{2}\right]du.
\end{align*}
It follows that
\[
-\frac{2(\delta+1)}{\delta}x^{(3\delta/2)+1}\int_{x^{-\delta/2}}^{+\infty}\exp\left[\frac{-u^{2}}{2}\right]du
\le\frac{\delta J}{4}-x^{\delta+1}\exp\left[\frac{-x^{-\delta}}{2}\right]<0.
\]
Since, 
\[
\int_{z}^{+\infty}\exp\left[\frac{-u^{2}}{2}\right]du<\frac{1}{z}\exp\left[\frac{-z^{2}}{2}\right]
\quad\mbox{for }z>0,
\]
then
\[
-\frac{2(\delta+1)}{\delta}x^{2\delta+1}\exp\left[\frac{-x^{-\delta}}{2}\right]
\le\frac{\delta J}{4}-x^{\delta+1}\exp\left[\frac{-x^{-\delta}}{2}\right]<0
\]
and finally
\[
x^{\delta+1}\exp\left[\frac{-x^{-\delta}}{2}\right]\left(1-\frac{2(\delta+1)}{\delta}x^{\delta}\right)\le\frac{\delta J}{4}<x^{\delta+1}\exp\left[\frac{-x^{-\delta}}{2}\right].
\]
Let $x_{\delta}$ be such that $2\delta^{-1}(\delta+1)x_{\delta}^{\delta}=1/2$. Then, for $0<x\le x_{\delta}<1$,
\[
\frac{\nu([-2x,2x])}{\nu([-x,x])}\le\frac{(2x)^{\delta+1}\exp\left[-(2x)^{-\delta}/2\right]}
{(1/2)x^{\delta+1}\exp\left[-x^{-\delta}/2\right]}=2^{\delta+2}
\exp\left[\frac{x^{-\delta}}{2}\left(1-2^{-\delta}\right)\right].
\]
If $x_{\delta}<x\le1$, then $\nu([-2x,2x])\le1$ while $\nu([-x,x])\ge\nu([-x_{\delta},x_{\delta}])$.
It follows that in both cases (\ref{eq-hypnu}) is satisfied with
\[
\log\left(\kappa_{0}(x)\st\right)=K_{0}+\frac{x^{-\delta}}{2}\left(1-2^{-\delta}\right)\le
K_{1}x^{-\delta}\quad\mbox{for all }x\in(0,1],
\]
where $K_{0}$ and $K_{1}$ only depend on $\delta$. As a consequence, 
\[
\frac{1}{n\gamma}\log\left(\kappa_{0}([r/\overline a]^{1/\alpha})\right)\left[\frac{\log(2\overline{a}/\underline{a})}{\alpha\log 2}+1\right]
\le K_{2}r^{-\delta/\alpha}n^{-1}
\]
where $K_{2}$ depends on $\alpha,\delta,\overline{a},\underline{a}$ and $\gamma$. The conclusion follows.

\subsection{Proof of Proposition~\ref{prop-selexp}}
For all $\gtheta,\gtheta'\in \bs{\Theta}_{m}'$, let 
\[
\Delta_{m}(\gtheta,\gtheta')={1\over 2}\cro{A_{m}(\gtheta)+A_{m}(\gtheta')}-A_{m}(\overline \gtheta)\quad\mbox{with}\quad\overline \gtheta=\frac{\gtheta+\gtheta'}{2}.
\]
Under our assumption on the maps $A_{m}$,
\begin{align*}
\Delta_{m}(\gtheta,\gtheta')&= {1\over 2}\int_{0}^{1}\cro{\int_{-t}^{t}\left<A''_{m}\pa{\overline \gtheta+u{\gtheta-\gtheta'\over 2}}{\gtheta-\gtheta'\over 2},{\gtheta-\gtheta'\over 2}\right>du}dt\\
&\le {\sigma_{m}^{2}\over 8}\ab{\gtheta-\gtheta'}^{2}\le {(m+1)\sigma_{m}^{2}\over 8}\ab{\gtheta-\gtheta'}_{\infty}^{2}
\end{align*}
with $\ab{\gtheta-\gtheta'}_{\infty}=\max_{j=0,\ldots,m}\ab{\theta_{j}-\theta_{j'}}$. Hence, by~\eref{eq-rhott'} for all $\gtheta,\gtheta'\in \bs{\Theta}_{m}'$,
\begin{equation}
h^{2}(t_{\gtheta},t_{\gtheta'})=1-\exp\cro{-\Delta_{m}(\gtheta,\gtheta')}\le {(m+1)\sigma_{m}^{2}\over 8}\ab{\gtheta-\gtheta'}_{\infty}^{2}.
\label{eq-Bay1}
\end{equation}

The mapping $A_{m}$ being continuous and strictly convex on $\bs{\Theta}_{m}'$, it follows from Proposition~\ref{mes-cond-modele} that the density sets 
\[
\overline{S}_{m}'=\ac{t_{\gtheta}=\exp\cro{\<\gtheta,\gT\>-A(\gtheta)},\ \gtheta\in\bs{\Theta}_{m}'}
\]
satisfy our Assumption~\ref{hypo-mes}, for all $m\in\cM$. Besides, we may choose $S_{m}'=\left\{t_{\gtheta}\,\left|\,\gtheta\in \bs{\Theta}_{m}'\cap \Q^{\N}\right.\right\}$ as a countable and dense subset of $\overline S_{m}'$ for all $m\in\cM$ so that $S'_{m}\subset S_{m'}'$ for all $(m,m')\in\cM^{2}$ with $m<m'$. 

Let us now set  $S_{0}=S_{0}'$, $\bs{\Theta}_{0}=\bs{\Theta}_{0}'$ and for all $m\ge 1$, $S_{m}=S_{m}'\cap \overline S_{m}$  and $\bs{\Theta}_{m}=\bs{\Theta}_{m}'\setminus \bs{\Theta}_{m-1}'=\{\gtheta\in\bs{\Theta}_{m}'\,|\,\theta_{m}\ne 0\}$. In order to apply our Theorem~\ref{thm-main2}, we first need to check that Assumption~\ref{H-mod} is satisfied. 

Since $t\mapsto t(x)$ is continuous on $\overline S_{m}'$, it is also continuous on $\overline S_{m}\subset \overline S_{m}'$ for all $m\in\cM$. More generally, for all $t'\in \overline S$ and $\gx\in\sX^{n}$, the mapping $t\mapsto \gPsi(\gx,\gt,\gt')$ is continuous on $\overline S_{m}$ since $t'>0$. Consequently, the measurable spaces $(\overline S_{m},\sS_{m})_{m\in\cM}$ satisfy Assumption~\ref{H-mod}-$\ref{H'i}$ and $\ref{H'ii}$. For $m,m'\in\cM$, we may use $S_{m}\cup S_{m'}$ and $S_{m}'\cup S_{m'}'$ as countable and dense subsets of $\overline S_{m}\cup\overline  S_{m'}$ and  $\overline S_{m}'\cup\overline  S_{m'}'$ respectively in order to define the quantities $\eps_{n}^{\overline S_{m}\cup\overline S_{m'}}(\gs)$ and $\eps_{n}^{\overline S_{m}'\cup\overline S_{m'}'}(\gs)$. Using~\eref{eps-EM} and the fact that $S_{m}\cup S_{m'}\subset S_{m}'\cup S_{m'}'=S_{m\vee m'}'$, we derive that, for all $m,m'\in\cM$, 
\[
\eps_{n}^{\overline S_{m}\cup\overline S_{m'}}(\gs)\le \eps_{n}^{\overline S_{m}'\cup\overline S_{m'}'}(\gs)\le \eps_{n}^{\overline S_{m\vee m'}'}(\gs)\le \overline \eps_{m\vee m'}=\overline \eps_{m}\vee\overline \eps_{m'}\le \sqrt{\overline \eps_{m}^{2}+\overline \eps_{m'}^{2}},
\]
hence~\eref{eq-eps} is satisfied, ~\eref{eq-pen} as well by the definition of $\pen$ and consequently all the conditions of Assumption~\ref{H-mod} are satisfied. 

Let us now turn to the quantity $\overline \eta(t)$ with $t\in\overline S$.
It  follows from (\ref{eq-Bay1}) that, for all $\gtheta\in \bs{\Theta}_{m}$,
\begin{align*}
\pi_{m}\pa{\sB^{\overline{S}_{m}}(t_{\gtheta},r)}&\ge\nu^{\otimes (m+1)}\pa{\gtheta'\in \bs{\Theta}_{m}'\,\left|\,\ab{\gtheta-\gtheta'}_{\infty}^{2}\le {8r^{2}\over (m+1)\sigma_{m}^{2}},\ \theta_{m}'\neq 0\right.}\\
&=\prod_{j=0}^{m} \nu\pa{\cro{\theta_{j}-{2\sqrt{2}r\over \sigma_{m}\sqrt{m+1} },\theta_{j}+{2\sqrt{2}r\over \sigma_{m}\sqrt{m+1}}}}\\
&= \cro{{2\sqrt{2}r\over \sigma_{m}M\sqrt{m+1}}\wedge 1}^{m+1}=\cro{{8r^{2}\over \sigma_{m}^{2}M^{2}(m+1)}\wedge 1}^{(m+1)/2}.
\end{align*}
%
We therefore deduce from~\eref{eq-eta} that, for all $t\in\overline S_{m}$,
\begin{align*}
\overline \eta_{n}^{2}(t)&=\inf_{r>0}\left(c_{7}r^{2}+{1\over 2n\beta}\log {1\over \gpi_{m}(\sB(t,r))}\right)\\
&\le \inf_{r>0}\left[c_{7}r^{2}+{m+1\over 4n\beta}\log\pa{1\bigvee {(m+1)\sigma_{m}^{2}M^{2}\over 8r^{2}}}\right].
\end{align*}
Setting $r^{2}=(m+1)/(4n\beta c_{7})$ we finally get
\[
\eta_{n}^{2}(t)\le {m+1\over 4n\beta}\cro{1+\log\pa{1\vee \pa{2c_{7}\beta n\sigma_{m}^{2}M^{2}}}}.
\]
The result follows by applying Theorem~\ref{thm-main2}. 

\section*{Acknowledgements}
The first author has received funding from the European Union's Horizon 2020 research and innovation programme under grant agreement N\textsuperscript{o} 811017.

The second author was supported by the grant ANR-17-CE40-0001-01 of the French National Research Agency ANR (project BASICS) and by Laboratoire J.A. Dieudonn\'e (Nice).

\bibliographystyle{apalike}

\end{document}